%% file: paper_screw_motion_tubes.tex
\documentclass[12pt,a4paper,twoside,openany]{amsart}
\usepackage{a4,latexsym,amscd,etex}

\usepackage[main=english, ngerman]{babel}
\usepackage{bm} 
\usepackage{amsaddr} 

\usepackage{setspace}
\usepackage{graphicx}
\usepackage{caption}
\usepackage{subcaption}
\usepackage{float}		
\usepackage{tikz}		
\usepackage{tikz-cd}	
\usepackage{tikz-3dplot}	

\usepackage{xcolor} 
\colorlet{shadecolor}{gray!22} 

\definecolor{myred}{HTML}{cc0000}
\definecolor{mylightgreen}{HTML}{85ad00}
\definecolor{mydarkgreen}{HTML}{3a8600}

\usepackage[
labelfont={sf},		
font={small},		
]{caption}

\newcommand{\executeiffilenewer}[3]{\ifnum\pdfstrcmp{\pdffilemoddate{#1}}{\pdffilemoddate{#2}}>0{\immediate\write18{#3}}\fi} 

\usepackage{multirow}	
\usepackage{tabularx}	
\usepackage{ragged2e}
\newcolumntype{L}[1]{>{\raggedright\arraybackslash}p{#1}}
\newcolumntype{C}[1]{>{\centering\arraybackslash}p{#1}}
\newcolumntype{R}[1]{>{\raggedleft\arraybackslash}p{#1}}
\newcolumntype{J}[1]{>{\justifying\arraybackslash}p{#1}}

\usepackage{rotating}

\usepackage{amsmath}
\usepackage{amstext}
\usepackage{amsthm}
\usepackage{amssymb}
\usepackage{mathtools}
\usepackage{stackengine}
\usepackage{scalerel}
\usepackage{pst-all}

\usepackage{enumerate}	
\usepackage{autobreak}	
\usepackage{verbatim}	
\usepackage{hyperref}	
\usepackage{extarrows}	
\usepackage{framed}		

\newlength{\leftbarwidth} 			
\newlength{\leftbarsep}
\colorlet{leftbarcolor}{black}

\usepackage[
	backend=biber,
	style=alphabetic, 
	sorting=nyt, 
	abbreviate=true, 
	firstinits=true, 
	maxbibnames=3, 
	isbn=false,	
	url=true, 
	doi=true, 
	hyperref=true, 
	backref=false, 
]{biblatex}

\DeclareFieldFormat[article]{title}{{#1}} 
\DeclareFieldFormat[thesis]{title}{{\mkbibemph{#1}}}

\DeclareSourcemap{
	\maps[datatype=bibtex]{
		\map{
			\step[fieldsource=issue, match=\regexp{\A[0-9]+\Z}, final]
			\step[fieldset=number, origfieldval]
			\step[fieldset=issue, null]
		}
	}
}

\DeclareBibliographyDriver{article}{%
	\printnames{author}%
	\addcomma\space
	\printfield{title}%
	\addcomma\space
	\printfield{journaltitle}%
	\space
	\textbf{\printfield{volume}}%
	\space
	\iffieldundef{number}
		{}
		{(\printfield{number})\space}%
	(\printfield{year})%
	\iffieldundef{pages}
		{\adddot}
		{\addcomma\space\printfield{pages}\adddot}%
	\space
	\printfield{doi}%
}

\DeclareBibliographyDriver{book}{%
	\printnames{author}%
	\addcomma\space
	\printfield{title}%
	\addcomma\space
	\iffieldundef{edition}
		{}
		{\printfield{edition}\addcomma\space}%
	\printlist{publisher}%
	\addcomma\space
	\printlist{location}%
	\space
	(\printfield{year})
	\adddot\space
	\printfield{doi}%
}

\DeclareBibliographyDriver{partofbook}{%
	\printnames{author}%
	\addcomma\space
	\printfield{title}%
	\addcomma\space
	\printfield{booktitle}%
	\addcomma\space
	\iffieldundef{volume}
	{}
	{\printfield{volume}\addcomma\space}%
	\iffieldundef{edition}
	{}
	{\printfield{edition}\addcomma\space}%
	\printlist{publisher}%
	\addcomma\space
	\printlist{location}%
	\space
	(\printfield{year})
	\iffieldundef{pages}
	{\adddot}
	{\addcomma\space\printfield{pages}\adddot}%
	\space
	\printfield{doi}%
}

\DeclareBibliographyDriver{thesis}{%
	\printnames{author}%
	\addcomma\space
	\printfield{title}%
	\iffieldundef{year}
	{}
	{\space(\printfield{year})}%
	\addcomma\space
	\printfield{note}%
	\adddot\space
	\printfield{url}%
}

\DeclareBibliographyDriver{online}{%
	\printnames{author}%
	\addcomma\space
	\printfield{title}%
	\iffieldundef{year}
	{}
	{\space(\printfield{year})}%
	\addcomma\space
	\printfield{note}%
	\adddot\space
	\printfield{eprint}%
}

\theoremstyle{plain}
\newtheorem{theo}{Theorem}[section]
\newtheorem*{theo*}{Theorem}
\newtheorem{prop}[theo]{Proposition}
\newtheorem{lem}[theo]{Lemma}
\newtheorem{cor}[theo]{Corollary}

\newtheorem*{conj*}{Conjecture}

\theoremstyle{definition}
\newtheorem{defi}[theo]{Definition}
\newtheorem{rem}[theo]{Remark}
\newtheorem*{rem*}{Remark}

\usepackage{enumitem}

\newcommand{\R}{\mathbb{R}}

\newcommand{\Z}{\mathbb{Z}}
\newcommand{\N}{\mathbb{N}}
\newcommand{\s}{\mathbb{S}}
\newcommand{\M}{\mathbb{M}}
\newcommand{\E}{\mathbb{E}}
\renewcommand{\H}{\mathbb{H}}

\renewcommand{\L}{\mathcal{L}}

\newcommand{\Berger}{\s^3_\mathrm{Berg}}
\newcommand{\EKT}{\E(\kappa,\tau)}

\newcommand{\HR}{\H^2\times\R}

\newcommand{\Nildrei}{\Nil_3}

\newcommand{\SR}{\s^2\times\R}
\newcommand{\SKR}{\s^2(\kappa)\times\R}
\newcommand{\MK}{\M^2(\kappa)}

\newcommand{\MKT}{\M(\kappa,\tau)}
\newcommand{\PSLzwei}{\widetilde{\PSL}_2(\R)}

\newcommand{\abs}[1]{\left|#1\right|}
\newcommand{\norm}[1]{\left\lVert#1\right\rVert}

\DeclareMathOperator{\area}{area}

\DeclareMathOperator{\arcoth}{arcoth}
\DeclareMathOperator{\artanh}{artanh}

\DeclareMathOperator{\Isom}{Isom}

\DeclareMathOperator{\Nil}{Nil}
\DeclareMathOperator{\PSL}{PSL}

\DeclareMathOperator{\sgn}{sgn}

\DeclareMathOperator{\vol}{vol}

\DeclareMathOperator{\sn}{sn_\kappa}
\DeclareMathOperator{\sns}{sn^2_\kappa}
\DeclareMathOperator{\snq}{sn^4_\kappa}
\DeclareMathOperator{\cs}{cs_\kappa}

\DeclareMathOperator{\ct}{ct_\kappa}

\DeclareMathOperator{\arcs}{arcs_\kappa}
\DeclareMathOperator{\arct}{arct_\kappa}

\DeclareMathOperator{\tn}{tn_\kappa}

\newcommand{\modspace}{\mathsf{Mod}_a(\kappa,\tau)}
\newcommand{\hmax}{h_a^\mathrm{max}}
\newcommand{\limH}{\lim\limits_{H\to\infty}}
\newcommand{\Jone}{\mathcal{J}_a^-}
\newcommand{\Jtwo}{\mathcal{J}_a^+}
\newcommand{\Jlimit}{\mathcal{J}_0}

\newcommand{\Jmax}{J_\mathrm{max}}
\newcommand{\Jtube}{J_a^\mathrm{tube}}

\newcommand{\Hbound}{\mathcal{H}_a^\mathrm{Nil}}
\newcommand{\Hex}{\mathcal{E}_a}
\newcommand{\Hcrit}{H_\mathrm{crit}}

\newcommand{\Hlimit}{\mathcal{H}_0}
\newcommand{\Htube}{H_a^\mathrm{tube}}

\newcommand{\Tfamily}{\mathcal{T}_{a}}

\renewcommand{\mod}{\operatorname{mod}}

\renewcommand{\epsilon}{\varepsilon}
\renewcommand{\theta}{\vartheta}
\renewcommand{\phi}{\varphi}

\newcommand{\termb}{\dfrac{\sns(r)+\left(4\tau\sns(\frac{r}{2})-a\right)^2}{\sns(r)}} 

\title[Invariant constant mean curvature tubes in homogeneous spaces]{Invariant constant mean curvature tubes in homogeneous spaces}
\author{Philipp Käse\textsuperscript{1}, Francisco Torralbo\textsuperscript{2}}
\address{\textsuperscript{1}Technische Universität Darmstadt, Fachbereich Mathematik, 64289 Darmstadt, Germany \\
\textsuperscript{2}Universidad de Granada, Departamento de Geometría y Topología, 18071 Granada, Spain}
\email{\textsuperscript{1}kaese@mathematik.tu-darmstadt.de, \textsuperscript{2}ftorralbo@ugr.es}

\begin{document}

\begin{abstract}
	We study the global geometry of families of tubes of constant mean curvature invariant under screw-motions in homogeneous $\EKT$-spaces. In particular, we study embeddedness and prove a foliation result. Moreover, we study numerically the isoperimetric profile in the compact case.
\end{abstract}

\maketitle

\section{Introduction}
\label{sec:introduction}

The study of surfaces of constant mean curvature (CMC in the sequel) in homogeneous \mbox{$3$-manifolds} $\EKT$ has attracted the attention of many geometers. Such surfaces have been studied extensively in the last two decades. The construction of examples have played an important role in the search for new features of CMC surfaces not present in the classical theory of CMC surfaces in space forms. The simple case of CMC surfaces invariant by a $1$-parameter group of isometries is of particular relevance because it serves as a first step to understand the non-invariant examples, which have usually a more sophisticated behavior.

This paper focuses on \emph{CMC tubes}. These are constant mean curvature cylinders (or tori in the compact case) invariant by a $1$-parameter group of isometries consisting of \emph{translations along a non-vertical geodesic}. If the geodesic is vertical with respect to the Riemannian submersion $\pi:\EKT\to \MK$, we find the \emph{Hopf cylinders}. These CMC surfaces are given as the preimage of a closed curve of constant curvature in $\MK$ under the projection map $\pi$. Hopf cylinders are also invariant by rotations around the vertical geodesic. We will not discuss these well-known surfaces in the present paper, but rather concentrate on non-vertical tubes. In the space form case $\kappa=4\tau^2$ all tubes are Riemannian tubes, i.e., equidistant surfaces to a given geodesic. This is no longer the case in $\EKT$ with $\kappa\neq4\tau^2$, although these surfaces lie at a bounded distance of that geodesic.

The existence of (non-vertical) CMC tubes was either only set in some particular cases or was completely missing in the classification results about invariant surfaces. The first example of a non-vertical CMC tube was explicitly given by Figueroa, Mercuri, and Pedrosa~\cite[Thm.~6]{figueroa_mercuri_pedrosa} in Heisenberg space $\Nildrei=\E(0,\tau)$ as a CMC cylinder invariant by translation along a \emph{horizontal} geodesic. In $\SR=\E(1,0)$ CMC tori were described by Pedrosa and Ritoré~\cite{pedrosa_ritore, pedrosa}. Although the latter examples were originally constructed as rotationally invariant CMC surfaces, they can also be thought as CMC tubes since rotations around a vertical geodesic in $\SR$ coincides with translations along the corresponding equatorial horizontal geodesic. Later, Montaldo and Onnis \cite{montaldo_onnis} described CMC tubes in $\HR=\E(-1,0)$. Vr\v{z}ina constructed a family of CMC tubes in both $\HR$~\cite{vrzina17} and $\SR$~\cite{vrzina16} using the Daniel correspondence, but embeddedness was missing in the second case. Moreover, there are examples of tubes in $\PSLzwei$, but only for horizontal geodesics~\cite{vrzina18}. On the other hand, tubes were missing in the classification result for invariant surfaces both in Berger spheres~\cite{torralbo10} and $\PSLzwei$~\cite{penafiel12} as shown in~\cite{kaese}.

All these previous results take advantage of the order 4 dihedral symmetry of the examples to construct them. López~\cite{lopez} was the first to spot numerically a change in the behavior in the nodoid family in $\mathrm{Sol}_3$ that gives rise to CMC tubes that only admit an order 2 dihedral symmetry. Very recently and independently Manzano~\cite{manzano24} and Käse~\cite{kaese} gave a general existence result, the former in the case of horizontal geodesics in all homogeneous spaces, and the latter for translations along geodesics that can be described via rotational screw motion groups. The latter approach allows us to deal in a unified way with the wide range of \textit{screw motion geodesics} (geodesic invariant under screw motions with pitch~$a$). For instance, this covers all geodesics if $\kappa>0$. This is the approach followed in this paper. However, it does not cover CMC tubes in $\HR$, because no geodesic in $\HR$ is a screw-motion geodesic (see e.g. \cite{onnis08,saearp,vrzina18}).


The main goal of this work is to study the global geometry of families of \textit{screw motion CMC tubes} in $\EKT$, especially the embeddedness and foliation properties. The main difficulty of the paper is the fact that the existence result given in \cite{kaese} is indirect and gives little geometric information about the surface. We restrict to supercritical curvature, that is $4H^2+\kappa>0$, because otherwise the profile curve of a screw motion CMC surface must be unbounded and no tubes can be expected (see \cite{onnis08,penafiel15} and Section~\ref{sec:screwmotion_surfaces}).

We first study the structure of the moduli space of screw motion CMC surfaces in $\EKT$ in Section~\ref{sec:moduli_space}. This provides further insights in the existence of screw motion CMC tubes. We study the limit of the tube family in terms of the mean curvature and show that the family converges to a geodesic. This gives a geometric explanation of the allowed values for the pitch found in the existence result. We further prove that the profile curves of some tubes do not admit a dihedral symmetry and we present a uniqueness result in Heisenberg space. In Section~\ref{sec:embedding} we study embeddedness of tubes and their compactness in the Berger sphere case. In Section~\ref{sec:foliation} we present an analysis of the foliation of ambient space by the tube family. Finally, Section~\ref{sec:isoperimetric_profile} studies numerically the isoperimetric profile of compact tubes in Berger spheres showing that they do not have better isoperimetric profile than the Hopf tori.

We can sum up our contribution in the following result:
\begin{theo*}
    Let $c_a$ be a geodesic invariant under a screw motion group with pitch $a$ in $\EKT$. Then there exists a constant $\Hex$, such that for any $H>\Hex$ there is a CMC tube invariant by translations along~$c_a$. Moreover:\vspace{-2mm}
    \setlist[1]{itemsep=5pt}
    \begin{itemize}[leftmargin=10mm]
    	\item There exists a continuous family of tubes $\Tfamily=\lbrace T_a(\mu)\colon\mu\in(0,\infty)\rbrace$ (Sec.~\ref{sec:construction_tubefamily}).
    	\item $T_a(\mu)$ converges to the geodesic $c_a$ as $\mu\to\infty$ (Lem.~\ref{lem:limit_infinity}).
    	\item For $\kappa\leq0$, the tube $T_a(\mu)$ is embedded for all $\mu$ in a neighborhood of $0$ and $+\infty$ (Thm.~\ref{theo:embedding_noncompact}).
    	\item For $\kappa>0$, a subfamily of $\Tfamily$ always foliates an open set of ambient space. In some cases the foliation is global (Thm.~\ref{theo:foliation_skr}, Cor.~\ref{cor:foliation_skr}, and Cor.~\ref{cor:foliation_berger}).
    	\item The profile curves of $T_a(\mu)$ in $\SKR$ and in $\Berger$ with $c_a$ horizontal have dihedral symmetry of order $4$, while for all other cases the profile curves only have a dihedral symmetry of order $2$ (Prop.~\ref{prop:dihedral_symmetry}).
    	\item For $\Nildrei$ there exists for any choice of $c_a$ a value $\Hbound$, such that there exists (up to isometry) exactly one tube along $c_a$ of constant mean curvature $H$ for each $H>\Hbound$ (Thm.~\ref{theo:uniqeness_nil}).
    \end{itemize}
\end{theo*}

The existence bound $\Hex$ (from \cite[Thm.~3.6]{kaese}) for a given screw motion geodesic $c_a$ is not necessarily sharp, i.e., there exist tubes along $c_a$ with mean curvature $H<\Hex$ in some cases (see Corollary~\ref{cor:nonsharp_existence}). Given a geodesic $c_a$ and mean curvature $H$ we also cannot guarantee the uniqueness of tubes except for the mentioned case of $\Nildrei$ and the previously know cases of $\SR$ and $\Berger$ with horizontal geodesic $c_a$. Nevertheless, there is numerical evidence that the sharp existence bound only depends on the choice of $c_a$ and tubes are always unique, in support of the following conjecture:

\begin{conj*}\label{conj:tube_existence}
	Given a screw motion geodesic $c$ there exists a unique constant $\Hlimit\geq0$ such that the following holds: There exists a tube of constant mean curvature $H$ invariant by translations along $c$ if and only if $H>\Hlimit$. If such a tube exists, it is unique (up to isometry).
\end{conj*}

Lemma~\ref{lem:intersectionlimit} ensures that the constant $\Hlimit$ must satisfy equation~\eqref{eq:intersectionlimit}.

We conclude this overview of our results with some remarks on their significance. The existence of embedded screw motion CMC tubes in $\EKT$ is of particular interest. For $\R^3$, Meeks proved that every properly embedded (non-minimal) CMC annulus stays at a bounded distance of a geodesic~\cite{meeks}. Korevaar, Kusner, and Solomon proved that every properly embedded CMC annulus at a bounded distance of a geodesic is rotationally invariant, and thus a Delaunay surface~\cite{korevaar_kusner_solomon}. This result was later adapted to $\H^3$ by Korevaar, Kusner, Meeks, and Solomon~\cite{korevaar_kusner_meeks_solomon} for CMC annuli of supercritical mean curvature. Mazet proved the same result for vertical geodesics in $\HR$: Every properly embedded CMC annuli at a bounded distance of a vertical geodesic is rotational invariant, and thus a Delaunay type surface~\cite{mazet15}.

On the contrary, the existence of embedded screw motion tubes in $\SR$, $\Nildrei$, $\PSLzwei$, and $\Berger$ (see Theorem~\ref{theo:embedding_noncompact} and Theorem~\ref{theo:embedding_compact}) shows that the above results cannot be generalized to any of these spaces, not even for just the vertical case as in Mazet's result. All screw motion tubes stay at a bounded distance of a vertical geodesic, but in general they are not rotationally invariant. The screw motion tubes therefore serve as counterexamples.

The Berger sphere case is of particular interest due to its compactness. Vertical and horizontal geodesics are great circles. However, all other geodesics will only close if a certain rationality condition is met. For the closed case the corresponding tubes are CMC tori that are non-congruent and only rotational invariant in the vertical case. We show that they are embedded for a certain range of the mean curvature (see Section~\ref{sec:embedding}). These examples indicate that the moduli space of embedded CMC tori in the Berger spheres is more involved than in the round sphere case, where Brendle~\cite{brendle} proves the Lawson conjecture (\textit{``The only embedded minimal torus is the Clifford torus''}) and Andrews and Li~\cite{andrews_li} extended the result to constant mean curvature $H>0$ showing that there are embedded tori different from Hopf tori,  but all of them are rotationally invariant. This leads to the question of whether all CMC tori in Berger spheres are invariant by a $1$-parameter group of isometries (not necessarily a group of rotations), or if these new examples present bifurcation properties like in the product spaces~\cite{manzano_torralbo}.

\subsubsection*{Acknowledgments} The first author has been supported by a research fellowship of the DAAD German Academic Exchange Service. The second author is supported by project PID2022-142559NB-I00 funded by MCIN/AEI/\-10.13039/\-501100011033. This paper is part of the first authors PhD project.

\section{Preliminaries}
\label{sec:preliminaries}

In this preliminary section we want to provide the necessary knowledge about the family of homogeneous manifolds $\EKT$, our model for these spaces, and the description of screw motion groups in $\EKT$. We recall some relevant topics about the description and classification of screw motion CMC surfaces in $\EKT$ as in~\cite{kaese}. We also introduce definitions and notation that are relevant throughout the paper.

\subsection{The family $\EKT$ and screw motion groups}\

Homogeneous Riemannian 3-manifolds with isometry group of dimension 4 comprise a $2$-parameter family $\EKT$ of Riemannian submersions $$\EKT \to \MK$$ to a simply connected surface of constant curvature $\kappa$ and bundle curvature $\tau$ with $\kappa\neq4\tau^2$. All fibers are geodesic. We refer to vectors parallel to fibers as \textit{vertical}, and vectors orthogonal to fibers as \textit{horizontal}.

The family contains the product spaces $\MK \times \mathbb{R}=\mathbb{E}(\kappa,0)$ as well as the Lie groups $\Nildrei = \mathbb{E}(0, \tau)$, $\PSLzwei$ ($\kappa < 0$), and $\mathrm{SU}(2)$ ($\kappa > 0$) with some left invariant metrics. The latter are known as Berger spheres~$\Berger$. The case $\kappa=4\tau^2$ corresponds to the space forms $\mathbb{R}^3$ ($\kappa = 0$) and $\mathbb{S}^3$ ($\kappa > 0$). We will not consider space forms and therefore assume $\kappa\neq4\tau^2$ throughout the paper.


In order to study screw motion surfaces in $\EKT \to \MK$ it is advantageous to use geodesic polar coordinates $(r, \theta, z)$ such that $r$ measures the arc-length of a radial geodesic in $\MK$. We use the model from \cite{kaese}:\vspace{3mm}
\begin{equation}\label{model}
	\left\lbrace\ \begin{aligned}
		&\MKT\coloneqq I_\kappa\times[0,2\pi)\times\R\ni(r,\theta,z), \\[3mm]
		&g\coloneqq dr^2
		+\sns\left(r\right)d\theta^2
		+\left(4\tau\sns\!\left(\frac{r}{2}\right)d\theta-dz\right)^2,
	\end{aligned}\right. 
\end{equation}

where $I_\kappa\coloneqq(0,\infty)$ if $\kappa\leq0$ or $I_\kappa\coloneqq(0,\frac{\pi}{\sqrt{\kappa}})$ if $\kappa>0$. Throughout the paper we use the generalized trigonometric functions\vspace{3mm}
\begin{equation*}
	\sn(r)\coloneqq\left\lbrace\begin{array}{ll}
		\frac{1}{\sqrt{\kappa}}\sin\left(\sqrt{\kappa}\,r\right) \\[0.5em]
		r \\[0.5em]
		\frac{1}{\sqrt{-\kappa}}\sinh\left(\sqrt{-\kappa}\,r\right)
	\end{array}\right.
	\hspace{8mm}
	\cs(r)\coloneqq\left\lbrace\begin{array}{ll}
		\cos\left(\sqrt{\kappa}\,r\right) & \text{ for } \kappa>0, \\[0.5em]
		1 & \text{ for } \kappa=0, \\[0.5em]
		\cosh\left(\sqrt{-\kappa}\,r\right) \hspace{8mm}\, & \text{ for } \kappa<0,
	\end{array}\right.
\end{equation*}

as well as $\tn(r)\coloneqq\frac{\sn(r)}{\cs(r)}$ and $\ct(r)\coloneqq\frac{\cs(r)}{\sn(r)}$, all depending continuously on the base curvature $\kappa$.

For $\kappa\leq0$ this is a model of $\EKT$ with the fiber at $r=0$ removed. For $\kappa>0$ it is a model for a covering of $\EKT$ without the fibers through the poles of the underlying $\s^2(\kappa)$ at $r=0$ and $r=\frac{\pi}{\sqrt{\kappa}}$. More specific, $\MKT$ is a covering over the third coordinate that corresponds to the fibers of $\EKT\to\MK$. But it is not the universal covering, since the polar angle $\theta$ is restricted to $[0,2\pi)$. We derive a model with compact fibers for $\Berger$ (still without the fibers at $r=0$ and $r=\frac{\pi}{\sqrt{\kappa}}$) when we consider $z\in \R \mod\frac{8\pi\tau}{\kappa}$ instead of $z\in\R$.

The set
\begin{equation*}
    G_a\coloneqq\lbrace L_{a,s}\colon (r,\theta,z)\mapsto(r,(\theta+s)\mod 2\pi,z+as)\,|\,s\in\R\rbrace
\end{equation*}

is the set of screw motions in $\MKT$ with pitch $a\in\R$ around the fiber over $r=0$. We refer to the fiber over $r=0$ as the \textit{axis} of the screw motion. The set $G_a$ is a closed one-parameter subgroup of $\Isom(\EKT)$ and therefore a Lie subgroup. It consists of rotations for $a=0$, which fix the axis pointwise. Note however, for $\tau\neq0$ the orbits of rotations are not horizontal with respect to the Riemannian submersion. Without loss of generality we can assume $\tau\geq0$ (by a change of orientation of $\MKT$). But with this choice we can no longer restrict the pitch to non-negative values, as the change from~$a$ to~$-a$ is also achieved by a change of orientation of ambient space. We therefore consider~$a\in\R$.

The action of $G_a$ on $\MKT$ is free and proper. We study further properties:

\begin{lem}\label{lem:group_properties}
    Suppose $\kappa\neq4\tau^2$ and define $\epsilon\coloneqq\sgn(\kappa-4\tau^2)$.
    
    $(\textit{i}\,)$ The screw motion group $G_a$ contains a geodesic orbit if and only if the product $a\tau\epsilon$ lies in the following open interval:
   	\begin{equation*}
   		a\tau\epsilon\in\left(-\infty,\frac{\epsilon}{2}\right)\text{ if }\kappa\leq0
   		\qquad\text{or}\qquad
   		a\tau\epsilon\in\left(\frac{4\tau^2\epsilon}{\kappa}-\dfrac{\epsilon}{2},\frac{\epsilon}{2}\right)\text{ if }\kappa>0.
   	\end{equation*}
   	Such a geodesic orbit can be parameterized as $c_a(s)=(\rho_a,s,as)$, where
   	\begin{equation*}
   		\rho_a=\arcs\left(\frac{\kappa}{\kappa-4\tau^2}\,(2a\tau-1)+1\right).
   	\end{equation*}
   	As $a\tau\epsilon\to\frac{\epsilon}{2}$ or $a\tau\epsilon\to\frac{4\tau^2\epsilon}{\kappa}-\frac{\epsilon}{2}$ we obtain the vertical geodesics with $\rho=0$ or $\rho=\frac{\pi}{\sqrt{\kappa}}$, respectively, which have been removed in our model $\MKT$.
   	
    $(\textit{ii}\,)$ For $\kappa>0$ the screw motion group $G_a$ is conjugate to the screw motion group $G_{\tilde a}$, setting $\tilde a\coloneqq\frac{4\tau}{\kappa}-a$. Moreover, the geodesic orbit~$c_a$ is congruent to the geodesic orbit $c_{\tilde a}$.
\end{lem}

\begin{proof}
    $(\textit{i}\,)$ The orbit $G_a(p)$ of a point $p=(r,\theta,z)\in\MKT$ can be parameterized as $c_a(s)=(r,\theta+s,z+as)$. We may assume $\theta=z=0$ because the space is homogeneous. 
    Then a straightforward computation verifies the curve $c_a$ is geodesic if and only if $r=\rho_a$, and the pitch $a$ is contained in the intervals claimed.
    
    $(\textit{ii}\,)$ For $\kappa>0$ consider the isometry
    \begin{equation}\label{eq:isometry_phi}
    	\Phi_a\!:\quad \MKT\to\MKT,
    	\qquad
    	\left(r,\theta,z\right)\mapsto\left(\frac{\pi}{\sqrt{\kappa}}-r,\theta,z+\left(\frac{4\tau}{\kappa}-2a\right)\theta\right).
    \end{equation}
    Then $L_{a,s}=\Phi_a^{-1} \circ L_{\tilde a,s} \circ \Phi_a$, that means $G_a$ is conjugate to $G_{\tilde a}$. Moreover, $\Phi_a$ maps $c_a$ to $c_{\tilde a}$, i.e., the curves are congruent.
\end{proof}

The fact that $G_a$ and $G_{\tilde a}$ with $\tilde a=\frac{4\tau}{\kappa}-a$ are conjugate for $\kappa>0$ reflects the symmetry of the ambient space $\SKR$ or $\Berger$. A screw motion around the fiber over the north pole $r=0$ agrees with the screw motion around the fiber over the south pole $r=\frac{\pi}{\sqrt{\kappa}}$, provided we change the pitch from~$a$ to~$\tilde{a}$. This transformation is represented by a reflection about the equatorial cylinder $\lbrace\frac{\pi}{2\sqrt{\kappa}}\rbrace\times\left[0,2\pi\right]\times\R$, see definition of $\Phi_a$ in \eqref{eq:isometry_phi}. This relation motivates the following definition:

\begin{defi}
	If $\kappa>0$, we call $\tilde{a}\coloneqq\frac{4\tau}{\kappa}-a$ the \textit{conjugate pitch} of $a$.
\end{defi}

In case $G_a$ contains a geodesic orbit, we can interpret the action of $G_a$ as left-translations along the geodesic $c_a$, that is $L_{a,s}(c_a(0))=c_a(s)$. Since $G_a$ and $G_{\tilde a}$ are conjugate for $\kappa>0$, we may focus on $a\tau\epsilon\geq\frac{2\tau^2\epsilon}{\kappa}$ for this case. As we will point out in Section\ref{sec:screwmotion_surfaces}, screw motion tubes exist exactly when $G_a$ contains a geodesic orbit \cite{kaese}. Therefore, we assign a name to these special values of the pitch:

\begin{defi}\label{def:admissible}
    We call the pitch $a$ of the screw motion group $G_a$ \textit{admissible} if
    \begin{equation*}
        a\tau\epsilon\in\left(-\infty,\frac{\epsilon}{2}\right)\text{ if }\kappa\leq0
        \qquad\text{or}\qquad
        a\tau\epsilon\in\left[\frac{2\tau^2\epsilon}{\kappa},\frac{\epsilon}{2}\right)\text{ if }\kappa>0.
    \end{equation*}
\end{defi}

We discuss some further properties of the geodesic orbit $c_a$: 

\begin{lem}\label{lem:angle}
    Suppose $G_a$ contains a geodesic orbit $c_a$. Let $\nu\in[0,\pi]$ be the angle of the geodesic~$c_a$ with respect to the fibers.	
    
    $(\textit{i}\,)$ The angle $\nu$ is constant and given by
    \begin{equation}\label{eq:angle}
    	\cos(\nu)=\frac{g\!\left(c'_a,\frac{\partial}{\partial z}\circ c_a\right)}{\norm{c'_a}_g\norm{\frac{\partial}{\partial z}}_g}
    	=\frac{(a\kappa-2\tau)\epsilon}{\sqrt{(\kappa-4\tau^2)(1+a^2\kappa-4a\tau)}}.
    \end{equation}
    
    $(\textit{ii}\,)$ For $\kappa>0$ the geodesic becomes vertical as $a\to\frac{1}{2\tau}$ or $a\to\frac{4\tau}{\kappa}-\frac{1}{2\tau}$. It is horizontal \hphantom{lllll}~for~$a=\frac{2\tau}{\kappa}$.
    
    $(\textit{iii}\,)$ For $\kappa\leq0$ the geodesic becomes vertical as $a\to\frac{1}{2\tau}$. It is never horizontal as
    \begin{equation*}
    	\lim\limits_{a\to\infty}\cos(\nu)=\sqrt{\frac{\kappa}{\kappa-4\tau^2}}.
    \end{equation*}
\end{lem}

\begin{proof}
    By a straightforward computation we get
    \begin{equation*}
    	\begin{aligned}
    		&g\left(c'_a,\tfrac{\partial}{\partial z}\right)
    		=a-4\tau\sns\left(\tfrac{\lambda}{2}\right)
    		=\frac{a\kappa-2\tau}{\kappa-4\tau^2},
    		\hspace{12mm}
    		g\left(\tfrac{\partial}{\partial z},\tfrac{\partial}{\partial z}\right)
    		=1, \\
    		&g\left(c'_a,c'_a\right)
    		=\sns(\lambda)+16\tau^2\snq\left(\tfrac{\lambda}{2}\right)-8a\tau\sns\left(\tfrac{\lambda}{2}\right)+a^2
    		=\frac{1+a^2\kappa-4a\tau}{\kappa-4\tau^2},
    	\end{aligned}
    \end{equation*}
    which yields \eqref{eq:angle}. Curves are vertical if $\cos\nu=\pm1$ and horizontal if $\cos\nu=0$.
\end{proof}

\begin{rem}\label{rem:angle}
    For $\kappa>0$ the set of geodesic orbits $\left\lbrace c_a\right\rbrace$ covers all non-vertical geodesics of the ambient space, because $\cos(\nu)$ covers all values in $(-1,1)$. Geodesics with $\cos(\nu)\in[0,1)$ are congruent to geodesics with $\cos(\tilde\nu)=-\cos(\nu)\in(-1,0]$ via $a\leftrightarrow\tilde a$ by Lemma~\ref{lem:group_properties}. For $\kappa=0$ and $\tau\neq0$ this includes all geodesics except the vertical and horizontal case. The latter case was covered by Manzano \cite{manzano24}. For $\kappa<0$ and $\tau\neq0$ it includes all geodesics with angle $\cos(\nu)>\sqrt{\frac{\kappa}{\kappa-4\tau^2}}$. For $\kappa<0$ and $\tau=0$ it covers none of the geodesics. The study of invariant CMC screw motion surfaces in these excluded cases is completely different, see for example \cite{onnis08,penafiel12,saearp,vrzina18}.
\end{rem}

\subsection{CMC Screw Motion Surfaces in $\EKT$}\
\label{sec:screwmotion_surfaces}

We study $G_a$-invariant CMC surfaces. Therefore we summarize below the basic properties of the system of ordinary differential equations for the profile curves, see \cite{kaese} for details and proofs.

Let $\gamma(t)=(r(t),h(t))$ be a curve in the quotient space $\MKT/G_a$ and $\Sigma=G_a(\gamma)$ the generated screw motion surface. Let $\sigma(t)\in\R$ be the angle between the coordinate direction~$\frac{\partial}{\partial r}$ and the tangent vector~$\gamma'(t)$. Then $\gamma$ has unit-speed and $\Sigma$ has constant mean curvature~$H$ if $\gamma$ satisfies
\begin{equation}\label{ode}\tag{ODE}
    \left\lbrace\begin{array}{l}
        r' =\cos\sigma, \\[1.0em]
        h'=\sqrt{\termb}\,\sin\sigma, \\[1.6em]
        2H =\ct(r)\sin\sigma+\sigma'.
    \end{array}\right.
\end{equation}

Moreover, $\gamma$ is a solution of \eqref{ode} if and only if the \textit{energy}
\begin{equation}\label{eq:energy}
    J(r(t),\sigma(t)) \coloneqq \dfrac{2H}{\kappa}\Big(\!\cs(r(t))-1\Big) +\sn(r(t))\sin\sigma(t)
\end{equation}
is constant in $t$ \cite[Thm.~2.3 \& Prop.~2.5]{kaese}. The expression for the energy extends to $\kappa=0$ continuously by taking the limit $\kappa\to0$. For simplicity, we only address the case $\kappa\neq0$ in the following with the understanding that the case $\kappa=0$ is obtained as the limit.


From now on we assume $H>0$ (possibly after a change of orientation). We also assume that the mean curvature is \textit{supercritical}, i.e., $H>\Hcrit$, where\vspace{3mm}
\begin{equation}\label{eq:Hcritical}
	\Hcrit(\EKT)=\left\lbrace \begin{array}{cc}
		0 & \text{for }\kappa\geq0, \\[4mm]
		\tfrac{\sqrt{-\kappa}}{2} & \text{for }\kappa<0.
	\end{array}\right. 
\end{equation}

This conditions is natural for the existence of CMC tubes, because the profile curves of screw motion CMC surfaces with \textit{subcritical} ($H<\Hcrit$) or \textit{critical} ($H=\Hcrit$) mean curvature are unbounded, see \cite{onnis08,penafiel15}. Therefore no tubes can be expected. Indeed Manzano~\cite[Cor.~4.2]{manzano24} has proven that there are no CMC surfaces at a bounded distance of a \emph{horizontal} geodesic if $H\leq\Hcrit$.

The \eqref{ode} has unique solutions for given initial data provided the pitch $a$ and mean curvature $H$ are fixed. For a given energy value $J$, the solution is unique up to vertical translation. Therefore, we use the following notation:

\begin{defi}
	We denote the solution curve of \eqref{ode} with pitch~$a$, mean curvature~$H$, and energy~$J$ by $\gamma_a(H,J)=(r_a(H,J),h_a(H,J))$ and the generated screw motion CMC surface by $\Sigma_a(H,J)$.  When the curve parameter is needed, we will write $\gamma_a(t;H,J)$. As an abbreviation, we often just write $\gamma(t)$ and skip the other parameters $a,H,J$. We never explicitly denote the base curvature $\kappa$ and the bundle curvature $\tau$ as fixed parameters of ambient space.
\end{defi}

The solution curves of \eqref{ode} can be classified by values of the energy $J$, where $J$ is bounded from above by some function $\Jmax=\Jmax(H)$ under the assumption of supercritical mean curvature. This yields a classification of screw motion CMC surfaces as the first author has shown in \cite[Thm.~3.1]{kaese}: Screw motion invariant surfaces in $\EKT$ of non-minimal supercritical constant mean curvature form a one-parameter family and are of one of the six types stated in Figure \ref{fig:classification}. For $J=\Jmax$ the screw motion surface $\Sigma_a(H,J)$ is a vertical cylinder. For $J\in(0,\Jmax)$ it is of unduloid type. For $J\in(-\infty,0), \kappa\leq0$ or $J\in(-\frac{4H}{\kappa},0), \kappa>0$ the screw motion surface $\Sigma_a(H,J)$ is either of nodoid type or a tube. Sphere type surfaces arise as limits for $J=0$ if $\kappa\leq0$, or for $J=0$ and $J=-\frac{4H}{\kappa}$ if $\kappa>0$. Not necessarily all surfaces types arise for every choice of parameters.

The set of all points $(H,J)$ corresponding to screw motion CMC surfaces $\Sigma_a(H,J)$ for fixed pitch $a\in\R$ form the moduli space:

\begin{defi}
	For pitch $a\in\R$ we define the \textit{moduli space} of screw motion CMC surfaces in $\MKT$ as the set of points $\lbrace(H,J)\rbrace\equiv\lbrace\Sigma_a(H,J)\rbrace$. That is
	\begin{equation}\label{eq:modulispace}
		\modspace\coloneqq\left\lbrace (H,J)\in[0,\infty)\times\R \colon \text{ there exists } \Sigma_a(H,J)\subseteq\MKT \right\rbrace.
	\end{equation}
\end{defi}

The structure of the moduli space $\modspace$ depends on the explicit choices of the parameters $\kappa$ and $\tau$ of ambient space $\MKT$ as well as the chosen pitch $a$ of the screw motion. An example of the moduli space for supercritical mean curvature is displayed in Figure~\ref{fig:modulispace}.

\begin{figure}[]
	\centering
	\tiny
	\def\svgwidth{\textwidth}\executeiffilenewer{curve_family.svg}{curve_family.pdf}{inkscape -z -D --file=curve_family.svg 	--export-pdf=curve_family.pdf --export-latex}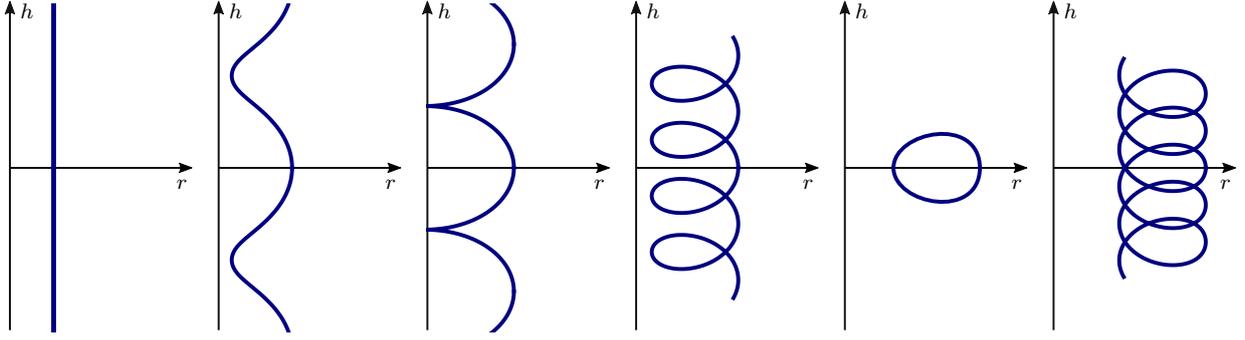
	\caption{Profile curves of the screw motion CMC surfaces in $\EKT$ according to the classification given in \cite[Thm.~3.1]{kaese}: vertical cylinder, unduloid type, sphere type, nodoid type I, tube, nodoid type II (left to right).}
	\label{fig:classification}
\end{figure}

The question of existence of screw motion CMC tubes was partially answered in \cite[Thm.~3.6 \& Cor.~3.8]{kaese}: \textit{If the pitch $a$ is admissible and the mean curvature $H$ is larger than the value
	\begin{equation}\label{eq:definition_Hexistence}
		\Hex=\sqrt{\frac{2\tau^2-a\tau\kappa}{4a\tau-2}},
	\end{equation}
then there exists an energy value $J$ such that the screw motion surface $\Sigma_a(H,J)$ is a tube. In reverse, the pitch must be admissible (or conjugate to an admissible value) if the screw motion surface is a tube. Moreover, the energy of those screw motion tubes must always lie in the interval $\left[\Jone(H),\Jtwo(H)\right]$, where
\begin{equation}\label{eq:definition_J1J2}
	\Jone(H)\coloneqq\frac{2H}{\kappa-4\tau^2}\,(2a\tau-1)
	\quad\text{and}\quad
	\Jtwo(H)\coloneqq\Jone(H)-\frac{2\tau^2-a\tau\kappa}{H(\kappa-4\tau^2)}.
\end{equation}
}

This motivates the restriction to admissible pitch throughout this paper. Moreover, given some value for the mean curvature $H$, we can always find an admissible pitch $a$ such that $H>\Hex$ and a tube with these parameters exists. The constant $\Hex$ is a technical artifact from applying the intermediate value theorem to prove existence. We show that $\Hex$ is only sharp if $2\tau^2-a\tau\kappa=0$. In the other case, there exist tubes with $H<\Hex$. We prove this statement in Corollary~\ref{cor:nonsharp_existence}. However, screw motion tubes do not exist for all $H>0$ in this case.

In $\SKR$ or $\Berger$ a screw motion with pitch $a$ around the fiber over the north pole $r=0$ is congruent to a screw motion with conjugate pitch $\tilde{a}=\frac{4\tau}{\kappa}-a$ around the fiber over the south pole $r=\frac{\pi}{\sqrt{\kappa}}$, because the groups $G_a$ and $G_{\tilde a}$ are conjugate by Lemma~\ref{lem:group_properties}. The invariant surfaces are congruent up to reflection at the equatorial cylinder $\lbrace\frac{\pi}{2\sqrt{\kappa}}\rbrace\times\left[0,2\pi\right]\times\R$ and suitable vertical translation. The energy changes from $J$ to $\tilde J=-J-\frac{4H}{\kappa}$ \cite[Lem.~2.7]{kaese}. This implies a symmetry in the moduli space for $\kappa>0$, provided we change the pitch to its conjugate value, see Figure~\ref{fig:modulispace}.

Since this paper focuses on existence and properties of tubes, we only consider the region of the moduli space that belongs to tubes and nodoid type surfaces. This motivates the following restriction to a specific subset of the moduli space:\vspace{3mm}
\begin{equation}\label{eq:definition_xia}
	\Xi\coloneqq\left\lbrace
	\begin{array}{ll}
		\big\lbrace \left(H,J\right): H>0, J\in(-\frac{4H}{\kappa},0) \big\rbrace & \text{ for } \kappa>0, \\[6mm]
		\big\lbrace \left(H,J\right): H>\frac{\sqrt{-\kappa}}{2}, J\in(-\infty,0) \big\rbrace & \text{ for } \kappa\leq0.
	\end{array}
	\right. 
\end{equation}

We divide the moduli subspace $\Xi$ into three subsets\vspace{3mm}
\begin{align}
	\Xi_a^+ &\coloneqq\left\lbrace \left(H,J\right)\in\Xi: J>\Jtwo(H)\right\rbrace, \label{eq:xi_subset1} \\[2mm]
	\Xi_a^0\, &\coloneqq\left\lbrace \left(H,J\right)\in\Xi: \Jone(H)\leq J\leq\Jtwo(H)\right\rbrace, \label{eq:xi_subset2} \\[2mm]
	\Xi_a^- &\coloneqq\left\lbrace \left(H,J\right)\in\Xi: J<\Jone(H)\right\rbrace, \label{eq:xi_subset3}
\end{align}

with $\Jone,\Jtwo$ from \eqref{eq:definition_J1J2}. Then
\begin{equation*}
	\Xi=\Xi_a^+\cup\Xi_a^0\cup\Xi_a^-
	\qquad\text{and}\qquad
	\Xi_a^+\cap\Xi_a^0\cap\Xi_a^-=\emptyset,
\end{equation*}
see Figure~\ref{fig:modulispace}. For $\kappa-4\tau^2>0$ the set $\Xi_a^+$ consists solely of nodoid type I surfaces and the set $\Xi_a^-$ consists solely of nodoid type II surfaces. For $\kappa-4\tau^2<0$ the situation is exactly the opposite. The set $\Xi_a^0$ can contain all three nodoid type I surfaces, nodoid type II surfaces, and tubes, but not necessarily all three types at once. Moreover, all tubes must be in $\Xi_a^0$ (see~\cite[Lem.~3.7]{kaese}).

\begin{figure}[]
	\centering
	\scriptsize
	\begin{subfigure}[t]{0.48\textwidth}
		\def\svgwidth{\textwidth}\executeiffilenewer{moduli_space_poscurv.svg}{moduli_space_poscurv.pdf}{inkscape -z -D --file=moduli_space_poscurv.svg 	--export-pdf=moduli_space_poscurv.pdf --export-latex}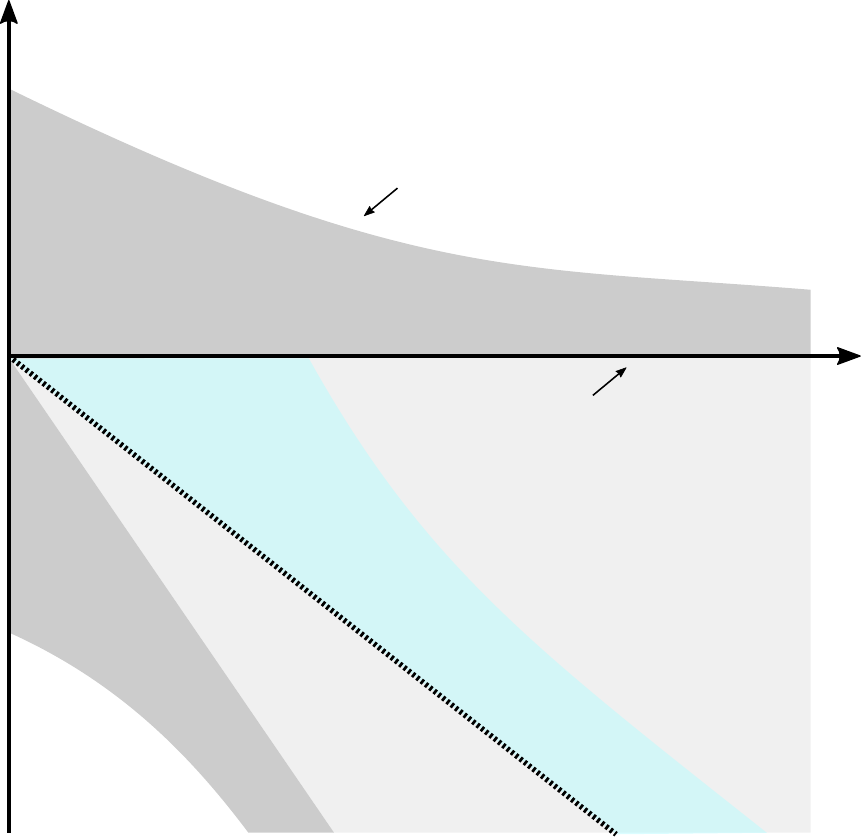
		\subcaption{Moduli space for $\kappa>0$.}
		\label{fig:modulispace_poscurv}
	\end{subfigure}
	\hfill
	\begin{subfigure}[t]{0.48\textwidth}
		\def\svgwidth{\textwidth}\executeiffilenewer{moduli_space_negcurv.svg}{moduli_space_negcurv.pdf}{inkscape -z -D --file=moduli_space_negcurv.svg 	--export-pdf=moduli_space_negcurv.pdf --export-latex}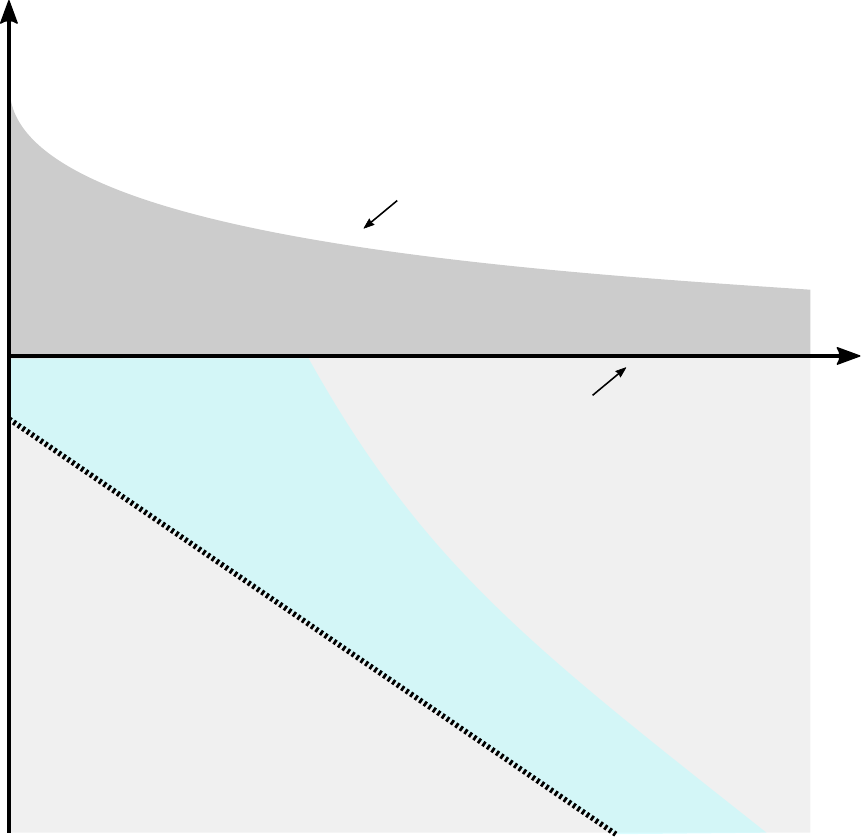
		\subcaption{Moduli space for $\kappa\leq0$.}
		\label{fig:modulispace_negcurv}
	\end{subfigure}
	\caption{Moduli space $\modspace$ for supercritical mean curvature $H>\Hcrit$, admissible pitch $a$ and $2\tau^2-a\tau\kappa\neq0$ (see Section~\ref{sec:moduli_space} for the last condition). The subsets $\Xi_a^+$, $\Xi_a^0$, and $\Xi_a^-$ are defined in \eqref{eq:xi_subset1}-\eqref{eq:xi_subset3}. The tube region $\Theta_a$ is a subset of $\Xi_a^0$ and must be in the blue shaded region of the moduli space.}
	\label{fig:modulispace}
\end{figure}

It can be shown that for every solution curve $\gamma(t)=\gamma_a(t;H,J)$ with $(H,J)\in\Xi$ the derivative of the angle function $\sigma'(t)$ is everywhere positive and bounded away from zero (see~\cite[Lem.~3.4]{kaese}). Thus, $\sigma$ attains all values in $\R$ and we can reparameterize the profile curve $\gamma(t)=(r(t),h(t))$ as $\gamma(\sigma)=(r(\sigma), h(\sigma))$ and rewrite \eqref{ode} as
\begin{equation}\label{ode_sigma}\tag{ODE$^\star$}
	\left\lbrace\begin{array}{l}
		\dfrac{dr}{d\sigma}=\dfrac{\cos\sigma}{2H-\ct(r)\sin\sigma}, \\[1.0em]
		\dfrac{dh}{d\sigma}=\dfrac{\sqrt{\sns(r)+\left(4\tau\sns(\frac{r}{2})-a\right)^2}}{2H\sn(r)-\cs(r)\sin\sigma}\,\sin\sigma.
	\end{array}\right.
\end{equation}

The radius function $r$ attains values exactly in the interval $[r_-,r_+]$, where
\begin{equation}\label{eq:radius_notation}
	r_\pm=r_\pm(H,J)=\pm\arct\!\left(2H\right)+\arcs\!\left(\frac{\kappa J+2H}{\sqrt{4H^2+\kappa}}\right).
\end{equation}

Without loss of generality we choose $r(\frac{\pi}{2};H,J)=r_+(H,J)$ and $h(\frac{\pi}{2};H,J)=0$ as initial conditions. From \eqref{eq:energy} we get an explicit formula for the radius
\begin{equation}\label{eq:radius_sigma}
	r(\sigma)=r_a(\sigma;H,J)=\arct\!\left(\frac{2H}{\sin\sigma}\right)+\arcs\!\left(\frac{\kappa J+2H}{\sqrt{4H^2+\kappa\sin^2\!\sigma}}\right)
\end{equation}
in terms of $\sigma$, with the maximal and minimal radius $r_\pm=r_\pm(H,J)$ for $\sigma=\frac{\pi}{2} (\mod 2\pi)$ and $\sigma=\frac{3\pi}{2} (\mod 2\pi)$, respectively. Note that it does not depend explicitly on the pitch~$a$. While we have an explicit formula for $r(\sigma)$, the height function $h(\sigma)$ is given by the integral
\begin{equation}\label{eq:height_sigma}
	h(\sigma)=h_a(\sigma;H,J)=\int\limits_{\frac{\pi}{2}}^{\sigma} \frac{dh}{d\sigma}\,d\sigma.
\end{equation}

It is easy to observe from \eqref{ode_sigma} that $\frac{dh}{d\sigma}(\sigma)>0$ for $\sigma\in[\frac{\pi}{2},\pi)$, $\frac{dh}{d\sigma}(\sigma)=0$ for $\sigma=\pi$, and $\frac{dh}{d\sigma}(\sigma)<0$ for $\sigma\in(\pi,\frac{3\pi}{2}]$. Therefore, $h_a(\sigma;H,J)$ reaches a local maximum at $\sigma=\pi$ and we define
\begin{equation}\label{eq:height_max}
	\hmax(H,J)\coloneqq h_a(\pi;H,J).
\end{equation}

The radius $r(\sigma)=r(\sigma;H,J)$ as well as the height derivative $\frac{dh}{d\sigma}(\sigma)=\frac{dh}{d\sigma}_a(\sigma;H,J)$ are $2\pi$-periodic functions in $\sigma$. Therefore, we can restrict our discussion to the arc for $\sigma\in[\frac{\pi}{2},\frac{5\pi}{2}]$. The profile curve of a nodoid type surface can be distinguished from the profile curve of a tube in so far as the arc $\sigma\in[\frac{\pi}{2},\frac{5\pi}{2}]$ is not closed. Because $r(\frac{5\pi}{2})=r(\frac{\pi}{2})=r_+$ this requires $h(\frac{5\pi}{2})\neq h(\frac{\pi}{2})$. Due to the symmetry under reflection at the line $h=h(\frac{\pi}{2})$ it suffices to compare $h(\frac{3\pi}{2})$ with $h(\frac{\pi}{2})$. We define a function $\delta_a$ which measures the difference $h(\frac{3\pi}{2})-h(\frac{\pi}{2})$:

\begin{defi}\label{def:closingdefect}
	For fixed $\kappa,\tau$ and $a$ admissible we define the \textit{closing defect}
	\begin{equation*}
		\delta_a\colon\Xi\to\R,~(H,J)\mapsto\int\limits_{\frac{\pi}{2}}^{\frac{3\pi}{2}} \frac{dh}{d\sigma}(\sigma;a,H,J)\,d\sigma.
	\end{equation*}
\end{defi}

The integrand is explicitly given as (see \cite{kaese})
\begin{equation}\label{eq:height_function}
	\frac{dh}{d\sigma}(\sigma;a,H,J)
	=\frac{\sqrt{f(\sigma)}}{(4H^2+\kappa\sin^2\!\sigma)\sqrt{\sin^2\!\sigma-\kappa J^2-4HJ}}\,\sin\sigma
\end{equation}
with
\begin{equation}\label{eq:height_function_f}
	f(\sigma;a,H,J)\coloneqq
	C_1\sin^4\!\sigma+C_2\sin^2\sigma+C_3
	+\left(C_4\sin^2\!\sigma+C_5\right)
	\sqrt{\sin^2\!\sigma-\kappa J^2-4HJ}\sin\sigma
\end{equation}
and $C_i=C_i(a,H,J)$ are the coefficient functions
\begin{equation*}
	\begin{aligned}
		& C_1=8\tau^2-4a\tau\kappa+a^2\kappa^2, \\
		& C_2=-32\tau^2HJ-4\tau^2\kappa J^2-16a\tau H^2+8a\tau\kappa HJ+4\kappa HJ+\kappa^2J^2+8H^2+8a^2\kappa H^2, \\
		& C_3=16\tau^2H^2J^2+32a\tau H^3J-4\kappa H^2J^2-16H^3J+16a^2H^4, \\
		& C_4=8\tau^2-4a\tau\kappa, \\
		& C_5=-16\tau^2HJ-16a\tau H^2+4\kappa HJ+8H^2.
	\end{aligned}
\end{equation*}

Then $\gamma_a(H,J)$ is a simple closed curve if and only if $\delta_a(H,J)=0$. The set
\begin{equation}
	\Theta_a\coloneqq\delta_a^{-1}(0)
	=\left\lbrace \left(H,J\right)\in\Xi_a: \Sigma_a(H,J) \text{ is a tube }\right\rbrace
	\subseteq\Xi_a^0
\end{equation}
is the subset of the moduli space containing exactly the screw motion tubes. We call it the \textit{tube region}. Nodoid type I surfaces are characterized by $\delta_a(H,J)>0$. Nodoid type II surfaces are characterized by $\delta_a(H,J)<0$, see Figure \ref{fig:curve_notation}.

The lack of an explicit formula for neither the closing defect $\delta_a$ nor the height $h$ makes the study of tubes a challenging problem. We can often provide only estimates, resulting in a non-sharp existence result and uniqueness only in specific cases.

\begin{figure}[]
	\centering
	\footnotesize
	\def\svgwidth{0.95\textwidth}\executeiffilenewer{curve_notation.svg}{curve_notation.pdf}{inkscape -z -D --file=curve_notation.svg 	--export-pdf=curve_notation.pdf --export-latex}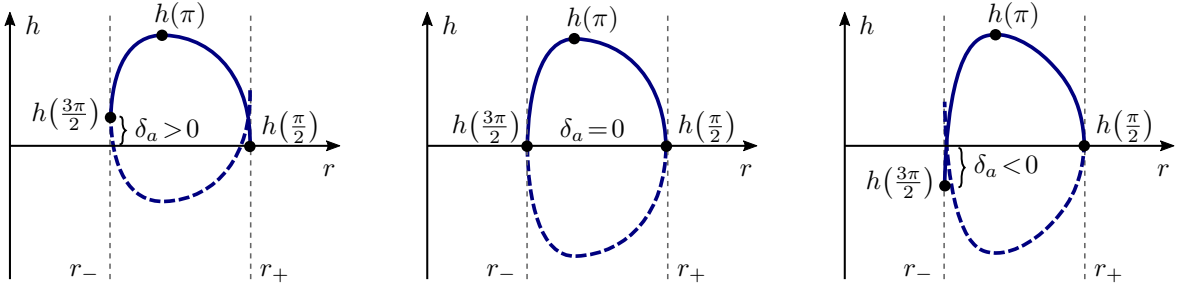
	\caption{Schematic profile curves with initial data $r(\frac{\pi}{2})=r_+, h(\frac{\pi}{2})=0$ for different energies (\textit{left to right}): nodoid type I ($\delta_a(H,J)>0$), tube ($\delta_a(H,J)=0$), and nodoid type II ($\delta_a(H,J)<0$), where $\delta_a$ is the closing defect, see Def.~\ref{def:closingdefect}.}
	\label{fig:curve_notation}
\end{figure}

To conclude this preliminary section we present the minimal screw motion surfaces with vanishing energy for $\kappa>0$, that is in the ambient spaces $\SKR$ and $\Berger$. These surfaces arise as limits of the tube family in Sections~\ref{sec:embedding} and~\ref{sec:foliation}.

\begin{lem}\label{lem:minimal_helicoid}
	Suppose $\kappa>0$. The minimal screw motion surface $\Sigma_{a,0,0}$ with vanishing energy is a spherical helicoid generated by a horizontal geodesic perpendicular to the screw motion axis.
\end{lem}

\begin{proof}
	For $\kappa>0$ and $H=0$ the energy condition $J=0$ implies $\sin\sigma=0$ and we obtain $r'=\pm1$ and $h'=0$ from~\eqref{ode}. Thus, the profile curve is a straight line perpendicular to the screw motion axis, which is a horizontal geodesic in the ambient space. It generates a spherical helicoid.
\end{proof}

Spherical helicoids have the topology of a cylinder or torus, because the generating curve is a closed geodesic. They are therefore tubes. Spherical helicoids were first introduced by Lawson in $\s^3$ \cite{lawson} and later Torralbo realized that they are also minimal in Berger spheres \cite{torralbo12}. Spherical helicoids are the only ruled minimal surfaces in $\Berger$~\cite{shin_kim_koh_lee_yang}. Moreover, they are the only surfaces that are simultaneously minimal in~$\Berger$ for any~$\kappa$ and~$\tau$~\cite[Prop.~1]{torralbo12}.

\section{Structure of the moduli space and tube region}
\label{sec:moduli_space}

We study the structure of the moduli subspace $\Xi\subseteq\modspace$ and the tube region $\Theta_a\subseteq\Xi$ for further geometric information about screw motion tubes.

In the special case $2\tau^2-a\tau\kappa=0$, that is $\SKR$ with arbitrary pitch $a\in\R$, or $\Berger$ with horizontal pitch $a=\frac{2\tau}{\kappa}$, we have $\Jone(H)=\Jtwo(H)=-\frac{2H}{\kappa}$, see \eqref{eq:definition_J1J2}. Then\vspace{3mm}
\begin{equation*}
	\Xi_a^0=\left\lbrace(H,-\tfrac{2H}{\kappa})\colon H>0\right\rbrace
\end{equation*}

is just a straight line in the moduli space. Moreover, \cite[Thm.~3.6]{kaese} provides the existence of tubes for all $H>0$, because $\Hex=0$. Therefore $\Theta_a=\Xi_a^0$, and we can define the \textit{tube family}
\begin{equation}\label{eq:def_tubefamily_hor}
	\Tfamily=\left\lbrace T_a(H)=\Sigma_a(H,J^\mathrm{tube}(H)) \colon H\in(0,\infty) \right\rbrace
\end{equation}

with the \textit{tube energy} $J^\mathrm{tube}(H)=-\frac{2H}{\kappa}$.

In the general case $2\tau^2-a\tau\kappa\neq0$, that is $\PSLzwei$, Heisenberg space $\Nildrei$, or $\Berger$ with non-horizontal pitch, the study of $\Theta_a$ demands more work, as there is no explicit formula for the closing defect $\delta_a$ at hand and we have to deal with the indeterminacy of the energy of the screw motion tubes. The aim is to define a continuous 1-parameter family of screw motion tubes (see Definition~\ref{def:tubefamily}).

\subsection{Structure of the tube region $\Theta_a$}\
\label{sec:moduli_space_separationproperty}

We gather some properties of the tube region $\Theta_a$. While we are interested in the case $2\tau^2-a\tau\kappa\neq0$, the following statement is also true for $2\tau^2-a\tau\kappa=0$.

\begin{lem}\label{lem:separationproperty}
	Suppose $a$ is admissible. Then the tube region $\Theta_a\subseteq\Xi$ satisfies: \vspace{-2mm}
	\begin{enumerate}[leftmargin=10mm]\setlength{\itemsep}{2mm}
		\item It is a non-empty real analytic set.
		\item It is a closed subset of $\Xi$ with empty interior.
		\item It separates $\Xi_a^+$ from $\Xi_a^-$. That is, the set $\Xi\backslash\Theta_a$ is not path-connected and $\Xi_a^+$ and $\Xi_a^-$ are contained in different path-connected components.
		\item It is a $(\mod 2)$ Euler space of dimension $1$. That is the underlying set of a locally finite simplicial complex, such that the link of each simplex has even Euler characteristic. 
	\end{enumerate}
\end{lem}

\begin{proof}
	(\textit{i}) The set $\Theta_a$ is non-empty by \cite[Thm.~3.6]{kaese}. The curve $\gamma_a(\sigma;H,J)$ is real analytic, because it is a solution of a real analytic ODE with real analytic initial conditions (Theorem of Cauchy-Kowalewsky). Since sums, products, quotients, compositions, derivatives and integrals of real analytic functions are again real analytic if defined (see e.g. \cite{krantz_parks}), $\delta_a$ is a real analytic function. Then the preimage $\delta_a^{-1}(0)$ is a real analytic set.
	
	(\textit{ii}) Fix $H=H_0$ or $J=J_0$ and consider the function $J\mapsto\delta_a(H_0,J)$ or $H\mapsto\delta_a(H,J_0)$ respectively. The common zero set of these functions is the set $\Theta_a$. These functions do not vanish identically, because $\delta_a(H,J)>0$ for all $(H,J)\in\Xi_a^+$ and $\delta_a(H,J)<0$ for all $(H,J)\in\Xi_a^-$, or vice versa, see \cite[Lem.~3.7]{kaese}. Because both functions are real analytic, the zeros must be isolated and the sets $\Theta_a\cap\lbrace H=H_0\rbrace$ and $\Theta_a\cap\lbrace J=J_0\rbrace$ are discrete. This implies that the interior of $\Theta_a$ must be empty. However, this does not imply that $\Theta_a$ is discrete. As the preimage of a closed set, $\Theta_a$ is a closed subset of $\Xi$.
	
	(\textit{iii}) Suppose there exists a continuous path connecting $\Xi_a^+$ and $\Xi_a^-$ in the complement~$\Xi\backslash\Theta_a$. This path connects points $(H_1,J_1)$ and $(H_2,J_2)$ in the moduli space with \mbox{$\delta_a(H_1,J_1)>0$} and $\delta_a(H_2,J_2)<0$. By the intermediate value theorem the path must pass through a point $(H_3,J_3)$ such that $\delta_a(H_3,J_3)=0$. This point would lie in $\Theta_a$ and we derive a contradiction.
	
	(\textit{iv}) It is a direct consequence of (\textit{i}) and \cite[Thm.~4.4]{fu_mccrory} that $\Theta_a$ is a $(\mod 2)$ Euler space (see also \cite{sullivan,burghelea_verona,hardt}). Because $\Theta_a$ has empty interior and separates $\Xi_a^+$ and $\Xi_a^-$, the simplicial complex has dimension 1.
\end{proof}

\subsection{Structure of the boundary of the tube region $\Theta_a$}\
\label{sec:moduli_space_boundary}

The tube region $\Theta_a$ separates $\Xi_a^+$ and $\Xi_a^-$. Roughly speaking this means that $\Theta_a$ must reach the boundary of $\Xi$ at some point. Otherwise we could just find a path in the complement $\Xi\backslash\Theta_a$ connecting  $\Xi_a^+$ and $\Xi_a^-$. The study of this behavior provides further insights in the structure of $\Theta_a$. For that purpose, let $\overline{\Xi}$ be the closure of $\Xi$ in $\modspace$ and $\partial\Xi=\overline{\Xi}\backslash\Xi$ its boundary. Then
\begin{equation}\label{eq:xi_closure}
	\overline{\Xi}=\left\lbrace
	\begin{array}{ll}
		\vphantom{\int\limits_0^0}\big\lbrace \left(H,J\right): H\geq0, J\in[-\frac{4H}{\kappa},0] \big\rbrace & \text{ for } \kappa>0, \\[6mm]
		\vphantom{\int\limits_0^0}\big\lbrace \left(H,J\right): H\geq\Hcrit, J\in(-\infty,0] \big\rbrace & \text{ for } \kappa\leq0.
	\end{array}
	\right. 
\end{equation}

The boundary $\partial\Xi$ is the union of two lines $\L_1$ and $\L_2$, see Figure~\ref{fig:modulispace}. For $\kappa>0$ these are the two lines
\begin{equation*}
	\L_1=\big\lbrace\!\left(H,0\right)\colon H\in[\Hcrit,\infty)\big\rbrace
	\qquad\text{and}\qquad
	\L_2=\left\lbrace\left(H,-\frac{4H}{\kappa}\right)\colon H\in[0,\infty)\right\rbrace,
\end{equation*}
and for $\kappa\leq0$ these are the two lines
\begin{equation*}
	\L_1=\big\lbrace\!\left(H,0\right)\colon H\in[\Hcrit,\infty)\big\rbrace
	\qquad\text{and}\qquad
	\L_2=\big\lbrace\!\left(\Hcrit,J\right)\colon J\in(-\infty,0]\big\rbrace.
\end{equation*}

Furthermore, let $\overline{\Theta}_a$ be the closure of the tube region $\Theta_a$ in $\overline{\Xi}$ and $\partial\Theta_a=\overline{\Theta}_a\backslash\Theta_a$ its boundary. Then $\partial\Theta_a\subseteq\partial\Xi$, because $\Theta_a$ is closed in $\Xi$ (see Lemma~\ref{lem:separationproperty}). Moreover, $\partial\Theta_a\subseteq\L_1$, because it is bounded away from $\L_2$:

\begin{lem}\label{lem:boundary_set}
	Suppose $a$ is admissible. Then $\partial\Theta_a\subseteq\L_1$. In particular, every point $(H,J)\in\partial\Theta_a$ has $J=0$.
\end{lem}

\begin{proof}
	Let $(H_0,J_0)\in\partial\Theta_a$ be a point in the boundary of the tube region. Then there must be a sequence $(H_n,J_n)_{n\in\N}$ in $\Theta_a$ converging to $(H_0,J_0)$. If we assume $(H_0,J_0)\not\in\L_1$, we can derive a contradiction. This is precisely what we do in the following, but separately for the three cases $\kappa>0$, $\kappa=0$, and $\kappa<0$. We note beforehand that the lines $\L_1$ and $\L_2$ intersect in the point $(0,0)$ if $\kappa\geq0$, and are disjoint if $\kappa<0$.
	
	\underline{Case $\kappa>0$:} Assume $(H_0,J_0)\in\L_2\backslash\lbrace(0,0)\rbrace$. Then $J_0=-\frac{4H_0}{\kappa}$. Because $\Theta_a$ is a subset of~$\Xi_a^0$, every point $(H_n,J_n)$ lies in $\Xi_a^0$ and therefore satisfies
	\begin{equation*}
		J_n\geq\Jone(H_n) =\frac{2H_n}{\kappa-4\tau^2}\,(2a\tau-1)\geq-\frac{2H_n}{\kappa}
	\end{equation*}
	for $a$ admissible. Then for $H_0>0$ we have
	\begin{equation*}
		\lim\limits_{n\to\infty} J_n \geq\lim\limits_{n\to\infty}\left(-\frac{2H_n}{\kappa}\right) =-\frac{2H_0}{\kappa}>-\frac{4H_0}{\kappa} =J_0,
	\end{equation*}
	i.e., $J_n$ does not converge to $J_0$, which is a contradiction.
	
	\underline{Case $\kappa=0$:} Assume $(H_0,J_0)\in\L_2\backslash\lbrace(0,0)\rbrace$. Then $H_0=\Hcrit=0$. Because $\Theta_a$ is a subset of~$\Xi_a^0$, every point $(H_n,J_n)$ lies in $\Xi_a^0$ and therefore satisfies
	\begin{equation*}
		J_n\geq\Jone(H_n) =\frac{H_n}{2\tau^2}\,(1-2a\tau)
	\end{equation*}
	and we have
	\begin{equation*}
		\lim\limits_{n\to\infty} J_n \geq\lim\limits_{n\to\infty}\left(\frac{H_n}{2\tau^2}\,(1-2a\tau)\right) =\frac{H_0}{2\tau^2}\,(1-2a\tau) =0>J_0,
	\end{equation*}
	i.e., $J_n$ does not converge to $J_0$, which is a contradiction.
	
	\underline{Case $\kappa<0$:} In the previous two cases our proof is based on the fact that $\Xi_a^0$ is bounded away from $\L_2$. But this is no longer true in case $\kappa<0$, see e.g. Figure~\ref{fig:modulispace_negcurv}. Thus, we have to give a different argument.
	
	Assume $(H_0,J_0)\in\L_2$. Then $H_0=\Hcrit$ and the profile curve $\gamma_a(H_0,J_0)$ with initial data $r(\frac{3\pi}{2};H_0,J_0)=r_-(H_0,J_0)$ and $h(\frac{3\pi}{2};H_0,J_0)=0$ generates a folium type surface by the classification theorem for screw motion surfaces of critical mean curvature (see~\cite{penafiel12} for a discussion and some examples and~\cite[Thm.~4.10]{kaese_phd} for the full classification). Furthermore, the profile curve $\gamma_a(H_n,J_n)$ with initial data $r(\frac{3\pi}{2};H_n,J_n)=r_-(H_n,J_n)$ and $h(\frac{3\pi}{2};H_n,J_n)=0$ generates a tube for all $n\in\N$. Then
	\begin{equation*}
		\lim\limits_{n\to\infty}r(\tfrac{3\pi}{2};H_n,J_n)
		=\lim\limits_{n\to\infty}r_-(H_n,J_n)
		=r_-(H_0,J_0)
	\end{equation*}
	and
	\begin{equation*}
		\lim\limits_{n\to\infty}h(\tfrac{3\pi}{2};H_n,J_n)
		=0,
	\end{equation*}
	i.e., the initial data of $\gamma_a(H_n,J_n)$ converge to the initial data of $\gamma_a(H_0,J_0)$. Since all curves are solution of the same analytic ODE, the curve $\gamma_a(H_n,J_n)$ must converge to the curve $\gamma_a(H_0,J_0)$. But
	\begin{equation*}
		\lim\limits_{n\to\infty}h(\tfrac{\pi}{2};H_n,J_n)=0,
	\end{equation*}
	because $h(\frac{\pi}{2};H_n,J_n)=0$ for all $n\in\N$, while the profile curve of a folium type surface has $h(\sigma;H_0,J_0)\to-\infty$ as $\sigma\to\frac{\pi}{2}$ (see~\cite[Prop.~4.14]{kaese_phd}), and we derive a contradiction.

	Altogether, $\partial\Theta_a\subseteq(\L_1\cup\L_2)$, but $\partial\Theta_a\cap(\L_2\backslash\L_1)=\emptyset$, and therefore $\partial\Theta_a\subseteq\L_1$.
\end{proof}

The existence of a sequence in $\Theta_a$ converging to the boundary $\partial\Theta_a$ provides a description of the boundary $\partial\Theta_a$ in terms of an integral equation:

\begin{lem}\label{lem:intersectionlimit}
	Suppose $a$ is admissible and $2\tau^2-a\tau\kappa\neq0$. Let $L_a$ denote the set of points $(H,0)\in\partial\Xi$, where $H$ is a solution of the integral equation
	\begin{equation}\label{eq:intersectionlimit}
		\int\limits_{\frac{\pi}{2}}^\pi p_a(\sigma;H)\,d\sigma 
		=\frac{\pi}{2}\abs{a}
	\end{equation}
	where
	\begin{equation*}
		p_a(\sigma;H)\coloneqq\frac{\left((4\tau-a\kappa)^2\sin^4\!\sigma+8H^2(2-4a\tau+a^2\kappa)\sin^2\!\sigma+16a^2H^4\right)^{1/2}}{4H^2+\kappa\sin^2\!\sigma}.
	\end{equation*}
	
	Then $\partial\Theta_a$ is a non-empty subset of $L_a$. Moreover, $\partial\Theta_a$ is the non-empty union of finitely many points $(H_i,0)$ with $H_i\in(\Hcrit,\Hex)$. In particular, $(0,0)\not\in\partial\Theta_a$.
\end{lem}

\begin{proof}
	We divide the proof into several steps.
	
	\underline{Step 1:} $(0,0)\not\in L_a$.
	
	Suppose $(0,0)\in L_a$. Then \eqref{eq:intersectionlimit} becomes
	\begin{equation*}
		\frac{\pi}{2}\frac{\abs{4\tau-a\kappa}}{\kappa}=\frac{\pi}{2}\abs{a}.
	\end{equation*}
	Squaring on both sides gives $2\tau^2-a\tau\kappa=0$, which contradicts our assumption $2\tau^2-a\tau\kappa\neq0$. Thus, $(0,0)\not\in L_a$.
	
	\underline{Step 2:} $L_a$ is the non-empty union of finitely many points $(H_i,0)$.
	
	Consider the function
	\begin{equation*}
		\ell_a\colon (\Hcrit,\infty)\to\R,\quad H\mapsto\int\limits_{\frac{\pi}{2}}^{\pi} \Big(p_a(\sigma;H)-\abs{a}\Big)\,d\sigma.
	\end{equation*}
	
	This function is real analytic, because $p_a$ is real analytic. If the set of zeros of $\ell_a$ contains an accumulation point, then $\ell_a$ must be the zero function by the Identity Theorem for real analytic functions \cite{krantz_parks}. However, a straightforward computation shows that the integrand $p_a(\sigma)-\abs{a}$ has a sign for all $H\geq\Hex$. Hence, $\ell_a(H)\neq0$ for all $H\geq\Hex$. Thus, the set of zeros of $\ell_a$ must consist of isolated points $H_i\in(0,\Hex)$. The zeros of $\ell_a$ cannot accumulate when approaching $H=\Hex$, due to $\ell_a(\Hex)\neq0$ (see computation above). Moreover, if we extend $\ell_a$ to $H=0$ for $\kappa>0$ we get $\ell_a(0)\neq0$ by Step~1. Hence, the zeros are also isolated in the compact interval $[0,\Hex]$ for $\kappa>0$, and there are only finitely many. In case $\kappa=0$ the function $p_a$ is unbounded for $H\to0$. In case $\kappa<0$ the function $p_a$ is well-defined for $H=\frac{\sqrt{-\kappa}}{2}$ and $\sigma\in(\frac{\pi}{2},\pi]$, but unbounded for $\sigma\to\frac{\pi}{2}$. Hence, the zeros are also isolated in the compact interval $[\frac{\sqrt{-\kappa}}{2},\Hex]$ for $\kappa\leq0$, so there can be only finitely many. Then the same is true for $L_a=\lbrace(H_i,0)\colon\ell_a(H_i)=0\rbrace$.
	
	\underline{Step 3:} $\partial\Theta_a\neq\emptyset$.
	
	Because $\Theta_a$ satisfies the separation property of Lemma~\ref{lem:separationproperty} (\textit{iii}), there exists a sequence $((H_n,J_n))_{n\in\N}$ in $\Theta_a$ converging to a point $(H_0,0)\in\partial\Theta_a$ by Lemma~\ref{lem:boundary_set}. Hence, $\partial\Theta_a\neq\emptyset$.

	\underline{Step 4:} $\partial\Theta_a\subseteq L_a$.
	
	Pick $(H_0,0)\in\partial\Theta_a$. Then there exists a sequence $((H_n,J_n))_{n\in\N}$ in $\Theta_a$ converging to $(H_0,0)$. Then $\delta_a(H_n,J_n)=0$ for all $n\in\N$ and
	\begin{equation}\label{eq:interchange_limit_1}
		\lim\limits_{n\to\infty}\delta_a(H_n,J_n)=0,
	\end{equation}
	
	see Figure~\ref{fig:sphere_jump}. Next we want to interchange limit and integration and prove
	\begin{equation}\label{eq:interchange_limit}
		\lim\limits_{n\to\infty}\delta_a(H_n,J_n)
		=\lim\limits_{n\to\infty} \int\limits_{\frac{\pi}{2}}^{\frac{3\pi}{2}} \frac{dh}{d\sigma}(\sigma;a,H_n,J_n)\,d\sigma
		=\int\limits_{\frac{\pi}{2}}^{\frac{3\pi}{2}} \lim\limits_{n\to\infty} \frac{dh}{d\sigma}(\sigma;a,H_n,J_n)\,d\sigma.
	\end{equation}
	
	\begin{figure}[]
		\centering
		\footnotesize
		\def\svgwidth{0.75\textwidth}\executeiffilenewer{sphere_jump.svg}{sphere_jump.pdf}{inkscape -z -D --file=sphere_jump.svg 	--export-pdf=sphere_jump.pdf --export-latex}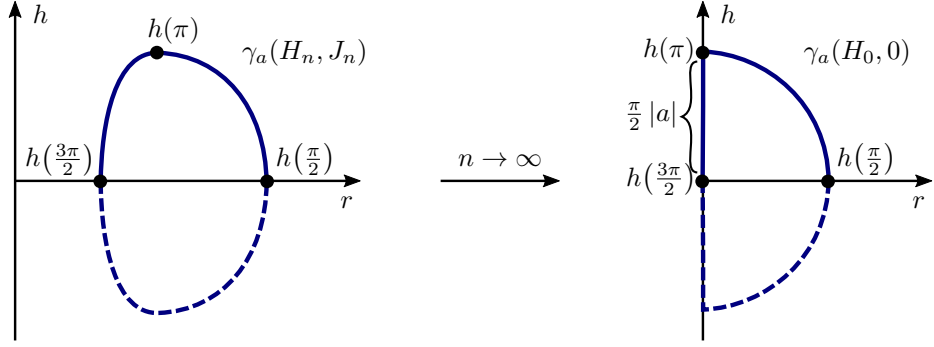
		\caption{Proof of Lemma~\ref{lem:intersectionlimit}, Step 4: Limit of a sequence of tube profile curves $\gamma_a(H_n,J_n)$ to a sphere type profile curve $\gamma_a(H_0,0)$.}
		\label{fig:sphere_jump}
	\end{figure}
	
	The integrand $\frac{dh}{d\sigma}$ is given by~\eqref{eq:height_function}. A careful inspection of \eqref{eq:height_function} taking into account the sign of $\sin\sigma$ gives the limit
	\begin{equation*}
		\lim\limits_{n\to\infty}\frac{dh}{d\sigma}(\sigma;a,H_n,J_n)=
		\left\lbrace\begin{array}{ll}
			p_a(\sigma;H_0) & \text{ for } \sin\sigma>0, \hspace{20mm}\, \\[4mm]
			-\abs{a} \hspace{1.4mm}\, & \text{ for } \sin\sigma<0.
		\end{array}\right.
	\end{equation*}
	
	This yields, splitting the integral at $\pi$,
	\begin{equation}\label{eq:interchange_limit_2}
		\begin{aligned}
			\int\limits_{\frac{\pi}{2}}^{\frac{3\pi}{2}} \lim\limits_{n\to\infty} \frac{dh}{d\sigma}(\sigma;a,H_n,J_n)\,d\sigma
			&=\int\limits_{\frac{\pi}{2}}^{\pi} p_a(\sigma;H_0)\,d\sigma -\frac{\pi}{2}\abs{a}.
		\end{aligned}
	\end{equation}
	
	Note that $\frac{dh}{d\sigma}(\sigma;a,H_n,J_n)$ does not converge uniformly in $\sigma$, because the limit function is no longer continuous. Nevertheless, all terms are uniformly bounded and we can interchange limit and integration \cite[Thm.~10.38]{apostol}. For $\kappa\leq0$ we have to restrict to $H>\epsilon>\Hcrit$ for the existence of a uniform bound. However, this does not affect our arguments since all points in $L_a$ are bounded away from $(\Hcrit,0)$. Thus, \eqref{eq:interchange_limit} remains true and \eqref{eq:interchange_limit_1} together with \eqref{eq:interchange_limit_2} proves \eqref{eq:intersectionlimit}. Therefore, $\partial\Theta_a\subseteq L_a$.
\end{proof}

\begin{rem}
	Numerical computations indicate that the integral equation~\eqref{eq:intersectionlimit} has a unique positive solution $\Hlimit$, so that $L_a$ and therefore also $\partial\Theta_a$ consist of only one point~$(\Hlimit,0)$. But again, we have no explicit formula for the integral, and the integrand $p_a(\sigma;H)$ is neither monotonic in $\sigma$ nor in $H$, so that we cannot prove this uniqueness here. A plot of the computed solution $\Hlimit$ as a function of the pitch $a$ is shown in Figure~\ref{fig:plot_H0}.
\end{rem}

Lemma~\ref{lem:intersectionlimit} does not only give us information about the boundary $\partial\Theta_a$, but it also implies the existence of tubes with $H<\Hex$:

\begin{cor}\label{cor:nonsharp_existence}
	Suppose $a$ is admissible and $2\tau^2-a\tau\kappa\neq0$. Then there exists a value $\Hex'<\Hex$, such that there exists a tube with mean curvature $H$ for all $H>\Hex'$. In particular, the existence bound $\Hex$ in \eqref{eq:definition_Hexistence} is not sharp.
\end{cor}

\begin{proof}
	Pick a point $(\Hex',0)$ in the boundary $\partial\Theta_a$, which is the limit of a separating component of~$\overline{\Theta}_a$. This separating component contains points $(H,J)\in\Theta_a$ for every $H>\Hex'$, i.e., there exists a tube with mean curvature $H$ for all $H>\Hex'$. Moreover, $\Hex'<\Hex$ by Lemma~\ref{lem:intersectionlimit}. 
\end{proof}

\begin{figure}[]
	\centering
	\footnotesize
	\def\svgwidth{0.48\textwidth}\executeiffilenewer{plot_H0_PSL2.svg}{plot_H0_PSL2.pdf}{inkscape -z -D --file=plot_H0_PSL2.svg 	--export-pdf=plot_H0_PSL2.pdf --export-latex}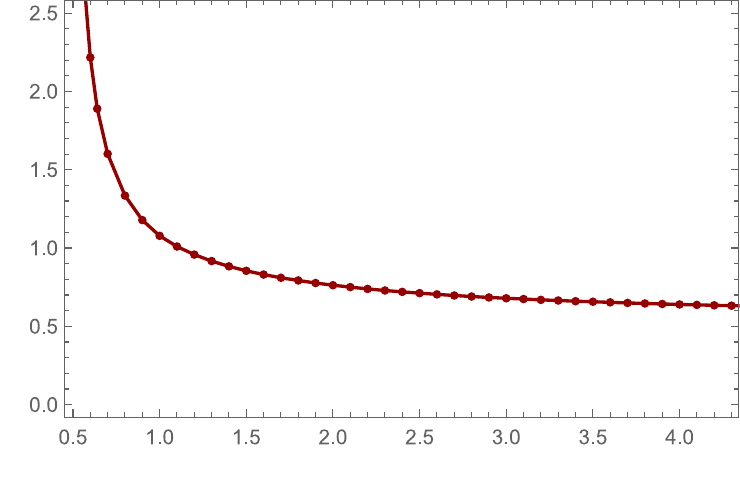
	\hfill
	\def\svgwidth{0.48\textwidth}\executeiffilenewer{plot_H0_Berger.svg}{plot_H0_Berger.pdf}{inkscape -z -D --file=plot_H0_Berger.svg 	--export-pdf=plot_H0_Berger.pdf --export-latex}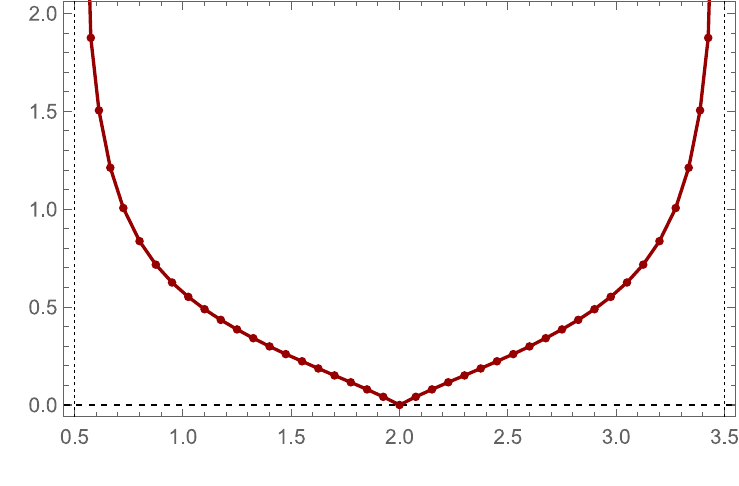
	\caption{Plot of the numerically computed solution $\Hlimit(a)$ of \eqref{eq:intersectionlimit} as a function of the pitch $a$ in $\PSLzwei=\E(-1,1)$ (left) and $\Berger=\E(1,1)$ (right). In both cases $\Hlimit(a)$ tends to infinity when $a$ approaches the boundaries of the admissible interval (vertical dotted lines). In the left plot $\Hlimit(a)$ approaches the critical curvature $\Hcrit=\frac{1}{2}$ as $a\to\infty$ (horizontal dotted line). In the right plot $\Hlimit(a)$ is symmetric with respect to the horizontal pitch $a=\frac{2\tau}{\kappa}$ with $\Hlimit(\frac{2\tau}{\kappa})=0$.}
	\label{fig:plot_H0}
\end{figure}

\subsection{Construction of a tube family $\Tfamily$}\
\label{sec:construction_tubefamily}

Some possible and impossible configurations of $\Theta_a$ subject to Lemma~\ref{lem:separationproperty} and Lemma~\ref{lem:intersectionlimit} are shown in Figure~\ref{fig:modulispace_tube_shapes}. We can now define a continuous tube family due to the following reasoning: There exists a connected component in $\Theta_a$, which separates~$\Xi_a^+$ from~$\Xi_a^-$ and converges to a point in $\partial\Theta_a$. It is the underlying set of a simplicial complex consisting only of edges and vertices. By removing edges and vertices from the simplicial complex we can construct a \textit{reduced} simplicial complex, which still separates, but the link of every vertex has Euler characteristic $2$. That is, this reduced simplicial complex can be parameterized as a continuous curve $(\Htube(\mu),\Jtube(\mu))$ for $\mu\in[0,\infty)$, where $(\Htube(0),\Jtube(0))\in\partial\Theta_a$ and $(\Htube(\mu),\Jtube(\mu))\to(\infty,-\infty)$ as $\mu\to\infty$. Such a curve need not be unique.

\begin{defi}\label{def:tubefamily}
	Let $(\Htube,\Jtube)\colon[0,\infty)\to\overline{\Theta}_a, \mu\mapsto(\Htube(\mu),\Jtube(\mu))$ be a curve as explained above. We call the continuous $1$-parameter family
	\begin{equation}\label{eq:def_tubefamily_nonhor}
		\Tfamily\coloneqq\left\lbrace T_a(\mu)\coloneqq\Sigma_a(\Htube(\mu),\Jtube(\mu)) \colon \mu\in(0,\infty) \right\rbrace
	\end{equation}
	a \textit{tube family}.
\end{defi}

\begin{figure}[]
	\centering
	\footnotesize
	\begin{subfigure}[t]{0.44\textwidth}
		\def\svgwidth{\textwidth}\executeiffilenewer{moduli_space_tube_1.svg}{moduli_space_tube_1.pdf}{inkscape -z -D --file=moduli_space_tube_1.svg 	--export-pdf=moduli_space_tube_1.pdf --export-latex}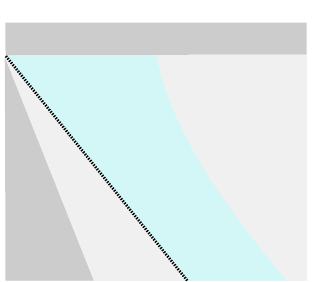
		\subcaption{$\Theta_a$ does not satisfy the separation property.}
	\end{subfigure}
	\hspace{12mm}
	\begin{subfigure}[t]{0.44\textwidth}
		\def\svgwidth{\textwidth}\executeiffilenewer{moduli_space_tube_3.svg}{moduli_space_tube_3.pdf}{inkscape -z -D --file=moduli_space_tube_3.svg 	--export-pdf=moduli_space_tube_3.pdf --export-latex}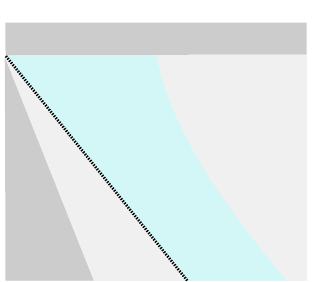
		\subcaption{$\Theta_a$ is not an Euler space, because it contains a vertex with three edges.}
	\end{subfigure}
	\\[4mm]
	\begin{subfigure}[t]{0.44\textwidth}
		\def\svgwidth{\textwidth}\executeiffilenewer{moduli_space_tube_2.svg}{moduli_space_tube_2.pdf}{inkscape -z -D --file=moduli_space_tube_2.svg 	--export-pdf=moduli_space_tube_2.pdf --export-latex}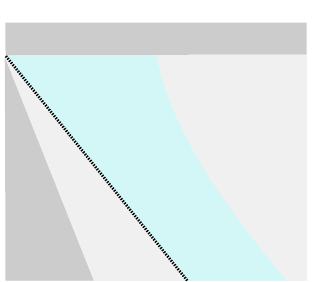
		\subcaption{$\partial\Theta_a$ has an accumulation point.}
	\end{subfigure}
	\hspace{12mm}
	\begin{subfigure}[t]{0.44\textwidth}
		\def\svgwidth{\textwidth}\executeiffilenewer{moduli_space_tube_0.svg}{moduli_space_tube_0.pdf}{inkscape -z -D --file=moduli_space_tube_0.svg 	--export-pdf=moduli_space_tube_0.pdf --export-latex}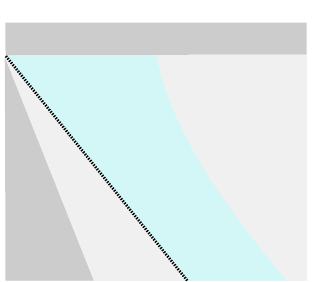
		\subcaption{$\Theta_a$ meets all requirements.}
	\end{subfigure}
	\caption{Four schematic pictures of $\Theta_a$ (dark blue). Pictures (a)-(c) violate some of the requirements of Lemma~\ref{lem:separationproperty} or Lemma~\ref{lem:intersectionlimit}. Only (d) satisfies all properties.}
	\label{fig:modulispace_tube_shapes}
\end{figure}

\begin{rem}\label{rem:construction_tubefamily}
	Numerical computation indicate that the curve $(\Htube(\mu),\Jtube(\mu))$ is unique. Moreover, the functions $\mu\mapsto\Htube(\mu)$ and $\mu\mapsto\Jtube(\mu)$ seem strictly monotonic for all admissible pitches. Therefore we expect we can reparameterize the curve as a graph over the mean curvature $H$,\vspace{3mm}
	\begin{equation}\label{eq:def_tubefamily_num}
		\Tfamily=\left\lbrace T_a(H)=\Sigma_a(H,\Jtube(H)) \colon H\in(\Hlimit,\infty) \right\rbrace,
	\end{equation}
	
	where $\Hlimit\coloneqq\Htube(\mu=0)$. A monotonicity result in this form would imply uniqueness of tubes for fixed pitch $a$ and mean curvature $H$. These observations led us to the conjecture stated in the introduction (see page~\pageref{conj:tube_existence}). The reader is invited to think of the family $\Tfamily$ as in \eqref{eq:def_tubefamily_num}. For the special case $2\tau^2-a\tau\kappa=0$ the definition in \eqref{eq:def_tubefamily_nonhor} coincides with \eqref{eq:def_tubefamily_hor} by choosing $\Htube(\mu)=\mu$ and $\Jtube(\mu)=-\frac{2\mu}{\kappa}$.
\end{rem}

\newpage

\subsection{Properties of the tube family $\Tfamily$}\

Despite the lack of uniqueness in the definition of $\Tfamily$, we can gather some properties of the tube family $\Tfamily$ for $\mu=0$ and $\mu\to\infty$.

\begin{lem}\label{lem:limitsurface}
	The limit surface of $\Tfamily$,
	\begin{equation*}
		\Sigma_a^{\lim}\coloneqq\lim\limits_{\mu\to0}T_a(\mu)
		=\lim\limits_{\mu\to0}\Sigma_a(\Htube(\mu),\Jtube(\mu))
	\end{equation*}
	 is, up to vertical translation, the surface $\Sigma_a(\Hlimit,\Jlimit)$ with $\Hlimit\coloneqq\Htube(0)$ and $\Jlimit\coloneqq\Jtube(0)=0$. Its maximal height $\hmax(\Hlimit,\Jlimit)=\frac{\pi}{2}\abs{a}$ is independent of the value of $\Hlimit$ and $\Jlimit$. If $2\tau^2-a\tau\kappa=0$, the limit surface is a spherical helicoid. Otherwise, $\Sigma_a^{\lim}$ is of sphere type.
\end{lem}

\begin{proof}
	The limit surface is by definition the surface $\Sigma_a(\Hlimit,\Jlimit)$ with $\Hlimit\coloneqq\Htube(0)$ and $\Jlimit\coloneqq\Jtube(0)=0$. It is $\Hlimit=0$ if $2\tau^2-a\tau\kappa=0$ and $\Hlimit>0$ if $2\tau^2-a\tau\kappa\neq0$ by Lemma~\ref{lem:intersectionlimit}. The classification theorem \cite[Thm.~3.1]{kaese} and Lemma~\ref{lem:minimal_helicoid} imply that this surface is a minimal spherical helicoid ($2\tau^2-a\tau\kappa=0$) or a sphere type surface ($2\tau^2-a\tau\kappa\neq0$). In both cases the maximal height $\hmax$ coincides with the left hand side of~\eqref{eq:intersectionlimit}, as the proof of Lemma~\ref{lem:intersectionlimit} implies. Therefore, $\hmax(\Hlimit,\Jlimit)=\frac{\pi}{2}\abs{a}$.
\end{proof}

\begin{lem}\label{lem:limit_infinity}
	The tube family $\Tfamily$ converges in distance to the geodesic $c_a(s)=(\rho_a,s,as)$ from Lemma~\ref{lem:group_properties} as $\mu\to\infty$.
\end{lem}

\begin{proof}
	By construction of the tube family we have $\Htube(\mu)\to\infty$ for $\mu\to\infty$. As $H$ tends to infinity, the limit of the minimal and maximal radius $r_\pm(H,J)$ (see \eqref{eq:radius_notation}) at energy levels $\Jone(H)$ and $\Jtwo(H)$ is given by
	\begin{equation*}
		\limH r_\pm(H,\Jone(H)) =\limH r_\pm(H,\Jtwo(H))
		=\arcs\left(\dfrac{\kappa}{\kappa-4\tau^2}\,(2a\tau-1)+1\right)=\rho_a.
	\end{equation*}	
	
	Since $\Jtube(\mu)$ always lies between $\Jone(\Htube(\mu))$ and $\Jtwo(\Htube(\mu))$ and $r_\pm(H,J)$ is monotonic in $J$, we can conclude
	\begin{equation*}
		\lim\limits_{\mu\to\infty} r_\pm(\Htube(\mu),\Jtube(\mu))=\rho_a.
	\end{equation*}	
	
	We can further conclude from $\Jone(\Htube(\mu))\leq\Jtube(\mu)\leq\Jtwo(\Htube(\mu))$ and the explicit expression \eqref{eq:energy} for $\Jone$ and $\Jtwo$ that $\Jtube(\mu)\sim\alpha\Htube(\mu)$ for some constant $\alpha<0$ as $\mu\to\infty$, i.e., $\Jtube\in\mathcal{O}(H)$. This implies that the numerator of $\frac{dh}{d\sigma}$ from \eqref{eq:height_function} is in $\mathcal{O}(H^2)$. On the other hand, the denominator of $\frac{dh}{d\sigma}$ is in $\mathcal{O}(H^3)$ by the same arguments. Therefore, $\frac{dh}{d\sigma}\in\mathcal{O}\left(\frac{1}{H}\right)$ for all $\sigma\in\left[\frac{\pi}{2},\frac{3\pi}{2}\right]$, and so $\frac{dh}{d\sigma}\to0$ as $H\to\infty$.
	
	Altogether, the family of profile curves $\lbrace\gamma_{a,\mu}\rbrace$ converges to the point $(\rho_a,0)\in\EKT/G_a$, and thus the family of tubes $\Tfamily$ converges to the curve $c_a(s)=L_{a,s}(\rho_a,0)=(\rho_a,s,as)$, which is a geodesic by Lemma~\ref{lem:group_properties}.
\end{proof}

Tubes with $2\tau^2-a\tau\kappa\neq0$ are characterized by a lower symmetry compared to tubes with $2\tau^2-a\tau\kappa=0$. Figure~\ref{fig:tubes_symmetry} shows two examples. This is particularly special in so far as these are the first examples of CMC tubes with lower symmetry.

\begin{figure}[]
	\centering
	\footnotesize
	\def\svgwidth{0.6\textwidth}\executeiffilenewer{tubes_symmetry.svg}{tubes_symmetry.pdf}{inkscape -z -D --file=tubes_symmetry.svg 	--export-pdf=tubes_symmetry.pdf --export-latex}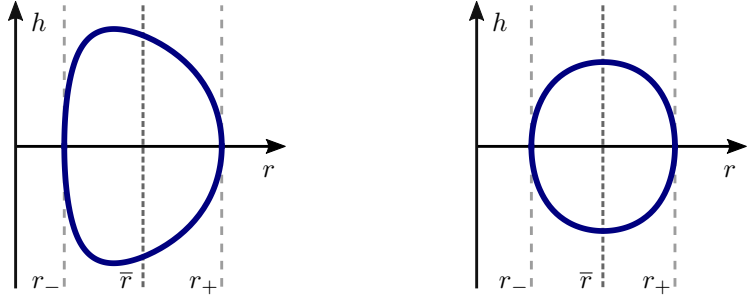
	\caption{Examples of profile curves of screw motion tubes. Both curves are symmetric under reflection at the line $\lbrace h=0\rbrace$, but only the right curve is also symmetric under reflection at the line $\lbrace r=\overline{r}\rbrace$.}
	\label{fig:tubes_symmetry}
\end{figure}

\begin{prop}\label{prop:dihedral_symmetry}
	The profile curve of a screw motion tube with pitch $a$ has dihedral symmetry of order $4$ if and only if $2\tau^2-a\tau\kappa=0$. That is the curve is invariant under reflection at two orthogonal lines. If $2\tau^2-a\tau\kappa\neq0$, the profile curve only has dihedral symmetry of order $2$.
\end{prop}

\begin{proof}
	Let $\gamma_a(\sigma;H,J)=(r(\sigma;H,J),h_a(\sigma;H,J))\subseteq\MKT/G_a$ be a profile curve of a tube with initial conditions $r(\frac{\pi}{2};H,J)=r_+(H,J)$ and $h_a(\frac{\pi}{2};H,J)=0$. Then $\gamma_a$ is symmetric under reflection at the line $\lbrace h=0\rbrace$ (see Section~\ref{sec:screwmotion_surfaces}), so the profile curve admits a dihedral symmetry of order at least $2$.
	
	If $2\tau^2-a\tau\kappa=0$, then the energy of any tube is $J=-\frac{2H}{\kappa}$ and the derivative of the height function $\frac{dh}{d\sigma}$, see~\eqref{eq:height_function}, is symmetric under reflection at the line $\lbrace r=\frac{\pi}{2\sqrt{\kappa}}\rbrace$, which is orthogonal to the line $\lbrace h=0\rbrace$ in $\MKT/G_a$. Thus, the profile curve has dihedral symmetry of order $4$. That is, it is invariant under the dihedral group consisting of two reflections.
	
	Now suppose $2\tau^2-a\tau\kappa\neq0$. The maximal height $\hmax(H,J)$ of the profile curve at $\sigma=\pi$ is only attained with $r(\pi;H,J)$. But for $\kappa\leq0$ it is always $r(\pi;H,J)\neq\frac{1}{2}(r_+(H,J)+r_-(H,J))$, see the explicit formula \eqref{eq:radius_sigma}. For $\kappa>0$ it is $r(\pi;H,J)=\frac{1}{2}(r_+(H,J)+r_-(H,J))$ only if $J=-\frac{2H}{\kappa}$, which does not lie in the interval between $\Jone$ and $\Jtwo$. Therefore, the maximal height is not attained with the mean radius $\overline{r}=\frac{1}{2}(r_+-r_-)$ of the tube. Hence, the profile curve cannot be symmetric under reflection at any line $\lbrace r=\text{const.}\rbrace$ and therefore does not admit a dihedral symmetry of order $4$.
\end{proof}

The symmetry of the profile curve is a consequence of the symmetry of ambient space. Horizontal geodesics in $\EKT$ admit rotations by an angle $\pi$, which induces reflections at any line $\lbrace h=h_0\rbrace$ in the orbit space. If $\lbrace h=h_0\rbrace$ meets the profile curve orthogonal, the profile curve is mapped onto itself under reflection at that line. Thus, the profile curve is invariant under reflection at that line. In particular, we obtain such a symmetry for the profile curves of screw motion tubes.

The additional symmetry of profile curves of screw motion tubes with $2\tau^2-a\tau\kappa=0$ comes from the base manifold $\s^2(\kappa)$. The equator line of $\s^2(\kappa)$ admits rotations by an angle $\pi$ as a horizontal geodesic of $\EKT$. This is again true for any surface in $\E(\kappa>0,\tau)$ and yields the conjugacy of the groups $G_a$ and $G_{\tilde a}$ (see Lemma~\ref{lem:group_properties}) and the previously mentioned one-to-one correspondence between $\Sigma_a(H,J)$ and $\Sigma_{\tilde a}(H,\tilde J)$ (see Section~\ref{sec:screwmotion_surfaces}). However, only the profile curves of screw motion tubes with $2\tau^2-a\tau\kappa=0$ are centered at the equator, meaning that the mean radius equals the equatorial radius $\frac{1}{2}(r_++r_-)=\frac{\pi}{2\sqrt{\kappa}}$. So only these profile curves are mapped onto themselves by reflection at the line $\lbrace\frac{\pi}{2\sqrt{\kappa}}\rbrace\times\R$ in the orbit space, and therefore admit an additional symmetry in contrast to the case $2\tau^2-a\tau\kappa\neq0$.

This symmetry result provides a geometric interpretation of the quantity $2\tau^2-a\tau\kappa$. It measures somehow the amount of asymmetry away from the more symmetric cases of $\SKR$ (with arbitrary pitch) and $\Berger$ with horizontal pitch $a=\frac{2\tau}{\kappa}$.

\subsection{A uniqueness result in Heisenberg space}\

In the special case of Heisenberg space $\Nildrei=\E(0,\tau), \tau>0$ we derive a uniqueness result for tubes with large mean curvature:

\begin{theo}\label{theo:uniqeness_nil}
	Consider $\Nildrei$ and suppose the pitch $a$ is admissible. Then there exists a value $\Hbound$, such that there exists exactly one tube with mean curvature $H$ around the geodesic $c_a$ for each $H>\Hbound$.
\end{theo}

\begin{proof}
	The existence of tubes with mean curvature $H>\Hex$ is guaranteed by \cite[Thm.~3.6]{kaese}. We obtain uniqueness if the function
	\begin{equation*}
		J\mapsto\delta_a(H,J)
	\end{equation*}
	is strictly monotonic for $J\in(\Jone(H),\Jtwo(H))$. We have no explicit formula at hand to prove monotonicity directly or to find a suitable estimate for the $J$-derivative of that function. Instead, we follow the idea from \cite[Lem.~3.7]{kaese} and compare $\frac{dh}{d\sigma}(\tilde\sigma)$ and $\frac{dh}{d\sigma}(\hat\sigma)$ for suitable pairs $(\tilde\sigma,\hat\sigma)$. This will lead to an auxiliary function $\psi(x)$ such that $\psi(x)\neq0$ is sufficient (but not necessary) for $\partial_J\delta_a(H,J)\neq0$. And so it will be sufficient for uniqueness.
	
	We divide our proof in two main steps. In step~1 we derive the aforementioned auxiliary function $\psi$. In step~2 we prove that this auxiliary function indeed satisfies $\psi(x)\neq0$.
	
	\underline{Step 1:} The $J$-derivative of the integrand $\dfrac{dh}{d\sigma}(\sigma;a,H,J)$ (see \eqref{eq:height_function}) is
	\begin{equation}\label{eq:height_Jderivative}
		\partial_J\left(\frac{dh}{d\sigma}\right)(\sigma;a,H,J)
		=\frac{\sqrt{g(\sigma;a,H,J)}}{\left(\sin^2\!\sigma-4HJ\right)^{3/2} \sqrt{f(\sigma;a,H,J)}}\,\sin\sigma.
	\end{equation}
	
	Here, $f$ is the function from \eqref{eq:height_function_f} with $\kappa=0$ and $g$ is the function
	\begin{equation*}
		g(\sigma;a,H,J)\coloneqq
		B_1\sin^2\sigma+B_2
		+B_3\sqrt{\sin^2\!\sigma-4HJ}\sin\sigma,
	\end{equation*}
	where $B_i=B_i(a,H,J)$ are the coefficient functions
	\begin{equation*}
		\begin{aligned}
			& B_1=2H\left(1-2a\tau\right) +4\tau^2J, \\
			& B_2=8H\left(a^2H^2-\tau^2J^2\right), \\
			& B_3=2H\left(1-2a\tau\right)+4\tau^2J.
		\end{aligned}
	\end{equation*}
	
	We observe that $\partial_J\left(\frac{dh}{d\sigma}\right)(\sigma)>0$ for $\sigma\in[\frac{\pi}{2},\pi)$ and $\partial_J\left(\frac{dh}{d\sigma}\right)(\sigma)<0$ for $\sigma\in(\pi,\frac{3\pi}{2}]$. For every $\hat\sigma\in[\frac{\pi}{2},\pi)$ there exists exactly one $\tilde\sigma\in(\pi,\frac{3\pi}{2}]$ such that $\sin\hat\sigma=-\sin\tilde\sigma\geq0$ and vice versa. The pointwise condition $-\partial_J\left(\frac{dh}{d\sigma}\right)(\tilde\sigma;a,H,J)<\partial_J\left(\frac{dh}{d\sigma}\right)(\hat\sigma;a,H,J)$ for all pairs $(\tilde\sigma,\hat\sigma)$ is sufficient for $\partial_J\delta_a(H,J)>0$. In the same way $-\partial_J\left(\frac{dh}{d\sigma}\right)(\tilde\sigma;a,H,J)>\partial_J\left(\frac{dh}{d\sigma}\right)(\hat\sigma;a,H,J)$ is sufficient for $\partial_J\delta_a(H,J)<0$. Both conditions are still sufficient if we exclude the pair $(\tilde\sigma,\hat\sigma)=(\frac{3\pi}{2},\frac{\pi}{2})$. The advantages are strict inequalities in the following as $|\sin\sigma|<1$.
	
	For all pairs $(\tilde\sigma,\hat\sigma)$ it is
	\begin{alignat*}{2}
		&& -\partial_J\left(\frac{dh}{d\sigma}\right)(\tilde\sigma) &\gtrless\partial_J\left(\frac{dh}{d\sigma}\right)(\hat\sigma) \\[4mm]
		\stackrel{\eqref{eq:height_Jderivative}}{\Leftrightarrow}&& \hspace{4mm} -\frac{\sqrt{g(\tilde\sigma)}}{\left(\sin^2\!\tilde\sigma-4HJ\right)^{3/2} \sqrt{f(\tilde\sigma)}}\,\sin\tilde\sigma
		&\gtrless \frac{\sqrt{g(\hat\sigma)}}{\left(\sin^2\!\hat\sigma-4HJ\right)^{3/2} \sqrt{f(\hat\sigma)}}\,\sin\hat\sigma \\[5mm]
		\Leftrightarrow&& g(\tilde\sigma) f(\hat\sigma) - g(\hat\sigma) f(\tilde\sigma)
		&\gtrless 0.
	\end{alignat*}
	
	A straightforward computation provides an expression for the last line:
	\begin{equation*}
		g(\tilde\sigma) f(\hat\sigma) - g(\hat\sigma) f(\tilde\sigma)
		=D_1\sin^2\sigma +D_2,
	\end{equation*}
	where $D_i=D_i(a,H,J)$ are the coefficient functions
	\begin{align*}
		& D_1 =32\tau^2H^2\left(2a^2H-J\left(2a\tau-1\right)\right), \\
		& D_2 =32H^2\left(2\tau^4J^3+\tau^2HJ^2\left(2a\tau-1\right)+H^2J\left(1-4a\tau-2a^2\tau^2\right)-a^2H^3\left(2a\tau-1\right)\right).
	\end{align*}
	We define the auxiliary function
	\begin{equation}
		\psi\colon(0,1)\to\R,\quad x\mapsto\psi(x)\coloneqq D_1x+D_2.
	\end{equation}
	
	Then $\psi(x)\neq0$ is indeed sufficient for $\partial_J\delta_a(H,J)\neq0$.
	
	\underline{Step 2:} We note that $D_1(a,H,J)$ is always positive for all $H>0$ if $a$ is admissible and $J\in(\Jone(H),\Jtwo(H))$. The same statement is not true for $D_2(a,H,J)$. However, we will show that there exists $\Hbound$ such that $D_2(a,H,J)>0$ for all $H>\Hbound$.
	
	First of all, the $J$-derivative of $D_2$ is given by
	\begin{equation*}
		\partial_J D_2(a,H,J)= 32H^2\left(6\tau^4J^2 +2\tau^2HJ\left(2a\tau-1\right) +H^2\left(1-4a\tau-2a^2\tau^2\right)\right).
	\end{equation*}
	
	This function is a parabola in $J$ open upwards. Let us denote its zeros by $j^\pm$. Then $[\Jone(H),\Jtwo(H)]\subset(j^-,j^+)$ for $H>\Hex$ using the explicit expression of the zeros $j^\pm$. This implies that $\partial_J D_2<0$ for all $J\in[\Jone(H),\Jtwo(H)]$ and $H>\Hex$.
		
	We can therefore estimate $D_2(a,H,J)\geq D_2(a,H,\Jtwo(H))$, where
	\begin{align*}
		D_2(a&,H,\Jtwo(H)) \\
		&=\frac{8}{\tau^2H} \left[ 2\left(1-6a\tau+8a^2\tau^2\right)H^6 +9\tau^2\left(1-4a\tau\right)H^4 +2\tau^4\left(1-2a\tau\right)H^2 +\tau^6 \right].
	\end{align*}
	
	Consider the terms in brackets. We substitute $y\coloneqq H^2$ and obtain a third-degree polynomial of the form $\alpha y^3+\beta y^2 +\gamma y +\delta$, where the prefactor of the leading order term is
	\begin{equation*}
		\alpha=2\left(1-6a\tau+8a^2\tau^2\right)>0.
	\end{equation*}
	
	Thus, the polynomial tends to $+\infty$ for large $y$. For $D_2(a,H,\Jtwo(H))$ this implies that there exists a value $\Hbound$, such that $D_2(a,H,\Jtwo(H))>0$ for all $H>\Hbound$.	To ensure existence of tubes and correctness of the previous arguments, choose $\Hbound\geq\Hex$. Then $\psi(x)>0$ for all $H>\Hbound$ and $J\in(\Jone(H),\Jtwo(H))$. \qedhere

\end{proof}

\begin{rem}
	The constant $\Hbound$ can be determined explicitly from the proof above as the square root of a zero of a third degree polynomial. A numerical evaluation shows that $\abs{\Hbound-\Hex}$ is small. We therefore obtain uniqueness for nearly all tubes.
\end{rem}

\begin{rem}\label{rem:uniqueness_nil}
	Step~1 of the proof of Theorem~\ref{theo:uniqeness_nil} can be generalized to $\EKT$. Step~2 on the other hand only applies to the special case of $\Nildrei$. For $\kappa\neq0$ the auxiliary function $\psi(x)$ is a second-degree polynomial. Unlike in $\Nildrei$ numerical computations indicate that the summands in the coefficient functions $D_i$ no longer satisfy useful monotonicity properties, which does not allow us to conclude $\psi(x)>0$ or $\psi(x)<0$. However, numerical computations also suggest that the statement remains true.
\end{rem}

\section{Embeddedness Results}
\label{sec:embedding}

For fixed admissible pitch $a$ we consider the family of tubes
\begin{equation*}
	\Tfamily=\left\lbrace T_a(\mu)\colon \mu\in(0,\infty) \right\rbrace
\end{equation*}
from \eqref{eq:def_tubefamily_nonhor}. We want to study embeddedness and foliation properties of this family. In the non-compact spaces $\SKR$, $\Nildrei$, and $\PSLzwei$ the profile curve $\gamma_a(\mu)$ of a tube $T_a(\mu)\in\Tfamily$ is an embedded simple loop. Thus, embeddedness of the tube can only fail when the screw motion causes the tube to intersect itself. Therefore,
\begin{equation}\label{eq:embedding_condition}
	2\hmax(\mu)<2\pi a
\end{equation}
is a necessary and sufficient condition for embeddedness, with the maximal height (see \eqref{eq:height_max})
\begin{equation*}
	\hmax(\mu)\coloneqq\hmax(\Htube(\mu),\Jtube(\mu)).
\end{equation*}

The limit surface $\Sigma_a^{\lim}=\lim\limits_{\mu\to0}T_a(\mu)$ and the geodesic $c_a=\lim\limits_{\mu\to\infty}T_a(\mu)$ (see Lemma~\ref{lem:limit_infinity}) are embedded. This implies that tubes in the family $\Tfamily$ close to $\Sigma_a^{\lim}$ and $c_a$ in $\Tfamily$ are also embedded:

\begin{theo}\label{theo:embedding_noncompact}
    Consider $\SKR$, $\Nildrei$, or $\PSLzwei$ and suppose that the pitch $a$ is admissible. Then there exists a neighborhood of $\mu=0$ and a neighborhood of $\mu=\infty$, such that $T_a(\mu)$ is embedded for all $\mu$ in these neighborhoods.
\end{theo}

Figure~\ref{fig:embedded_nil} shows some examples of embedded tubes in $\Nildrei$ and $\PSLzwei$.

\begin{proof}
	The embeddedness condition~\eqref{eq:embedding_condition} is satisfied for small $\mu$, because $\hmax(\mu)\to\frac{\pi}{2}\abs{a}$ as $\mu\to0$ by Lemma~\ref{lem:limitsurface}. Moreover, \eqref{eq:embedding_condition} is also satisfied for large $\mu$, because $\Tfamily$ converges to the embedded geodesic $c_a$ with $\hmax(\mu)\to0$ as $\mu\to\infty$ by Lemma~\ref{lem:limit_infinity}.
\end{proof}

\begin{figure}[]
	\centering
	\small
	\hspace{15mm}
	\vstretch{.75}{\includegraphics[height=0.22\textheight]{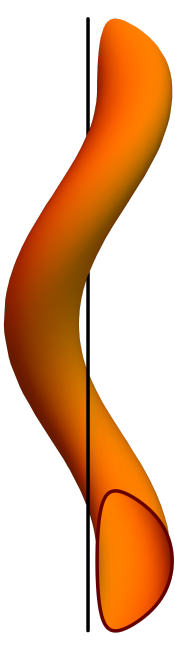}}
	\hfill
	\vstretch{.75}{\includegraphics[height=0.22\textheight]{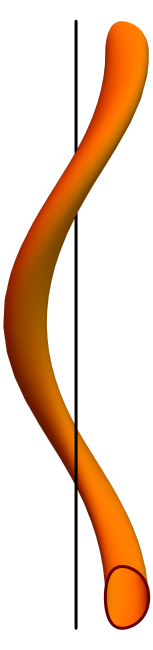}}
	\hfill
	\vstretch{.75}{\includegraphics[height=0.22\textheight]{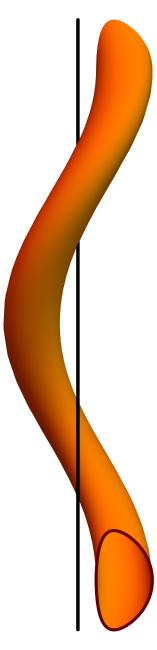}}
	\hfill
	\vstretch{.75}{\includegraphics[height=0.22\textheight]{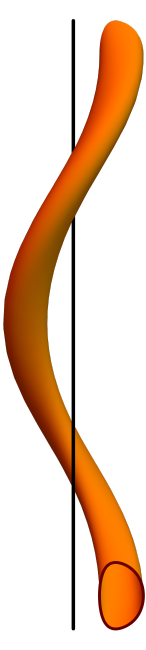}}
	\hspace{17mm}
	\caption{Numerically obtained examples of embedded tubes in $\Nildrei=\E(0,\tau)$ and $\PSLzwei=\E(-1,\tau)$ with pitch $a=1$. The chosen parameters are (left to right): $H=1$ in $\E(0,1)$, $H=2$ in $\E(0,2)$, $H=1.5$ in $\E(-1,1)$, and $H=2$ in $\E(-1,2)$. These tubes extend by simple periodicity and are non-compact.}
	\label{fig:embedded_nil}
\end{figure}

Now let us turn our attention to Berger spheres. Our model \eqref{model} is a non-compact fiber-wise covering of $\Berger$. In order to study the embeddedness of tubes in the compact setting, we just consider the vertical direction to be periodic. This is equivalent to compact fibers. Then condition~\eqref{eq:embedding_condition} is only necessary for embeddedness, but not sufficient anymore.

First of all, the orbits of the screw motion group $G_a$ must be compact, so that the tubes close up. The length of a fiber is~$\frac{8\pi\tau}{\kappa}$ and the required closing condition reads
\begin{equation}\label{eq:closing_condition}
	\frac{8\pi\tau}{\kappa}\,n=2\pi a\,m
\end{equation}
for some $n,m\in\N$. Here, $m$ represents the number of rotations around the screw motion axis and $n$ the number of rotations around the antipodal axis. The case $n=1$ is of special interest to us, as it reduces the possibility of self-intersections. A tube satisfying~\eqref{eq:closing_condition} must have pitch
\begin{equation*}
	a_{n,m}\coloneqq \frac{n}{m}\,\frac{4\tau}{\kappa}.
\end{equation*}

For $n=1$ and $m=2$ we recover the horizontal pitch $a_{1,2}=\frac{2\tau}{\kappa}$.

\begin{lem}\label{lem:closing_condition_berger}
	The pitch $a=a_{n,m}$ is admissible in $\Berger$ if and only if
	\begin{equation}\label{eq:pitch_berger}
		\frac{n}{m}\,\epsilon\in
		\left[\frac{\epsilon}{2}, \frac{\kappa}{8\tau^2}\,\epsilon\right),
	\end{equation}
	
	where $\epsilon\coloneqq\sgn(\kappa-4\tau^2)$. For $n=1$ there exists a tube with pitch $a=a_{1,m}$ if and only if
	\begin{align}
		&m\epsilon\leq2\epsilon
		\qquad\text{and}\qquad
		\left(\kappa-\frac{8\tau^2}{m}\right)\,\epsilon>0, \label{eq:pitch_berger1} \\
		\shortintertext{or}
		&m\epsilon\geq2\epsilon
		\qquad\text{and}\qquad
		\left(\kappa-\frac{8(m-1)\tau^2}{m}\right)\,\epsilon>0. \label{eq:pitch_berger2}
	\end{align}
	
	In particular, if $\kappa-8\tau^2>0$, there exists a tube with pitch $a=a_{1,m}$ for all $m\in\N$.
\end{lem}

\begin{proof}
	Condition~\eqref{eq:pitch_berger} follows directly from the definition of $a_{n,m}$ and Definition~\ref{def:admissible} of admissible pitch. For $n=1$ this is equivalent to \eqref{eq:pitch_berger1}. The existence theorem \cite[Thm.~3.6]{kaese} states that tubes exist if the pitch is admissible, or if it is conjugate to an admissible pitch. The conjugated pitch of $a_{n,m}$ is
	\begin{equation*}
		\tilde a_{n,m}=\frac{4\tau}{\kappa}-a_{n,m}=a_{m-n,m}.
	\end{equation*}
	
	It is admissible if and only if
	\begin{equation}
		\frac{m-n}{m}\,\epsilon\in
		\left[\frac{\epsilon}{2}, \frac{\kappa}{8\tau^2}\,\epsilon\right).
	\end{equation}
	
	For $n=1$ this is equivalent to \eqref{eq:pitch_berger2}. An evaluation of the conditions \eqref{eq:pitch_berger1} and \eqref{eq:pitch_berger2} shows that these are always met if $\kappa-8\tau^2>0$.
\end{proof}

A tube with pitch $a_{1,m}$ intersects itself if $m$-times the vertical diameter of the profile curve ($=2\hmax$) is equal or larger than the length of a fiber. That is,
\begin{equation}\label{eq:embedding_condition2}
    2m\hmax(\mu)<\frac{8\pi\tau}{\kappa}
\end{equation}
must be satisfied for embeddedness. This inequality coincides with the embeddedness condition~\eqref{eq:embedding_condition}. Thus, we can disregard~\eqref{eq:embedding_condition2} and obtain an embedding result similar to the non-compact case:

\begin{theo}\label{theo:embedding_compact}
    Consider the Berger sphere $\Berger$ and suppose that the pitch is $a=a_{1,m}$, where $m$ satisfies \eqref{eq:pitch_berger1} or \eqref{eq:pitch_berger2}, respectively. Then there exists a neighborhood of $\mu=0$ and a neighborhood of $\mu=\infty$, such that $T_a(\mu)$ is embedded for all $\mu$ in these neighborhoods.
\end{theo}

\begin{proof}
    The pitch $a_{1,m}$ satisfies the closing condition~\eqref{eq:closing_condition} and the tube family exists by Lemma~\ref{lem:closing_condition_berger}. The embeddedness condition~\eqref{eq:embedding_condition} is satisfied for small $\mu$, because $\hmax(\mu)\to\frac{\pi}{2}\abs{a}$ as $\mu\to0$ by Lemma~\ref{lem:limitsurface}. Moreover, \eqref{eq:embedding_condition} is also satisfied for large $\mu$, because $\Tfamily$ converges to the embedded geodesic $c_a$ with $\hmax(\mu)\to0$ as $\mu\to\infty$ by Lemma~\ref{lem:limit_infinity}.
\end{proof}

\begin{figure}[]
	\centering
	\small
	\hspace{1mm}
	\includegraphics[height=0.15\textheight]{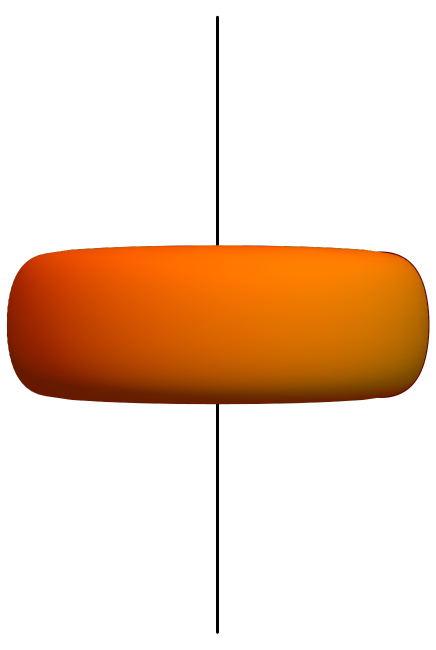}
	\hfill
	\includegraphics[height=0.15\textheight]{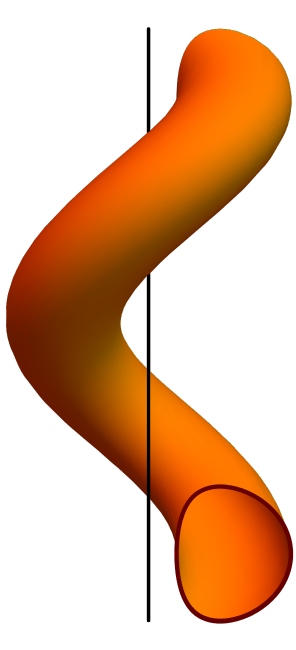}
	\hfill
	\includegraphics[height=0.15\textheight]{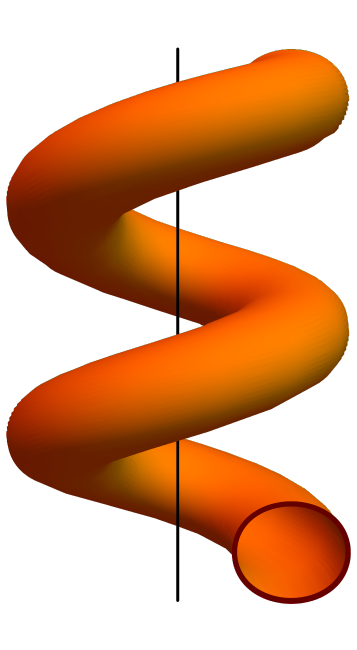}
	\hfill
	\includegraphics[height=0.15\textheight]{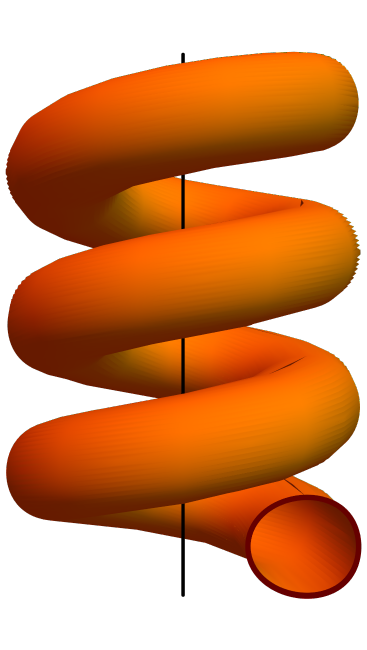}
	\hfill
	\includegraphics[height=0.15\textheight]{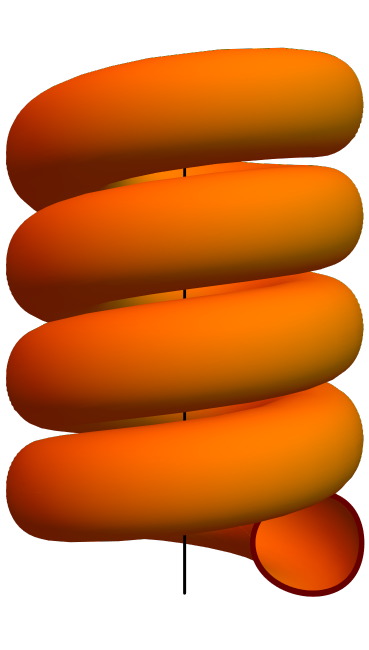}
	\hfill
	\includegraphics[height=0.15\textheight]{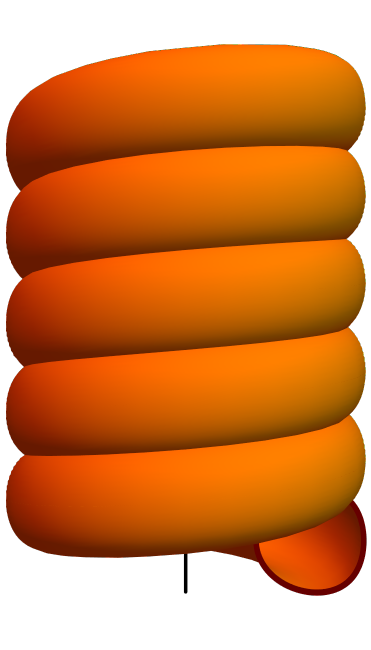}
	\hspace{1mm}
	\caption{Numerically obtained examples of tubes in $\Berger=\E(4,0.5)$ with $H=1$. The pitch is (left to right): $a=0$ (rotational), $a=a_{1,1}$, $a=a_{1,2}$ (horizontal), $a=a_{1,3}$, $a=a_{1,4}$, and $a=a_{1,5}$. The first two tubes are isometric. All tubes are compact in $\Berger$ and all tubes except the last one are embedded.}
	\label{fig:embedded_berger}
\end{figure}

For arbitrary mean curvature, the embeddedness condition~\eqref{eq:embedding_condition} is hard to verify, because $\hmax$ is not explicit. Figure~\ref{fig:embedded_berger} shows some examples of embedded and non-embedded tubes in $\Berger$. Tubes tend to be embedded for small $m$ and non-embedded for large $m$ depending on the explicit value of $\kappa$ and $\tau$.

\section{Foliation Results}
\label{sec:foliation}

We turn our attention to foliations by tubes and consider the special case $2\tau^2-a\tau\kappa=0$, that is, $\SKR$ with arbitrary pitch $a\in\R$ or $\Berger$ with horizontal pitch $a=\frac{2\tau}{\kappa}$. Recall that in this special case $\Htube(\mu)=\mu$ and $\Jtube(\mu)=-\frac{2\mu}{\kappa}$ and we can consider the mean curvature $H$ as the parameter of $\Tfamily$ as explained in Remark~\ref{rem:construction_tubefamily}. In the Berger sphere case Manzano has shown that the family of horizontal tubes $\Tfamily$ produces a foliation if and only if $(1-x_0^2)\kappa-4\tau^2\leq0$, where $x_0\approx0.83356$ is the unique positive solution of $x\artanh(x)=1$ \cite[Thm.~4.1]{manzano24}. If $\Tfamily$ produces a foliation, then it can be extended to foliate $\Berger$ without the two geodesics arising as $H\to\pm\infty$. We prove a similar result for $\SKR$.

\begin{theo}\label{theo:foliation_skr}
    In $\SKR$ the family of tubes $\Tfamily$ produces a foliation if and only if
    \begin{equation*}
        \abs{a}\geq\sqrt{\frac{1-x_0^2}{x_0^2\,\kappa}},
    \end{equation*}
    where $x_0\approx0.83356$ is the unique positive solution of $x\artanh(x)=1$. Moreover, if $\Tfamily$ produces a foliation, we can extend $\Tfamily$ by adding the minimal spherical helicoid $\Sigma_{a,0,0}$ and the vertically translated family $\Tfamily'\coloneqq\Tfamily+(0,0,\pi\abs{a})$ with opposite orientation. This extended family
    \begin{equation}\label{eq:extended_tubefamily}                   
	\overline{\Tfamily}\coloneqq\Tfamily\cup\lbrace\Sigma_{a,0,0}\rbrace\cup\Tfamily'
	=\left\lbrace T_a(H): H\in(-\infty,\infty)\right\rbrace
	=\left\lbrace T_a(H): 4H^2+\kappa>0\right\rbrace
    \end{equation}
    foliates $\SKR$ without the two geodesics arising as $H\to\pm\infty$, see Figure~\ref{fig:foliation_surface}.
\end{theo}

\begin{figure}[]
	\centering
	\small
	\def\svgwidth{0.6\textwidth}\executeiffilenewer{foliation_surface_proof.svg}{foliation_surface_proof.pdf}{inkscape -z -D --file=foliation_surface_proof.svg 	--export-pdf=foliation_surface_proof.pdf --export-latex}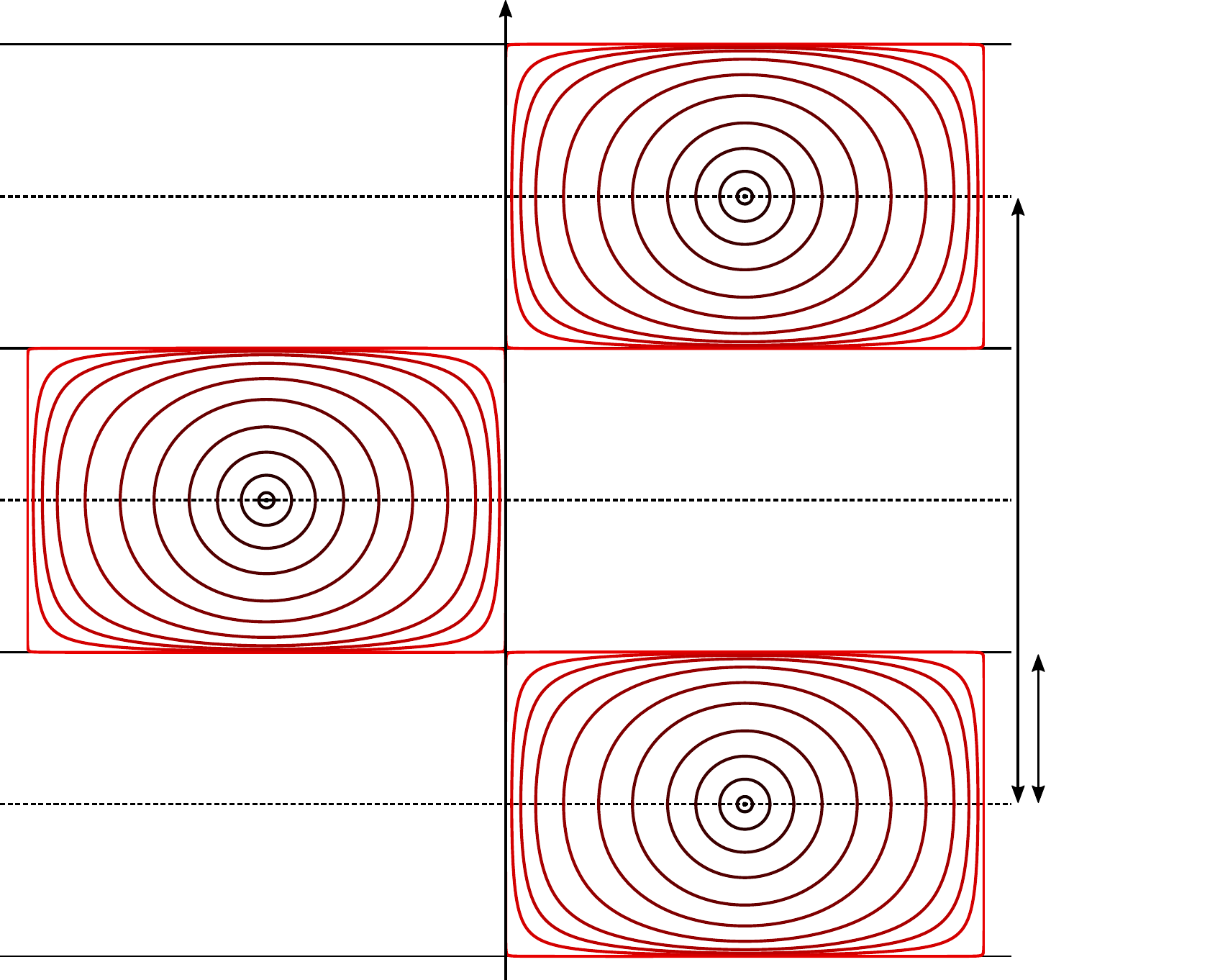
	\caption{Foliation by tubes in $\SKR$, see Theorem~\ref{theo:foliation_skr}. The bold curves represent the tube family $\Tfamily$, the thin curves the vertically translated family $\Tfamily'$.}
	\label{fig:foliation_surface}
\end{figure}

\begin{proof}
    Consider first the family $\Tfamily$ parameterized by $H>0$. As $H\to\infty$, the profile curves $\gamma_a(\sigma;H)=(r(\sigma;H),h_a(\sigma;H))$ converge uniformly to $(\frac{\pi}{2\sqrt{\kappa}},0)$ (see Lemma~\ref{lem:limit_infinity}), where $\gamma_a(\sigma;H)$ stays away from $(\frac{\pi}{2\sqrt{\kappa}},0)$ for all $H>0$. We can further assume $\sigma\in[\frac{\pi}{2},\pi]$ due to the dihedral symmetry of the profile curve (see Proposition~\ref{prop:dihedral_symmetry}). 
	
    The family of curves $\gamma_a(\sigma;H)$ is continuous with respect to $H$. Therefore it will fail to foliate the orbit space if and only if we can find $H_1>H_2>0$ such that $\gamma_a(\,\cdot\,;H_1)$ is tangent to $\gamma_a(\,\cdot\,;H_2)$ at some point. That is, there are $\sigma_1,\sigma_2\in[\frac{\pi}{2},\pi]$ and $\rho>0$ such that $\gamma_a(\sigma_1;H_1)=\gamma_a(\sigma_2;H_2)$ and $\gamma'_a(\sigma_1;H_1)=\rho\,\gamma'_a(\sigma_2;H_2)$. As $\sigma$ is the angle between the tangent $\gamma'$ and the radial direction, this implies $\sigma_1=\sigma_2$. Therefore the tangential point is characterized by
    \begin{align}
        (r(\sigma;H_1),h_a(\sigma;H_1))&=(r(\sigma;H_2),h_a(\sigma;H_2)) \label{eq:foliation_condition1} \\
        (r'(\sigma;H_1),h'_a(\sigma;H_1))&=(\rho\,r'(\sigma;H_2),\rho\,h'_a(\sigma;H_2)). \label{eq:foliation_condition2}
    \end{align}

	For $\sigma\neq\pi$ we obtain $r(\sigma;H)=\arct\left(\frac{2H}{\sin\sigma}\right)+\frac{\pi}{2\sqrt{\kappa}}$ from \eqref{eq:radius_sigma}. Then the first component of \eqref{eq:foliation_condition1} is equivalent to
	\begin{equation}\label{eq:foliation_proof_1}
		\frac{H_1}{\sin\sigma}=\frac{H_2}{\sin\sigma},
	\end{equation}
	
	which implies $H_1=H_2$, contradicting our assumption $H_1>H_2$. This means no such tangent point can occur for $\sigma\neq\pi$. For $\sigma=\pi$ we get $r(\pi;H_1)=r(\pi;H_2)=\frac{\pi}{2\sqrt{\kappa}}$. Therefore the family $\lbrace\gamma_a(\,\cdot\,;H)\rbrace$ produces a foliation if and only if $H\mapsto h_a(\pi;H)=\hmax(H)$ is monotonic.
    
    We get an explicit expression for $\hmax(H)$ from \eqref{eq:height_max} and \eqref{eq:height_function} by integration
    \begin{equation}\label{eq:height_skr}
        \begin{aligned}
        \hmax(H)
        &=\int\limits_{\frac{\pi}{2}}^\pi \frac{\sqrt{a^2\kappa^2\sin^2\!\sigma+4H^2(1+\kappa a^2)}}{4H^2+\kappa\sin^2\!\sigma}\,\sin\sigma\,d\sigma \\
        &=-\frac{2H}{\sqrt{\kappa}\sqrt{4H^2+\kappa}}\, \arcoth\!\left(\frac{\sqrt{4H^2+\kappa}\sqrt{1+\kappa a^2}}{\sqrt{\kappa}}\right) \\
        &\hspace{20mm}-\abs{a}\arcsin\!\left(\frac{\abs{a}\kappa}{\sqrt{a^2\kappa^2+4H^2(1+\kappa a^2)}}\right).
        \end{aligned}
    \end{equation}
    
    The derivative of \eqref{eq:height_skr} with respect to $H$ is\vspace{3mm}
    \begin{equation}\label{eq:height_Hderivative}
        \partial_H\hmax(H)
        =\frac{2\sqrt{\kappa}}{(4H^2+\kappa)^{3/2}}\,\arcoth\!\left(\frac{\sqrt{4H^2+\kappa}\sqrt{1+\kappa a^2}}{\sqrt{\kappa}}\right)
        -\frac{2\sqrt{1+\kappa a^2}}{4H^2+\kappa}.
    \end{equation}

    Substitute $x\coloneqq\frac{\sqrt{\kappa}}{\sqrt{4H^2+\kappa}\sqrt{1+\kappa a^2}}$. Then $\partial_H\hmax(H)\leq0$ reads\vspace{3mm}
    \begin{equation*}
        \arcoth\!\left(\frac{1}{x}\right)-\frac{1}{x}\leq0
        \qquad
        \Leftrightarrow
        \qquad
        x\arctan(x)\leq1.
    \end{equation*}
    
    We resubstitute and see that $\partial_H\hmax(H)\leq0$ for all $H>0$ if and only if $\abs{a}\geq\sqrt{\frac{1-x_0^2}{x_0^2\,\kappa}}$. Thus, the profile curves of $\Tfamily$ produce a foliation in the orbit space $\EKT/G_a$ if and only if $\abs{a}\geq\sqrt{\frac{1-x_0^2}{x_0^2\,\kappa}}$.
    
    It remains to show that also the tubes produce a foliation and that this foliation actually covers $\EKT$ minus the geodesic $c_a$. Since $\gamma_a(\,\cdot\,;H)$ converges to the spherical helicoid as $H\to0$, we deduce that $\lbrace\gamma_a(\,\cdot\,;H)\colon H\in(0,\infty)\rbrace$ foliates the entire open rectangle $(0,\frac{\pi}{\sqrt{\kappa}})\times(-\frac{\pi}{2}\abs{a},\frac{\pi}{2}\abs{a})$ minus the point $(\frac{\pi}{2\sqrt{\kappa}},0)$ corresponding to $c_a$. In particular, the generated tubes cannot intersect themselves as \eqref{eq:embedding_condition} is satisfied. Moreover, the foliation is complete if we add the minimal spherical helicoid ($H=0$) as well as the vertically translated family with reversed orientation ($H<0$), see Figure \ref{fig:foliation_surface}.
\end{proof}

This foliation result implies that (under the given constrains) tubes in $\SKR$ are embedded. This was missing in \cite{vrzina16}, where the existence of these tubes were proven.

In case the family $\Tfamily$ does not produce a foliation, we can always remove some tubes with small mean curvature, such that the remaining tubes produce a foliation of a proper subset of the ambient space:

\begin{cor}\label{cor:foliation_skr}
    For $\SKR$ suppose $\abs{a}<\sqrt{\frac{1-x_0^2}{x_0^2\,\kappa}}$, i.e., $\Tfamily$ does not produce a foliation. Then the tube height $\hmax(H)$ attains its maximum at $H=\frac{\sqrt{\kappa}}{2}\sqrt{\frac{1-x_0^2(1+\kappa a^2)}{x_0^2(1+\kappa a^2)}}$ and the subfamily of tubes
    \begin{equation*}
        \hat{\Tfamily}\coloneqq\left\lbrace T_a(H):
        H\in\left(\frac{\sqrt{\kappa}}{2}\sqrt{\frac{1-x_0^2(1+\kappa a^2)}{x_0^2(1+\kappa a^2)}},\infty\right)\right\rbrace
    \end{equation*}
    foliates an open subset of $\SKR$.
\end{cor}

\begin{proof}
    We can conclude from the proof of Theorem~\ref{theo:foliation_skr} that the tube height $\hmax(H)$ attains its maximum when $\partial_H\hmax(H)=0$, and that the tubes foliate if and only if $\partial_H\hmax(H)<0$. Using the explicit expression \eqref{eq:height_Hderivative} yields the statement.
\end{proof}

Following the same arguments we can derive the analog result in $\Berger$ from the proof of \cite[Thm.~4.1]{manzano24}:

\begin{cor}\label{cor:foliation_berger}
    For $\Berger$ suppose $(1-x_0^2)\kappa-4\tau^2>0$, i.e., $\Tfamily$ with horizontal pitch $a=\frac{2\tau}{\kappa}$ does not produce a foliation. Then the tube height $\hmax(H)$ attains its maximum at $H=\sqrt{\frac{(1-x_0^2)\kappa-4\tau^2}{4x_0^2}}$ and the subfamily of tubes
    \begin{equation*}
        \hat{\Tfamily}\coloneqq\left\lbrace T_a(H):
        H\in\left(\sqrt{\frac{(1-x_0^2)\kappa-4\tau^2}{4x_0^2}},\infty\right)\right\rbrace
    \end{equation*}
    foliates an open subset of $\Berger$.
\end{cor}

\section{Isoperimetric Profile in Berger spheres}
\label{sec:isoperimetric_profile}

We consider the isoperimetric problem in Berger spheres:

\begin{center}
	\parbox{0.8\textwidth}{
		\textit{Given a number $V\in\left(0,\vol(\Berger)\right)$, find the embedded compact surfaces of least area enclosing a domain of volume $V$.}
	}
\end{center}


Torralbo and Urbano have proven that for $\frac{4\tau^2}{\kappa}\in\left[\frac{1}{3},1\right)$ the solutions to the isoperimetric problem are CMC spheres~\cite[Cor.~7.1]{torralbo_urbano}. But for $\frac{4\tau^2}{\kappa}\in\left(0,\frac{1}{3}\right)$ the situation is quite different: There are examples of stable CMC Hopf tori, but also unstable CMC spheres in $\Berger(\kappa,\tau)$ (see \cite{torralbo_urbano}). In this case the solutions can be either stable CMC spheres or stable CMC~tori and the isoperimetric problem is still open. Manzano has recently shown numerically that horizontal tubes can be discarded from this question, as they have larger area for given volume than CMC spheres if $\kappa-4\tau^2>0$, and larger area than CMC Hopf tori if $\kappa-4\tau^2<0$~\cite{manzano24}. We provide numerical evidence indicating that also the screw motion tubes can be discarded by computing the isoperimetric profile $H\mapsto(\vol(T_a(H)),\area(T_a(H)))$.

The tube $T_a(H)$ is compact if and only if the pitch satisfies the closing condition~\eqref{eq:closing_condition}, that is $a=a_{n,m}=\frac{n}{m}\,\frac{4\tau}{\kappa}$. As before we restrict our analysis to $n=1$.

The tube $T_a(H)$ with pitch $a=a_{1,m}$ can be parameterized by
\begin{equation*}
	\left[\frac{\pi}{2},\frac{5\pi}{2}\right]\times\left[0,2\pi m\right]\to\EKT,
	\quad
	\left(\sigma,\theta\right)\mapsto\left(r(\sigma),\theta,h_a(\sigma)+a_{1,m}\theta\right),
\end{equation*}
where $r(\sigma)=r(\sigma;H,\Jtube(H))$ and $h_a(\sigma;H)=h_a(\sigma;H,\Jtube(H))$ are solutions of \eqref{ode_sigma}. The enclosed volume can be parameterized by
\begin{equation*}
	\left[r_-,r_+\right]\times\left[0,2\pi m\right]\times\left[-h_a(r),h_a(r)\right]\to\EKT,
	\quad
	\left(r,\theta,h\right)\mapsto\left(r,\theta,h+a_{1,m}\theta\right),
\end{equation*}
where $r_\pm=r_\pm(H)$ are the minimal and maximal radius, see \eqref{eq:radius_notation}, and $h_a(r)$ is the reparametrization of $h_a(\sigma)$ in terms of $r=r(\sigma)$. Area and volume of $T_{a_{n,m}}(H)$ can be derived from these parametrizations via integration as
\begin{equation*}
	\begin{aligned}
		\area(T_{a_{n,m}}(H))
		&=4\pi m\int\limits_{\frac{\pi}{2}}^{\frac{3\pi}{2}} \left[\left(\sns(r)+\left(4\tau\sns\!\left(\tfrac{r}{2}\right)-a\right)^2\right)\left(\frac{dr}{d\sigma}\right)^2 +\sns(r)\left(\frac{dh}{d\sigma}\right)^2\right]^{1/2}\,d\sigma, \\[0.5em]
		\vol(T_{a_{n,m}}(H))
		&=4\pi m\int\limits_{r_-}^{r_+} \sn(r)\,h(r)\,dr.
	\end{aligned}
\end{equation*}

Note that tubes in $\Berger$ also enclose the complementary volume
\begin{equation*}
    \widetilde{\vol}(T_a(H))=\vol(\Berger)-\vol(T_a(H)).
\end{equation*}

For numerical evaluation we may fix $\kappa=4$ after rescaling. We only present the isoperimetric profile of tubes with (non-horizontal) pitch $a=a_{1,1}$, because tubes with pitch $a=a_{1,m}$ for $m\geq3$ have larger area for given volume. By Lemma~\ref{lem:closing_condition_berger} the tubes with pitch $a=a_{1,1}$ only exist for $\tau^2\in\left(0,\frac{1}{2}\right)$. The isoperimetric profiles are shown in Figure~\ref{fig:isoperimetric_profile} for different values of $\tau$. The data for the horizontal CMC tubes were computed using the formulas given in \cite{manzano24}. The data for the CMC spheres and CMC tori were computed using the formulas given in \cite{torralbo10}.

Although the tubes $T_{a_{1,1}}(H)$ have a better isoperimetric profile than the horizontal tubes, they still have larger area for given volume than CMC spheres or CMC Hopf tori.

\newpage

\begin{figure}[H]
	\centering
	\tiny
	\def\svgwidth{0.48\textwidth}\executeiffilenewer{isoperimetric_profile_tau0.2.svg}{isoperimetric_profile_tau0.2.pdf}{inkscape -z -D --file=isoperimetric_profile_tau0.2.svg 	--export-pdf=isoperimetric_profile_tau0.2.pdf --export-latex}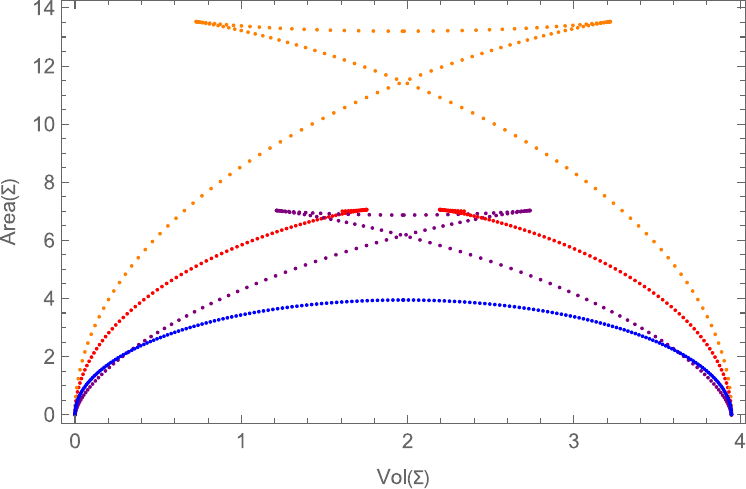
	\hfill
	\def\svgwidth{0.48\textwidth}\executeiffilenewer{isoperimetric_profile_tau0.25.svg}{isoperimetric_profile_tau0.25.pdf}{inkscape -z -D --file=isoperimetric_profile_tau0.25.svg 	--export-pdf=isoperimetric_profile_tau0.25.pdf --export-latex}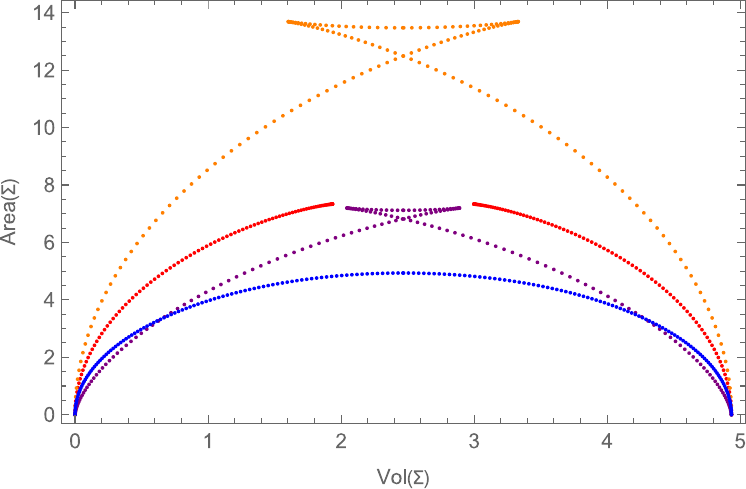
	\\[5mm]
	\def\svgwidth{0.48\textwidth}\executeiffilenewer{isoperimetric_profile_tau0.32.svg}{isoperimetric_profile_tau0.32.pdf}{inkscape -z -D --file=isoperimetric_profile_tau0.32.svg 	--export-pdf=isoperimetric_profile_tau0.32.pdf --export-latex}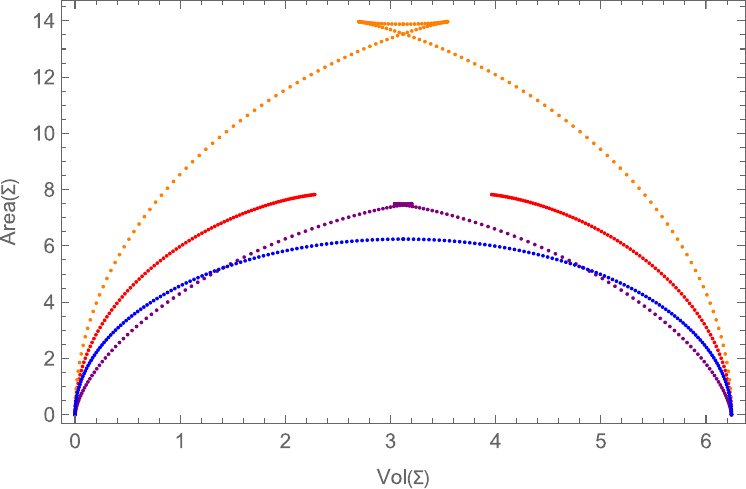
	\hfill
	\def\svgwidth{0.48\textwidth}\executeiffilenewer{isoperimetric_profile_tau0.45.svg}{isoperimetric_profile_tau0.45.pdf}{inkscape -z -D --file=isoperimetric_profile_tau0.45.svg 	--export-pdf=isoperimetric_profile_tau0.45.pdf --export-latex}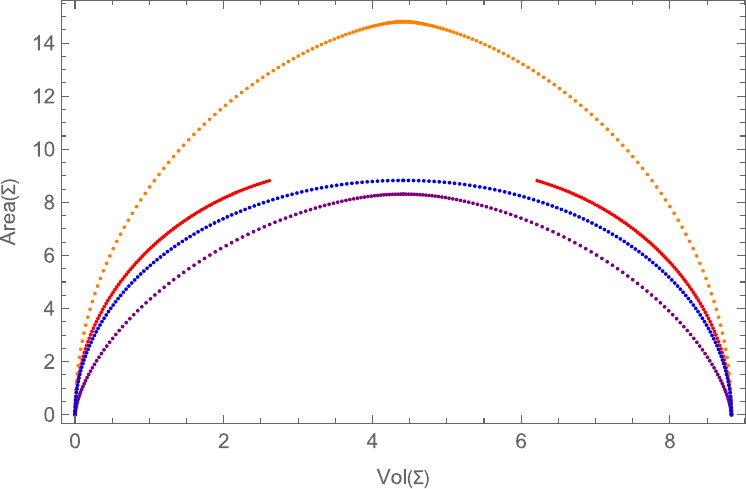
	\\[5mm]
	\def\svgwidth{0.48\textwidth}\executeiffilenewer{isoperimetric_profile_tau0.71.svg}{isoperimetric_profile_tau0.71.pdf}{inkscape -z -D --file=isoperimetric_profile_tau0.71.svg 	--export-pdf=isoperimetric_profile_tau0.71.pdf --export-latex}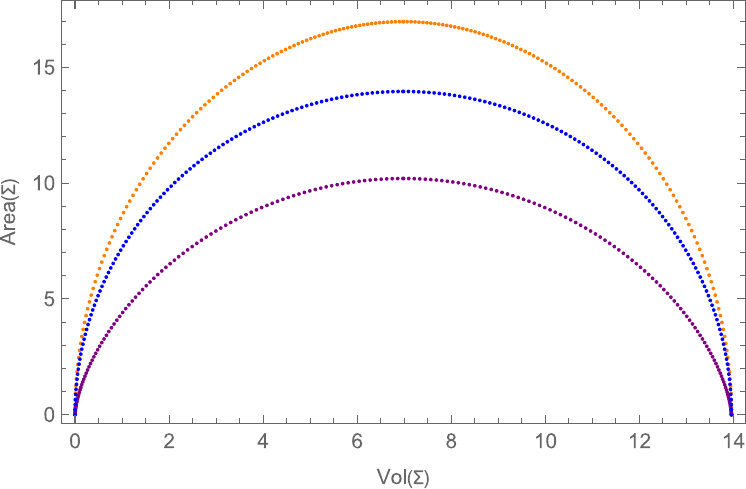
	\hfill
	\def\svgwidth{0.48\textwidth}\executeiffilenewer{isoperimetric_profile_tau1.41.svg}{isoperimetric_profile_tau1.41.pdf}{inkscape -z -D --file=isoperimetric_profile_tau1.41.svg 	--export-pdf=isoperimetric_profile_tau1.41.pdf --export-latex}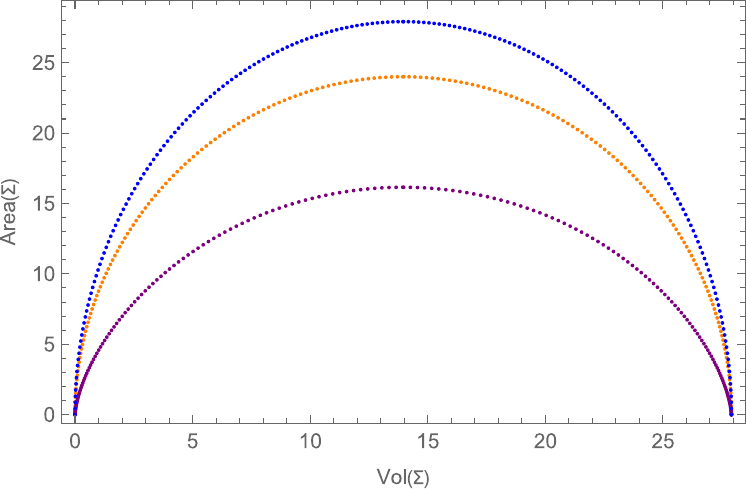
	\caption{Numerical plots of the isoperimetric profile $H\mapsto\left(\vol(\Sigma),\area(\Sigma)\right)$. Shown are CMC spheres (purple), CMC Hopf tori (blue), CMC tubes with horizontal pitch $a_{1,2}=\frac{2\tau}{\kappa}$ (orange), and CMC tubes with non-horizontal pitch $a_{1,1}=\frac{4\tau}{\kappa}$ (red). The ambient space is the Berger sphere $\E(4,\tau)$, where $\tau$ is as indicated. Notice that the tubes with $a=a_{1,1}$ do not exist for all volumes, see broken red curves in the first and second row.}
	\label{fig:isoperimetric_profile}
\end{figure}

\newpage
\printbibliography		

\end{document}

%% file: curve_family.pdf_tex
\begingroup%
  \makeatletter%
  \providecommand\color[2][]{%
    \errmessage{(Inkscape) Color is used for the text in Inkscape, but the package 'color.sty' is not loaded}%
    \renewcommand\color[2][]{}%
  }%
  \providecommand\transparent[1]{%
    \errmessage{(Inkscape) Transparency is used (non-zero) for the text in Inkscape, but the package 'transparent.sty' is not loaded}%
    \renewcommand\transparent[1]{}%
  }%
  \providecommand\rotatebox[2]{#2}%
  \newcommand*\fsize{\dimexpr\f@size pt\relax}%
  \newcommand*\lineheight[1]{\fontsize{\fsize}{#1\fsize}\selectfont}%
  \ifx\svgwidth\undefined%
    \setlength{\unitlength}{1890.91755396bp}%
    \ifx\svgscale\undefined%
      \relax%
    \else%
      \setlength{\unitlength}{\unitlength * \real{\svgscale}}%
    \fi%
  \else%
    \setlength{\unitlength}{\svgwidth}%
  \fi%
  \global\let\svgwidth\undefined%
  \global\let\svgscale\undefined%
  \makeatother%
  \begin{picture}(1,0.27085084)%
    \lineheight{1}%
    \setlength\tabcolsep{0pt}%
    \put(0,0){\includegraphics[width=\unitlength,page=1]{curve_family.pdf}}%
    \put(0.13943631,0.11709502){\makebox(0,0)[lt]{\lineheight{1.25}\smash{\begin{tabular}[t]{l}\textbf{$r$}\end{tabular}}}}%
    \put(0.0126061,0.25637961){\makebox(0,0)[lt]{\lineheight{1.25}\smash{\begin{tabular}[t]{l}\textbf{$h$}\end{tabular}}}}%
    \put(0,0){\includegraphics[width=\unitlength,page=2]{curve_family.pdf}}%
    \put(0.30898001,0.11709502){\makebox(0,0)[lt]{\lineheight{1.25}\smash{\begin{tabular}[t]{l}\textbf{$r$}\end{tabular}}}}%
    \put(0.18214981,0.25637961){\makebox(0,0)[lt]{\lineheight{1.25}\smash{\begin{tabular}[t]{l}\textbf{$h$}\end{tabular}}}}%
    \put(0,0){\includegraphics[width=\unitlength,page=3]{curve_family.pdf}}%
    \put(0.47852369,0.11709502){\makebox(0,0)[lt]{\lineheight{1.25}\smash{\begin{tabular}[t]{l}\textbf{$r$}\end{tabular}}}}%
    \put(0.35169349,0.25637961){\makebox(0,0)[lt]{\lineheight{1.25}\smash{\begin{tabular}[t]{l}\textbf{$h$}\end{tabular}}}}%
    \put(0,0){\includegraphics[width=\unitlength,page=4]{curve_family.pdf}}%
    \put(0.6480674,0.11709502){\makebox(0,0)[lt]{\lineheight{1.25}\smash{\begin{tabular}[t]{l}\textbf{$r$}\end{tabular}}}}%
    \put(0.52123718,0.25637961){\makebox(0,0)[lt]{\lineheight{1.25}\smash{\begin{tabular}[t]{l}\textbf{$h$}\end{tabular}}}}%
    \put(0,0){\includegraphics[width=\unitlength,page=5]{curve_family.pdf}}%
    \put(0.81761105,0.11709502){\makebox(0,0)[lt]{\lineheight{1.25}\smash{\begin{tabular}[t]{l}\textbf{$r$}\end{tabular}}}}%
    \put(0.69078089,0.25637961){\makebox(0,0)[lt]{\lineheight{1.25}\smash{\begin{tabular}[t]{l}\textbf{$h$}\end{tabular}}}}%
    \put(0,0){\includegraphics[width=\unitlength,page=6]{curve_family.pdf}}%
    \put(0.98715477,0.11709502){\makebox(0,0)[lt]{\lineheight{1.25}\smash{\begin{tabular}[t]{l}\textbf{$r$}\end{tabular}}}}%
    \put(0.8603246,0.25637961){\makebox(0,0)[lt]{\lineheight{1.25}\smash{\begin{tabular}[t]{l}\textbf{$h$}\end{tabular}}}}%
    \put(0,0){\includegraphics[width=\unitlength,page=7]{curve_family.pdf}}%
  \end{picture}%
\endgroup%

%% file: moduli_space_poscurv.pdf_tex
\begingroup%
  \makeatletter%
  \providecommand\color[2][]{%
    \errmessage{(Inkscape) Color is used for the text in Inkscape, but the package 'color.sty' is not loaded}%
    \renewcommand\color[2][]{}%
  }%
  \providecommand\transparent[1]{%
    \errmessage{(Inkscape) Transparency is used (non-zero) for the text in Inkscape, but the package 'transparent.sty' is not loaded}%
    \renewcommand\transparent[1]{}%
  }%
  \providecommand\rotatebox[2]{#2}%
  \newcommand*\fsize{\dimexpr\f@size pt\relax}%
  \newcommand*\lineheight[1]{\fontsize{\fsize}{#1\fsize}\selectfont}%
  \ifx\svgwidth\undefined%
    \setlength{\unitlength}{415.3458252bp}%
    \ifx\svgscale\undefined%
      \relax%
    \else%
      \setlength{\unitlength}{\unitlength * \real{\svgscale}}%
    \fi%
  \else%
    \setlength{\unitlength}{\svgwidth}%
  \fi%
  \global\let\svgwidth\undefined%
  \global\let\svgscale\undefined%
  \makeatother%
  \begin{picture}(1,0.96664028)%
    \lineheight{1}%
    \setlength\tabcolsep{0pt}%
    \put(0,0){\includegraphics[width=\unitlength,page=1]{moduli_space_poscurv.pdf}}%
    \put(0.03642593,0.93590471){\makebox(0,0)[lt]{\lineheight{1.25}\smash{\begin{tabular}[t]{l}\textbf{$J$}\end{tabular}}}}%
    \put(0.95672621,0.57475011){\makebox(0,0)[lt]{\lineheight{1.25}\smash{\begin{tabular}[t]{l}\textbf{$H$}\end{tabular}}}}%
    \put(0.46410203,0.74724434){\makebox(0,0)[lt]{\lineheight{1.25}\smash{\begin{tabular}[t]{l}\textbf{$\text{vertical cylinder}$}\end{tabular}}}}%
    \put(0.13723308,0.82961483){\makebox(0,0)[lt]{\lineheight{1.25}\smash{\begin{tabular}[t]{l}\textbf{$J_\mathrm{max}(H)$}\end{tabular}}}}%
    \put(0.81516334,0.50895164){\makebox(0,0)[lt]{\lineheight{1.25}\smash{\begin{tabular}[t]{l}\textbf{$J=0$}\end{tabular}}}}%
    \put(0.4845786,0.47880407){\makebox(0,0)[lt]{\lineheight{1.25}\smash{\begin{tabular}[t]{l}\textbf{$\text{sphere type}$}\end{tabular}}}}%
    \put(0.05058299,0.642235){\makebox(0,0)[lt]{\lineheight{1.25}\smash{\begin{tabular}[t]{l}\textbf{$\text{unduloid type}$}\end{tabular}}}}%
    \put(0.37867052,0.12289455){\makebox(0,0)[lt]{\lineheight{1.25}\smash{\begin{tabular}[t]{l}\textbf{$\Xi_a^-$}\end{tabular}}}}%
    \put(0.71420629,0.28057426){\makebox(0,0)[lt]{\lineheight{1.25}\smash{\begin{tabular}[t]{l}\textbf{$\Xi_a^+$}\end{tabular}}}}%
    \put(0.02044857,0.02209378){\makebox(0,0)[lt]{\lineheight{1.25}\smash{\begin{tabular}[t]{l}\textbf{$H=0$}\end{tabular}}}}%
    \put(0.61672711,0.37219178){\makebox(0,0)[lt]{\lineheight{1.25}\smash{\begin{tabular}[t]{l}\textbf{$\text{nodoid type}$}\end{tabular}}}}%
    \put(0.0493098,0.31875418){\makebox(0,0)[lt]{\lineheight{1.25}\smash{\begin{tabular}[t]{l}\textbf{$\rotatebox{-48}{\text{unduloid type}}$}\end{tabular}}}}%
    \put(0.77257453,0.11171464){\makebox(0,0)[lt]{\lineheight{1.25}\smash{\begin{tabular}[t]{l}\textbf{$\Jtwo(H)$}\end{tabular}}}}%
    \put(0.50344866,0.02392067){\makebox(0,0)[lt]{\lineheight{1.25}\smash{\begin{tabular}[t]{l}\textbf{$\Jone(H)$}\end{tabular}}}}%
    \put(0,0){\includegraphics[width=\unitlength,page=2]{moduli_space_poscurv.pdf}}%
    \put(0.34351538,0.57418965){\makebox(0,0)[lt]{\lineheight{1.25}\smash{\begin{tabular}[t]{l}\textbf{$\Hex$}\end{tabular}}}}%
    \put(0.29730838,0.40891226){\makebox(0,0)[lt]{\lineheight{1.25}\smash{\begin{tabular}[t]{l}\textbf{$\Xi_a^0$}\end{tabular}}}}%
  \end{picture}%
\endgroup%

%% file: moduli_space_negcurv.pdf_tex
\begingroup%
  \makeatletter%
  \providecommand\color[2][]{%
    \errmessage{(Inkscape) Color is used for the text in Inkscape, but the package 'color.sty' is not loaded}%
    \renewcommand\color[2][]{}%
  }%
  \providecommand\transparent[1]{%
    \errmessage{(Inkscape) Transparency is used (non-zero) for the text in Inkscape, but the package 'transparent.sty' is not loaded}%
    \renewcommand\transparent[1]{}%
  }%
  \providecommand\rotatebox[2]{#2}%
  \newcommand*\fsize{\dimexpr\f@size pt\relax}%
  \newcommand*\lineheight[1]{\fontsize{\fsize}{#1\fsize}\selectfont}%
  \ifx\svgwidth\undefined%
    \setlength{\unitlength}{415.20785522bp}%
    \ifx\svgscale\undefined%
      \relax%
    \else%
      \setlength{\unitlength}{\unitlength * \real{\svgscale}}%
    \fi%
  \else%
    \setlength{\unitlength}{\svgwidth}%
  \fi%
  \global\let\svgwidth\undefined%
  \global\let\svgscale\undefined%
  \makeatother%
  \begin{picture}(1,0.96701827)%
    \lineheight{1}%
    \setlength\tabcolsep{0pt}%
    \put(0,0){\includegraphics[width=\unitlength,page=1]{moduli_space_negcurv.pdf}}%
    \put(0.03610577,0.93634088){\makebox(0,0)[lt]{\lineheight{1.25}\smash{\begin{tabular}[t]{l}\textbf{$J$}\end{tabular}}}}%
    \put(0.95671183,0.57506631){\makebox(0,0)[lt]{\lineheight{1.25}\smash{\begin{tabular}[t]{l}\textbf{$H$}\end{tabular}}}}%
    \put(0.46392401,0.73316722){\makebox(0,0)[lt]{\lineheight{1.25}\smash{\begin{tabular}[t]{l}\textbf{$\text{vertical cylinder}$}\end{tabular}}}}%
    \put(0.12249577,0.79027653){\makebox(0,0)[lt]{\lineheight{1.25}\smash{\begin{tabular}[t]{l}\textbf{$J_\mathrm{max}(H)$}\end{tabular}}}}%
    \put(0.8151017,0.50924595){\makebox(0,0)[lt]{\lineheight{1.25}\smash{\begin{tabular}[t]{l}\textbf{$J=0$}\end{tabular}}}}%
    \put(0.48440733,0.47908836){\makebox(0,0)[lt]{\lineheight{1.25}\smash{\begin{tabular}[t]{l}\textbf{$\text{sphere type}$}\end{tabular}}}}%
    \put(0.05026751,0.628123){\makebox(0,0)[lt]{\lineheight{1.25}\smash{\begin{tabular}[t]{l}\textbf{$\text{unduloid type}$}\end{tabular}}}}%
    \put(0.26951241,0.13424426){\makebox(0,0)[lt]{\lineheight{1.25}\smash{\begin{tabular}[t]{l}\textbf{$\Xi_a^-$}\end{tabular}}}}%
    \put(0.70079812,0.33016892){\makebox(0,0)[lt]{\lineheight{1.25}\smash{\begin{tabular}[t]{l}\textbf{$\Xi_a^+$}\end{tabular}}}}%
    \put(0.02012304,0.02222625){\makebox(0,0)[lt]{\lineheight{1.25}\smash{\begin{tabular}[t]{l}\textbf{$H=\frac{\sqrt{-\kappa}}{2}$}\end{tabular}}}}%
    \put(0.04987789,0.22381602){\makebox(0,0)[lt]{\lineheight{1.25}\smash{\begin{tabular}[t]{l}\textbf{$\text{nodoid type}$}\end{tabular}}}}%
    \put(0.77249896,0.11187695){\makebox(0,0)[lt]{\lineheight{1.25}\smash{\begin{tabular}[t]{l}\textbf{$\Jtwo(H)$}\end{tabular}}}}%
    \put(0.49244572,0.02405375){\makebox(0,0)[lt]{\lineheight{1.25}\smash{\begin{tabular}[t]{l}\textbf{$\Jone(H)$}\end{tabular}}}}%
    \put(0,0){\includegraphics[width=\unitlength,page=2]{moduli_space_negcurv.pdf}}%
    \put(0.34329721,0.57450564){\makebox(0,0)[lt]{\lineheight{1.25}\smash{\begin{tabular}[t]{l}\textbf{$\Hex$}\end{tabular}}}}%
    \put(0.27501874,0.40324926){\makebox(0,0)[lt]{\lineheight{1.25}\smash{\begin{tabular}[t]{l}\textbf{$\Xi_a^0$}\end{tabular}}}}%
  \end{picture}%
\endgroup%

%% file: curve_notation.pdf_tex
\begingroup%
  \makeatletter%
  \providecommand\color[2][]{%
    \errmessage{(Inkscape) Color is used for the text in Inkscape, but the package 'color.sty' is not loaded}%
    \renewcommand\color[2][]{}%
  }%
  \providecommand\transparent[1]{%
    \errmessage{(Inkscape) Transparency is used (non-zero) for the text in Inkscape, but the package 'transparent.sty' is not loaded}%
    \renewcommand\transparent[1]{}%
  }%
  \providecommand\rotatebox[2]{#2}%
  \newcommand*\fsize{\dimexpr\f@size pt\relax}%
  \newcommand*\lineheight[1]{\fontsize{\fsize}{#1\fsize}\selectfont}%
  \ifx\svgwidth\undefined%
    \setlength{\unitlength}{429.28015137bp}%
    \ifx\svgscale\undefined%
      \relax%
    \else%
      \setlength{\unitlength}{\unitlength * \real{\svgscale}}%
    \fi%
  \else%
    \setlength{\unitlength}{\svgwidth}%
  \fi%
  \global\let\svgwidth\undefined%
  \global\let\svgscale\undefined%
  \makeatother%
  \begin{picture}(1,0.23015718)%
    \lineheight{1}%
    \setlength\tabcolsep{0pt}%
    \put(0,0){\includegraphics[width=\unitlength,page=1]{curve_notation.pdf}}%
    \put(0.27215454,0.08850571){\makebox(0,0)[lt]{\lineheight{1.25}\smash{\begin{tabular}[t]{l}\textbf{$r$}\end{tabular}}}}%
    \put(0.01746145,0.2109711){\makebox(0,0)[lt]{\lineheight{1.25}\smash{\begin{tabular}[t]{l}\textbf{$h$}\end{tabular}}}}%
    \put(0.21832608,0.00311467){\makebox(0,0)[lt]{\lineheight{1.25}\smash{\begin{tabular}[t]{l}\textbf{$r_+$}\end{tabular}}}}%
    \put(0.05539167,0.00313088){\makebox(0,0)[lt]{\lineheight{1.25}\smash{\begin{tabular}[t]{l}\textbf{$r_-$}\end{tabular}}}}%
    \put(0,0){\includegraphics[width=\unitlength,page=2]{curve_notation.pdf}}%
    \put(0.21951986,0.12545547){\makebox(0,0)[lt]{\lineheight{1.25}\smash{\begin{tabular}[t]{l}\textbf{$h\!\left(\frac{\pi}{2}\right)$}\end{tabular}}}}%
    \put(0.02285147,0.13628396){\makebox(0,0)[lt]{\lineheight{1.25}\smash{\begin{tabular}[t]{l}\textbf{$h\!\left(\frac{3\pi}{2}\right)$}\end{tabular}}}}%
    \put(0.11125855,0.12225027){\makebox(0,0)[lt]{\lineheight{1.25}\smash{\begin{tabular}[t]{l}\textbf{$\delta_a\!>\!0$}\end{tabular}}}}%
    \put(0,0){\includegraphics[width=\unitlength,page=3]{curve_notation.pdf}}%
    \put(0.62871385,0.08850574){\makebox(0,0)[lt]{\lineheight{1.25}\smash{\begin{tabular}[t]{l}\textbf{$r$}\end{tabular}}}}%
    \put(0.37402077,0.2109711){\makebox(0,0)[lt]{\lineheight{1.25}\smash{\begin{tabular}[t]{l}\textbf{$h$}\end{tabular}}}}%
    \put(0.57488541,0.00311467){\makebox(0,0)[lt]{\lineheight{1.25}\smash{\begin{tabular}[t]{l}\textbf{$r_+$}\end{tabular}}}}%
    \put(0.41195101,0.00313088){\makebox(0,0)[lt]{\lineheight{1.25}\smash{\begin{tabular}[t]{l}\textbf{$r_-$}\end{tabular}}}}%
    \put(0,0){\includegraphics[width=\unitlength,page=4]{curve_notation.pdf}}%
    \put(0.98527318,0.08850571){\makebox(0,0)[lt]{\lineheight{1.25}\smash{\begin{tabular}[t]{l}\textbf{$r$}\end{tabular}}}}%
    \put(0.7305801,0.2109711){\makebox(0,0)[lt]{\lineheight{1.25}\smash{\begin{tabular}[t]{l}\textbf{$h$}\end{tabular}}}}%
    \put(0.93144474,0.00311467){\makebox(0,0)[lt]{\lineheight{1.25}\smash{\begin{tabular}[t]{l}\textbf{$r_+$}\end{tabular}}}}%
    \put(0.7685103,0.00313088){\makebox(0,0)[lt]{\lineheight{1.25}\smash{\begin{tabular}[t]{l}\textbf{$r_-$}\end{tabular}}}}%
    \put(0,0){\includegraphics[width=\unitlength,page=5]{curve_notation.pdf}}%
    \put(0.57561494,0.12560526){\makebox(0,0)[lt]{\lineheight{1.25}\smash{\begin{tabular}[t]{l}\textbf{$h\!\left(\frac{\pi}{2}\right)$}\end{tabular}}}}%
    \put(0,0){\includegraphics[width=\unitlength,page=6]{curve_notation.pdf}}%
    \put(0.38244075,0.12549653){\makebox(0,0)[lt]{\lineheight{1.25}\smash{\begin{tabular}[t]{l}\textbf{$h\!\left(\frac{3\pi}{2}\right)$}\end{tabular}}}}%
    \put(0,0){\includegraphics[width=\unitlength,page=7]{curve_notation.pdf}}%
    \put(0.93285517,0.12598744){\makebox(0,0)[lt]{\lineheight{1.25}\smash{\begin{tabular}[t]{l}\textbf{$h\!\left(\frac{\pi}{2}\right)$}\end{tabular}}}}%
    \put(0.73618676,0.07741416){\makebox(0,0)[lt]{\lineheight{1.25}\smash{\begin{tabular}[t]{l}\textbf{$h\!\left(\frac{3\pi}{2}\right)$}\end{tabular}}}}%
    \put(0,0){\includegraphics[width=\unitlength,page=8]{curve_notation.pdf}}%
    \put(0.12789068,0.22121156){\makebox(0,0)[lt]{\lineheight{1.25}\smash{\begin{tabular}[t]{l}\textbf{$h(\pi)$}\end{tabular}}}}%
    \put(0,0){\includegraphics[width=\unitlength,page=9]{curve_notation.pdf}}%
    \put(0.48033811,0.21831464){\makebox(0,0)[lt]{\lineheight{1.25}\smash{\begin{tabular}[t]{l}\textbf{$h(\pi)$}\end{tabular}}}}%
    \put(0,0){\includegraphics[width=\unitlength,page=10]{curve_notation.pdf}}%
    \put(0.84046282,0.22168235){\makebox(0,0)[lt]{\lineheight{1.25}\smash{\begin{tabular}[t]{l}\textbf{$h(\pi)$}\end{tabular}}}}%
    \put(0,0){\includegraphics[width=\unitlength,page=11]{curve_notation.pdf}}%
    \put(0.47471531,0.12225027){\makebox(0,0)[lt]{\lineheight{1.25}\smash{\begin{tabular}[t]{l}\textbf{$\delta_a\!=\!0$}\end{tabular}}}}%
    \put(0.82850255,0.09172787){\makebox(0,0)[lt]{\lineheight{1.25}\smash{\begin{tabular}[t]{l}\textbf{$\delta_a\!<\!0$}\end{tabular}}}}%
    \put(0,0){\includegraphics[width=\unitlength,page=12]{curve_notation.pdf}}%
  \end{picture}%
\endgroup%

%% file: sphere_jump.pdf_tex
\begingroup%
  \makeatletter%
  \providecommand\color[2][]{%
    \errmessage{(Inkscape) Color is used for the text in Inkscape, but the package 'color.sty' is not loaded}%
    \renewcommand\color[2][]{}%
  }%
  \providecommand\transparent[1]{%
    \errmessage{(Inkscape) Transparency is used (non-zero) for the text in Inkscape, but the package 'transparent.sty' is not loaded}%
    \renewcommand\transparent[1]{}%
  }%
  \providecommand\rotatebox[2]{#2}%
  \newcommand*\fsize{\dimexpr\f@size pt\relax}%
  \newcommand*\lineheight[1]{\fontsize{\fsize}{#1\fsize}\selectfont}%
  \ifx\svgwidth\undefined%
    \setlength{\unitlength}{283.14688352bp}%
    \ifx\svgscale\undefined%
      \relax%
    \else%
      \setlength{\unitlength}{\unitlength * \real{\svgscale}}%
    \fi%
  \else%
    \setlength{\unitlength}{\svgwidth}%
  \fi%
  \global\let\svgwidth\undefined%
  \global\let\svgscale\undefined%
  \makeatother%
  \begin{picture}(1,0.37046577)%
    \lineheight{1}%
    \setlength\tabcolsep{0pt}%
    \put(0,0){\includegraphics[width=\unitlength,page=1]{sphere_jump.pdf}}%
    \put(0.3592739,0.14477919){\makebox(0,0)[lt]{\lineheight{1.25}\smash{\begin{tabular}[t]{l}\textbf{$r$}\end{tabular}}}}%
    \put(0.02614066,0.34637052){\makebox(0,0)[lt]{\lineheight{1.25}\smash{\begin{tabular}[t]{l}\textbf{$h$}\end{tabular}}}}%
    \put(0,0){\includegraphics[width=\unitlength,page=2]{sphere_jump.pdf}}%
    \put(0.28939761,0.19043071){\makebox(0,0)[lt]{\lineheight{1.25}\smash{\begin{tabular}[t]{l}\textbf{$h\!\left(\frac{\pi}{2}\right)$}\end{tabular}}}}%
    \put(0,0){\includegraphics[width=\unitlength,page=3]{sphere_jump.pdf}}%
    \put(0.01771585,0.19026586){\makebox(0,0)[lt]{\lineheight{1.25}\smash{\begin{tabular}[t]{l}\textbf{$h\!\left(\frac{3\pi}{2}\right)$}\end{tabular}}}}%
    \put(0,0){\includegraphics[width=\unitlength,page=4]{sphere_jump.pdf}}%
    \put(0.15024562,0.33098777){\makebox(0,0)[lt]{\lineheight{1.25}\smash{\begin{tabular}[t]{l}\textbf{$h(\pi)$}\end{tabular}}}}%
    \put(0,0){\includegraphics[width=\unitlength,page=5]{sphere_jump.pdf}}%
    \put(0.97857208,0.14477919){\makebox(0,0)[lt]{\lineheight{1.25}\smash{\begin{tabular}[t]{l}\textbf{$r$}\end{tabular}}}}%
    \put(0.77095391,0.34634222){\makebox(0,0)[lt]{\lineheight{1.25}\smash{\begin{tabular}[t]{l}\textbf{$h$}\end{tabular}}}}%
    \put(0.89673707,0.19043071){\makebox(0,0)[lt]{\lineheight{1.25}\smash{\begin{tabular}[t]{l}\textbf{$h\!\left(\frac{\pi}{2}\right)$}\end{tabular}}}}%
    \put(0.66874993,0.16907545){\makebox(0,0)[lt]{\lineheight{1.25}\smash{\begin{tabular}[t]{l}\textbf{$h\!\left(\frac{3\pi}{2}\right)$}\end{tabular}}}}%
    \put(0,0){\includegraphics[width=\unitlength,page=6]{sphere_jump.pdf}}%
    \put(0.68690839,0.30931203){\makebox(0,0)[lt]{\lineheight{1.25}\smash{\begin{tabular}[t]{l}\textbf{$h(\pi)$}\end{tabular}}}}%
    \put(0.48656666,0.19067904){\color[rgb]{0,0,0}\makebox(0,0)[lt]{\lineheight{1.25}\smash{\begin{tabular}[t]{l}$\text{\footnotesize $n\to\infty$}$\end{tabular}}}}%
    \put(0,0){\includegraphics[width=\unitlength,page=7]{sphere_jump.pdf}}%
    \put(0.66692991,0.23958148){\makebox(0,0)[lt]{\lineheight{1.25}\smash{\begin{tabular}[t]{l}\textbf{$\frac{\pi}{2}\abs{a}$}\end{tabular}}}}%
    \put(0.2527961,0.30754691){\makebox(0,0)[lt]{\lineheight{1.25}\smash{\begin{tabular}[t]{l}\textbf{$\gamma_a(H_n,J_n)$}\end{tabular}}}}%
    \put(0.86119979,0.3076479){\makebox(0,0)[lt]{\lineheight{1.25}\smash{\begin{tabular}[t]{l}\textbf{$\gamma_a(H_0,0)$}\end{tabular}}}}%
  \end{picture}%
\endgroup%

%% file: plot_H0_PSL2.pdf_tex
\begingroup%
  \makeatletter%
  \providecommand\color[2][]{%
    \errmessage{(Inkscape) Color is used for the text in Inkscape, but the package 'color.sty' is not loaded}%
    \renewcommand\color[2][]{}%
  }%
  \providecommand\transparent[1]{%
    \errmessage{(Inkscape) Transparency is used (non-zero) for the text in Inkscape, but the package 'transparent.sty' is not loaded}%
    \renewcommand\transparent[1]{}%
  }%
  \providecommand\rotatebox[2]{#2}%
  \newcommand*\fsize{\dimexpr\f@size pt\relax}%
  \newcommand*\lineheight[1]{\fontsize{\fsize}{#1\fsize}\selectfont}%
  \ifx\svgwidth\undefined%
    \setlength{\unitlength}{354.96710205bp}%
    \ifx\svgscale\undefined%
      \relax%
    \else%
      \setlength{\unitlength}{\unitlength * \real{\svgscale}}%
    \fi%
  \else%
    \setlength{\unitlength}{\svgwidth}%
  \fi%
  \global\let\svgwidth\undefined%
  \global\let\svgscale\undefined%
  \makeatother%
  \begin{picture}(1,0.65697887)%
    \lineheight{1}%
    \setlength\tabcolsep{0pt}%
    \put(0,0){\includegraphics[width=\unitlength,page=1]{plot_H0_PSL2.pdf}}%
    \put(0.52411413,0.00290249){\makebox(0,0)[lt]{\lineheight{1.25}\smash{\begin{tabular}[t]{l}\textbf{$a$}\end{tabular}}}}%
    \put(-0.00137556,0.34722622){\makebox(0,0)[lt]{\lineheight{1.25}\smash{\begin{tabular}[t]{l}\textbf{$\rotatebox{90}{$\Hlimit(a)$}$}\end{tabular}}}}%
    \put(0,0){\includegraphics[width=\unitlength,page=2]{plot_H0_PSL2.pdf}}%
  \end{picture}%
\endgroup%

%% file: plot_H0_Berger.pdf_tex
\begingroup%
  \makeatletter%
  \providecommand\color[2][]{%
    \errmessage{(Inkscape) Color is used for the text in Inkscape, but the package 'color.sty' is not loaded}%
    \renewcommand\color[2][]{}%
  }%
  \providecommand\transparent[1]{%
    \errmessage{(Inkscape) Transparency is used (non-zero) for the text in Inkscape, but the package 'transparent.sty' is not loaded}%
    \renewcommand\transparent[1]{}%
  }%
  \providecommand\rotatebox[2]{#2}%
  \newcommand*\fsize{\dimexpr\f@size pt\relax}%
  \newcommand*\lineheight[1]{\fontsize{\fsize}{#1\fsize}\selectfont}%
  \ifx\svgwidth\undefined%
    \setlength{\unitlength}{354.29055786bp}%
    \ifx\svgscale\undefined%
      \relax%
    \else%
      \setlength{\unitlength}{\unitlength * \real{\svgscale}}%
    \fi%
  \else%
    \setlength{\unitlength}{\svgwidth}%
  \fi%
  \global\let\svgwidth\undefined%
  \global\let\svgscale\undefined%
  \makeatother%
  \begin{picture}(1,0.65779623)%
    \lineheight{1}%
    \setlength\tabcolsep{0pt}%
    \put(0,0){\includegraphics[width=\unitlength,page=1]{plot_H0_Berger.pdf}}%
    \put(-0.00137819,0.34643644){\makebox(0,0)[lt]{\lineheight{1.25}\smash{\begin{tabular}[t]{l}\textbf{$\rotatebox{90}{$\Hlimit(a)$}$}\end{tabular}}}}%
    \put(0.52255394,0.00290799){\makebox(0,0)[lt]{\lineheight{1.25}\smash{\begin{tabular}[t]{l}\textbf{$a$}\end{tabular}}}}%
    \put(0,0){\includegraphics[width=\unitlength,page=2]{plot_H0_Berger.pdf}}%
  \end{picture}%
\endgroup%

%% file: moduli_space_tube_1.pdf_tex
\begingroup%
  \makeatletter%
  \providecommand\color[2][]{%
    \errmessage{(Inkscape) Color is used for the text in Inkscape, but the package 'color.sty' is not loaded}%
    \renewcommand\color[2][]{}%
  }%
  \providecommand\transparent[1]{%
    \errmessage{(Inkscape) Transparency is used (non-zero) for the text in Inkscape, but the package 'transparent.sty' is not loaded}%
    \renewcommand\transparent[1]{}%
  }%
  \providecommand\rotatebox[2]{#2}%
  \newcommand*\fsize{\dimexpr\f@size pt\relax}%
  \newcommand*\lineheight[1]{\fontsize{\fsize}{#1\fsize}\selectfont}%
  \ifx\svgwidth\undefined%
    \setlength{\unitlength}{157.87544632bp}%
    \ifx\svgscale\undefined%
      \relax%
    \else%
      \setlength{\unitlength}{\unitlength * \real{\svgscale}}%
    \fi%
  \else%
    \setlength{\unitlength}{\svgwidth}%
  \fi%
  \global\let\svgwidth\undefined%
  \global\let\svgscale\undefined%
  \makeatother%
  \begin{picture}(1,0.85554491)%
    \lineheight{1}%
    \setlength\tabcolsep{0pt}%
    \put(0,0){\includegraphics[width=\unitlength,page=1]{moduli_space_tube_1.pdf}}%
    \put(0.04397846,0.82926369){\makebox(0,0)[lt]{\lineheight{1.25}\smash{\begin{tabular}[t]{l}\textbf{$J$}\end{tabular}}}}%
    \put(0.95446131,0.72116233){\makebox(0,0)[lt]{\lineheight{1.25}\smash{\begin{tabular}[t]{l}\textbf{$H$}\end{tabular}}}}%
    \put(0.70921694,0.47071484){\makebox(0,0)[lt]{\lineheight{1.25}\smash{\begin{tabular}[t]{l}\textbf{$\Xi_a^+$}\end{tabular}}}}%
    \put(0.33643784,0.08506682){\makebox(0,0)[lt]{\lineheight{1.25}\smash{\begin{tabular}[t]{l}\textbf{$\Xi_a^-$}\end{tabular}}}}%
    \put(0,0){\includegraphics[width=\unitlength,page=2]{moduli_space_tube_1.pdf}}%
    \put(0.56391215,0.05181275){\makebox(0,0)[lt]{\lineheight{1.25}\smash{\begin{tabular}[t]{l}\textbf{$\Xi_a^0$}\end{tabular}}}}%
    \put(0,0){\includegraphics[width=\unitlength,page=3]{moduli_space_tube_1.pdf}}%
    \put(0.46137616,0.71792769){\makebox(0,0)[lt]{\lineheight{1.25}\smash{\begin{tabular}[t]{l}\textbf{$\Hex$}\end{tabular}}}}%
  \end{picture}%
\endgroup%

%% file: moduli_space_tube_3.pdf_tex
\begingroup%
  \makeatletter%
  \providecommand\color[2][]{%
    \errmessage{(Inkscape) Color is used for the text in Inkscape, but the package 'color.sty' is not loaded}%
    \renewcommand\color[2][]{}%
  }%
  \providecommand\transparent[1]{%
    \errmessage{(Inkscape) Transparency is used (non-zero) for the text in Inkscape, but the package 'transparent.sty' is not loaded}%
    \renewcommand\transparent[1]{}%
  }%
  \providecommand\rotatebox[2]{#2}%
  \newcommand*\fsize{\dimexpr\f@size pt\relax}%
  \newcommand*\lineheight[1]{\fontsize{\fsize}{#1\fsize}\selectfont}%
  \ifx\svgwidth\undefined%
    \setlength{\unitlength}{157.87544632bp}%
    \ifx\svgscale\undefined%
      \relax%
    \else%
      \setlength{\unitlength}{\unitlength * \real{\svgscale}}%
    \fi%
  \else%
    \setlength{\unitlength}{\svgwidth}%
  \fi%
  \global\let\svgwidth\undefined%
  \global\let\svgscale\undefined%
  \makeatother%
  \begin{picture}(1,0.85554491)%
    \lineheight{1}%
    \setlength\tabcolsep{0pt}%
    \put(0,0){\includegraphics[width=\unitlength,page=1]{moduli_space_tube_3.pdf}}%
    \put(0.04397846,0.82926369){\makebox(0,0)[lt]{\lineheight{1.25}\smash{\begin{tabular}[t]{l}\textbf{$J$}\end{tabular}}}}%
    \put(0.95446131,0.72116233){\makebox(0,0)[lt]{\lineheight{1.25}\smash{\begin{tabular}[t]{l}\textbf{$H$}\end{tabular}}}}%
    \put(0.70921694,0.47071484){\makebox(0,0)[lt]{\lineheight{1.25}\smash{\begin{tabular}[t]{l}\textbf{$\Xi_a^+$}\end{tabular}}}}%
    \put(0.33643784,0.08506682){\makebox(0,0)[lt]{\lineheight{1.25}\smash{\begin{tabular}[t]{l}\textbf{$\Xi_a^-$}\end{tabular}}}}%
    \put(0,0){\includegraphics[width=\unitlength,page=2]{moduli_space_tube_3.pdf}}%
    \put(0.56391215,0.05181275){\makebox(0,0)[lt]{\lineheight{1.25}\smash{\begin{tabular}[t]{l}\textbf{$\Xi_a^0$}\end{tabular}}}}%
    \put(0,0){\includegraphics[width=\unitlength,page=3]{moduli_space_tube_3.pdf}}%
    \put(0.46137616,0.71792769){\makebox(0,0)[lt]{\lineheight{1.25}\smash{\begin{tabular}[t]{l}\textbf{$\Hex$}\end{tabular}}}}%
  \end{picture}%
\endgroup%

%% file: moduli_space_tube_2.pdf_tex
\begingroup%
  \makeatletter%
  \providecommand\color[2][]{%
    \errmessage{(Inkscape) Color is used for the text in Inkscape, but the package 'color.sty' is not loaded}%
    \renewcommand\color[2][]{}%
  }%
  \providecommand\transparent[1]{%
    \errmessage{(Inkscape) Transparency is used (non-zero) for the text in Inkscape, but the package 'transparent.sty' is not loaded}%
    \renewcommand\transparent[1]{}%
  }%
  \providecommand\rotatebox[2]{#2}%
  \newcommand*\fsize{\dimexpr\f@size pt\relax}%
  \newcommand*\lineheight[1]{\fontsize{\fsize}{#1\fsize}\selectfont}%
  \ifx\svgwidth\undefined%
    \setlength{\unitlength}{157.87544632bp}%
    \ifx\svgscale\undefined%
      \relax%
    \else%
      \setlength{\unitlength}{\unitlength * \real{\svgscale}}%
    \fi%
  \else%
    \setlength{\unitlength}{\svgwidth}%
  \fi%
  \global\let\svgwidth\undefined%
  \global\let\svgscale\undefined%
  \makeatother%
  \begin{picture}(1,0.85554491)%
    \lineheight{1}%
    \setlength\tabcolsep{0pt}%
    \put(0,0){\includegraphics[width=\unitlength,page=1]{moduli_space_tube_2.pdf}}%
    \put(0.04397846,0.82926369){\makebox(0,0)[lt]{\lineheight{1.25}\smash{\begin{tabular}[t]{l}\textbf{$J$}\end{tabular}}}}%
    \put(0.95446131,0.72116233){\makebox(0,0)[lt]{\lineheight{1.25}\smash{\begin{tabular}[t]{l}\textbf{$H$}\end{tabular}}}}%
    \put(0.70921694,0.47071484){\makebox(0,0)[lt]{\lineheight{1.25}\smash{\begin{tabular}[t]{l}\textbf{$\Xi_a^+$}\end{tabular}}}}%
    \put(0.33643784,0.08506682){\makebox(0,0)[lt]{\lineheight{1.25}\smash{\begin{tabular}[t]{l}\textbf{$\Xi_a^-$}\end{tabular}}}}%
    \put(0,0){\includegraphics[width=\unitlength,page=2]{moduli_space_tube_2.pdf}}%
    \put(0.56391215,0.05181275){\makebox(0,0)[lt]{\lineheight{1.25}\smash{\begin{tabular}[t]{l}\textbf{$\Xi_a^0$}\end{tabular}}}}%
    \put(0,0){\includegraphics[width=\unitlength,page=3]{moduli_space_tube_2.pdf}}%
    \put(0.46137616,0.71792769){\makebox(0,0)[lt]{\lineheight{1.25}\smash{\begin{tabular}[t]{l}\textbf{$\Hex$}\end{tabular}}}}%
  \end{picture}%
\endgroup%

%% file: moduli_space_tube_0.pdf_tex
\begingroup%
  \makeatletter%
  \providecommand\color[2][]{%
    \errmessage{(Inkscape) Color is used for the text in Inkscape, but the package 'color.sty' is not loaded}%
    \renewcommand\color[2][]{}%
  }%
  \providecommand\transparent[1]{%
    \errmessage{(Inkscape) Transparency is used (non-zero) for the text in Inkscape, but the package 'transparent.sty' is not loaded}%
    \renewcommand\transparent[1]{}%
  }%
  \providecommand\rotatebox[2]{#2}%
  \newcommand*\fsize{\dimexpr\f@size pt\relax}%
  \newcommand*\lineheight[1]{\fontsize{\fsize}{#1\fsize}\selectfont}%
  \ifx\svgwidth\undefined%
    \setlength{\unitlength}{157.87544632bp}%
    \ifx\svgscale\undefined%
      \relax%
    \else%
      \setlength{\unitlength}{\unitlength * \real{\svgscale}}%
    \fi%
  \else%
    \setlength{\unitlength}{\svgwidth}%
  \fi%
  \global\let\svgwidth\undefined%
  \global\let\svgscale\undefined%
  \makeatother%
  \begin{picture}(1,0.85554491)%
    \lineheight{1}%
    \setlength\tabcolsep{0pt}%
    \put(0,0){\includegraphics[width=\unitlength,page=1]{moduli_space_tube_0.pdf}}%
    \put(0.04397846,0.82926369){\makebox(0,0)[lt]{\lineheight{1.25}\smash{\begin{tabular}[t]{l}\textbf{$J$}\end{tabular}}}}%
    \put(0.95446131,0.72116233){\makebox(0,0)[lt]{\lineheight{1.25}\smash{\begin{tabular}[t]{l}\textbf{$H$}\end{tabular}}}}%
    \put(0.70921694,0.47071484){\makebox(0,0)[lt]{\lineheight{1.25}\smash{\begin{tabular}[t]{l}\textbf{$\Xi_a^+$}\end{tabular}}}}%
    \put(0.33643784,0.08506682){\makebox(0,0)[lt]{\lineheight{1.25}\smash{\begin{tabular}[t]{l}\textbf{$\Xi_a^-$}\end{tabular}}}}%
    \put(0,0){\includegraphics[width=\unitlength,page=2]{moduli_space_tube_0.pdf}}%
    \put(0.56391215,0.05181275){\makebox(0,0)[lt]{\lineheight{1.25}\smash{\begin{tabular}[t]{l}\textbf{$\Xi_a^0$}\end{tabular}}}}%
    \put(0,0){\includegraphics[width=\unitlength,page=3]{moduli_space_tube_0.pdf}}%
    \put(0.46137616,0.71792769){\makebox(0,0)[lt]{\lineheight{1.25}\smash{\begin{tabular}[t]{l}\textbf{$\Hex$}\end{tabular}}}}%
  \end{picture}%
\endgroup%

%% file: tubes_symmetry.pdf_tex
\begingroup%
  \makeatletter%
  \providecommand\color[2][]{%
    \errmessage{(Inkscape) Color is used for the text in Inkscape, but the package 'color.sty' is not loaded}%
    \renewcommand\color[2][]{}%
  }%
  \providecommand\transparent[1]{%
    \errmessage{(Inkscape) Transparency is used (non-zero) for the text in Inkscape, but the package 'transparent.sty' is not loaded}%
    \renewcommand\transparent[1]{}%
  }%
  \providecommand\rotatebox[2]{#2}%
  \newcommand*\fsize{\dimexpr\f@size pt\relax}%
  \newcommand*\lineheight[1]{\fontsize{\fsize}{#1\fsize}\selectfont}%
  \ifx\svgwidth\undefined%
    \setlength{\unitlength}{766.04495809bp}%
    \ifx\svgscale\undefined%
      \relax%
    \else%
      \setlength{\unitlength}{\unitlength * \real{\svgscale}}%
    \fi%
  \else%
    \setlength{\unitlength}{\svgwidth}%
  \fi%
  \global\let\svgwidth\undefined%
  \global\let\svgscale\undefined%
  \makeatother%
  \begin{picture}(1,0.39390057)%
    \lineheight{1}%
    \setlength\tabcolsep{0pt}%
    \put(0,0){\includegraphics[width=\unitlength,page=1]{tubes_symmetry.pdf}}%
    \put(0.34418669,0.15320217){\makebox(0,0)[lt]{\lineheight{1.25}\smash{\begin{tabular}[t]{l}\textbf{$r$}\end{tabular}}}}%
    \put(0.03111709,0.35348774){\makebox(0,0)[lt]{\lineheight{1.25}\smash{\begin{tabular}[t]{l}\textbf{$h$}\end{tabular}}}}%
    \put(0,0){\includegraphics[width=\unitlength,page=2]{tubes_symmetry.pdf}}%
    \put(0.968294,0.15320217){\makebox(0,0)[lt]{\lineheight{1.25}\smash{\begin{tabular}[t]{l}\textbf{$r$}\end{tabular}}}}%
    \put(0.65522433,0.35348774){\makebox(0,0)[lt]{\lineheight{1.25}\smash{\begin{tabular}[t]{l}\textbf{$h$}\end{tabular}}}}%
    \put(0,0){\includegraphics[width=\unitlength,page=3]{tubes_symmetry.pdf}}%
    \put(0.24530535,0.0022674){\makebox(0,0)[lt]{\lineheight{1.25}\smash{\begin{tabular}[t]{l}\textbf{$r_+$}\end{tabular}}}}%
    \put(0.03222159,0.00227649){\makebox(0,0)[lt]{\lineheight{1.25}\smash{\begin{tabular}[t]{l}\textbf{$r_-$}\end{tabular}}}}%
    \put(0.66452239,0.00227645){\makebox(0,0)[lt]{\lineheight{1.25}\smash{\begin{tabular}[t]{l}\textbf{$r_-$}\end{tabular}}}}%
    \put(0.15133648,0.00229465){\makebox(0,0)[lt]{\lineheight{1.25}\smash{\begin{tabular}[t]{l}\textbf{$\overline{r}$}\end{tabular}}}}%
    \put(0.77386966,0.00229465){\makebox(0,0)[lt]{\lineheight{1.25}\smash{\begin{tabular}[t]{l}\textbf{$\overline{r}$}\end{tabular}}}}%
    \put(0.85886043,0.00226744){\makebox(0,0)[lt]{\lineheight{1.25}\smash{\begin{tabular}[t]{l}\textbf{$r_+$}\end{tabular}}}}%
  \end{picture}%
\endgroup%

%% file: foliation_surface_proof.pdf_tex
\begingroup%
  \makeatletter%
  \providecommand\color[2][]{%
    \errmessage{(Inkscape) Color is used for the text in Inkscape, but the package 'color.sty' is not loaded}%
    \renewcommand\color[2][]{}%
  }%
  \providecommand\transparent[1]{%
    \errmessage{(Inkscape) Transparency is used (non-zero) for the text in Inkscape, but the package 'transparent.sty' is not loaded}%
    \renewcommand\transparent[1]{}%
  }%
  \providecommand\rotatebox[2]{#2}%
  \newcommand*\fsize{\dimexpr\f@size pt\relax}%
  \newcommand*\lineheight[1]{\fontsize{\fsize}{#1\fsize}\selectfont}%
  \ifx\svgwidth\undefined%
    \setlength{\unitlength}{810.42095947bp}%
    \ifx\svgscale\undefined%
      \relax%
    \else%
      \setlength{\unitlength}{\unitlength * \real{\svgscale}}%
    \fi%
  \else%
    \setlength{\unitlength}{\svgwidth}%
  \fi%
  \global\let\svgwidth\undefined%
  \global\let\svgscale\undefined%
  \makeatother%
  \begin{picture}(1,0.80658245)%
    \lineheight{1}%
    \setlength\tabcolsep{0pt}%
    \put(0,0){\includegraphics[width=\unitlength,page=1]{foliation_surface_proof.pdf}}%
    \put(0.85488597,0.4397432){\makebox(0,0)[lt]{\lineheight{1.25}\smash{\begin{tabular}[t]{l}\textbf{$2\pi a$}\end{tabular}}}}%
    \put(0.8729875,0.19291856){\makebox(0,0)[lt]{\lineheight{1.25}\smash{\begin{tabular}[t]{l}\textbf{$\hmax(0)=\dfrac{\pi}{2}\,a$}\end{tabular}}}}%
    \put(0,0){\includegraphics[width=\unitlength,page=2]{foliation_surface_proof.pdf}}%
  \end{picture}%
\endgroup%

%% file: isoperimetric_profile_tau0.2.pdf_tex
\begingroup%
  \makeatletter%
  \providecommand\color[2][]{%
    \errmessage{(Inkscape) Color is used for the text in Inkscape, but the package 'color.sty' is not loaded}%
    \renewcommand\color[2][]{}%
  }%
  \providecommand\transparent[1]{%
    \errmessage{(Inkscape) Transparency is used (non-zero) for the text in Inkscape, but the package 'transparent.sty' is not loaded}%
    \renewcommand\transparent[1]{}%
  }%
  \providecommand\rotatebox[2]{#2}%
  \newcommand*\fsize{\dimexpr\f@size pt\relax}%
  \newcommand*\lineheight[1]{\fontsize{\fsize}{#1\fsize}\selectfont}%
  \ifx\svgwidth\undefined%
    \setlength{\unitlength}{357.92187502bp}%
    \ifx\svgscale\undefined%
      \relax%
    \else%
      \setlength{\unitlength}{\unitlength * \real{\svgscale}}%
    \fi%
  \else%
    \setlength{\unitlength}{\svgwidth}%
  \fi%
  \global\let\svgwidth\undefined%
  \global\let\svgscale\undefined%
  \makeatother%
  \begin{picture}(1,0.65447243)%
    \lineheight{1}%
    \setlength\tabcolsep{0pt}%
    \put(0,0){\includegraphics[width=\unitlength,page=1]{isoperimetric_profile_tau0.2.pdf}}%
    \put(0.11431279,0.58851876){\makebox(0,0)[lt]{\lineheight{1.25}\smash{\begin{tabular}[t]{l}\textbf{$\tau^2=1/25$}\end{tabular}}}}%
  \end{picture}%
\endgroup%

%% file: isoperimetric_profile_tau0.25.pdf_tex
\begingroup%
  \makeatletter%
  \providecommand\color[2][]{%
    \errmessage{(Inkscape) Color is used for the text in Inkscape, but the package 'color.sty' is not loaded}%
    \renewcommand\color[2][]{}%
  }%
  \providecommand\transparent[1]{%
    \errmessage{(Inkscape) Transparency is used (non-zero) for the text in Inkscape, but the package 'transparent.sty' is not loaded}%
    \renewcommand\transparent[1]{}%
  }%
  \providecommand\rotatebox[2]{#2}%
  \newcommand*\fsize{\dimexpr\f@size pt\relax}%
  \newcommand*\lineheight[1]{\fontsize{\fsize}{#1\fsize}\selectfont}%
  \ifx\svgwidth\undefined%
    \setlength{\unitlength}{357.92188074bp}%
    \ifx\svgscale\undefined%
      \relax%
    \else%
      \setlength{\unitlength}{\unitlength * \real{\svgscale}}%
    \fi%
  \else%
    \setlength{\unitlength}{\svgwidth}%
  \fi%
  \global\let\svgwidth\undefined%
  \global\let\svgscale\undefined%
  \makeatother%
  \begin{picture}(1,0.65447243)%
    \lineheight{1}%
    \setlength\tabcolsep{0pt}%
    \put(0,0){\includegraphics[width=\unitlength,page=1]{isoperimetric_profile_tau0.25.pdf}}%
    \put(0.11391083,0.59719231){\makebox(0,0)[lt]{\lineheight{1.25}\smash{\begin{tabular}[t]{l}\textbf{$\tau^2=1/16$}\end{tabular}}}}%
  \end{picture}%
\endgroup%

%% file: isoperimetric_profile_tau0.32.pdf_tex
\begingroup%
  \makeatletter%
  \providecommand\color[2][]{%
    \errmessage{(Inkscape) Color is used for the text in Inkscape, but the package 'color.sty' is not loaded}%
    \renewcommand\color[2][]{}%
  }%
  \providecommand\transparent[1]{%
    \errmessage{(Inkscape) Transparency is used (non-zero) for the text in Inkscape, but the package 'transparent.sty' is not loaded}%
    \renewcommand\transparent[1]{}%
  }%
  \providecommand\rotatebox[2]{#2}%
  \newcommand*\fsize{\dimexpr\f@size pt\relax}%
  \newcommand*\lineheight[1]{\fontsize{\fsize}{#1\fsize}\selectfont}%
  \ifx\svgwidth\undefined%
    \setlength{\unitlength}{357.92188074bp}%
    \ifx\svgscale\undefined%
      \relax%
    \else%
      \setlength{\unitlength}{\unitlength * \real{\svgscale}}%
    \fi%
  \else%
    \setlength{\unitlength}{\svgwidth}%
  \fi%
  \global\let\svgwidth\undefined%
  \global\let\svgscale\undefined%
  \makeatother%
  \begin{picture}(1,0.65447243)%
    \lineheight{1}%
    \setlength\tabcolsep{0pt}%
    \put(0,0){\includegraphics[width=\unitlength,page=1]{isoperimetric_profile_tau0.32.pdf}}%
    \put(0.11474405,0.59703108){\makebox(0,0)[lt]{\lineheight{1.25}\smash{\begin{tabular}[t]{l}\textbf{$\tau^2=1/10$}\end{tabular}}}}%
  \end{picture}%
\endgroup%

%% file: isoperimetric_profile_tau0.45.pdf_tex
\begingroup%
  \makeatletter%
  \providecommand\color[2][]{%
    \errmessage{(Inkscape) Color is used for the text in Inkscape, but the package 'color.sty' is not loaded}%
    \renewcommand\color[2][]{}%
  }%
  \providecommand\transparent[1]{%
    \errmessage{(Inkscape) Transparency is used (non-zero) for the text in Inkscape, but the package 'transparent.sty' is not loaded}%
    \renewcommand\transparent[1]{}%
  }%
  \providecommand\rotatebox[2]{#2}%
  \newcommand*\fsize{\dimexpr\f@size pt\relax}%
  \newcommand*\lineheight[1]{\fontsize{\fsize}{#1\fsize}\selectfont}%
  \ifx\svgwidth\undefined%
    \setlength{\unitlength}{357.92188074bp}%
    \ifx\svgscale\undefined%
      \relax%
    \else%
      \setlength{\unitlength}{\unitlength * \real{\svgscale}}%
    \fi%
  \else%
    \setlength{\unitlength}{\svgwidth}%
  \fi%
  \global\let\svgwidth\undefined%
  \global\let\svgscale\undefined%
  \makeatother%
  \begin{picture}(1,0.65447243)%
    \lineheight{1}%
    \setlength\tabcolsep{0pt}%
    \put(0,0){\includegraphics[width=\unitlength,page=1]{isoperimetric_profile_tau0.45.pdf}}%
    \put(0.11445621,0.59703849){\makebox(0,0)[lt]{\lineheight{1.25}\smash{\begin{tabular}[t]{l}\textbf{$\tau^2=1/5$}\end{tabular}}}}%
  \end{picture}%
\endgroup%

%% file: isoperimetric_profile_tau0.71.pdf_tex
\begingroup%
  \makeatletter%
  \providecommand\color[2][]{%
    \errmessage{(Inkscape) Color is used for the text in Inkscape, but the package 'color.sty' is not loaded}%
    \renewcommand\color[2][]{}%
  }%
  \providecommand\transparent[1]{%
    \errmessage{(Inkscape) Transparency is used (non-zero) for the text in Inkscape, but the package 'transparent.sty' is not loaded}%
    \renewcommand\transparent[1]{}%
  }%
  \providecommand\rotatebox[2]{#2}%
  \newcommand*\fsize{\dimexpr\f@size pt\relax}%
  \newcommand*\lineheight[1]{\fontsize{\fsize}{#1\fsize}\selectfont}%
  \ifx\svgwidth\undefined%
    \setlength{\unitlength}{357.92188074bp}%
    \ifx\svgscale\undefined%
      \relax%
    \else%
      \setlength{\unitlength}{\unitlength * \real{\svgscale}}%
    \fi%
  \else%
    \setlength{\unitlength}{\svgwidth}%
  \fi%
  \global\let\svgwidth\undefined%
  \global\let\svgscale\undefined%
  \makeatother%
  \begin{picture}(1,0.65447243)%
    \lineheight{1}%
    \setlength\tabcolsep{0pt}%
    \put(0,0){\includegraphics[width=\unitlength,page=1]{isoperimetric_profile_tau0.71.pdf}}%
    \put(0.1144544,0.59713577){\makebox(0,0)[lt]{\lineheight{1.25}\smash{\begin{tabular}[t]{l}\textbf{$\tau^2=1/2$}\end{tabular}}}}%
  \end{picture}%
\endgroup%

%% file: isoperimetric_profile_tau1.41.pdf_tex
\begingroup%
  \makeatletter%
  \providecommand\color[2][]{%
    \errmessage{(Inkscape) Color is used for the text in Inkscape, but the package 'color.sty' is not loaded}%
    \renewcommand\color[2][]{}%
  }%
  \providecommand\transparent[1]{%
    \errmessage{(Inkscape) Transparency is used (non-zero) for the text in Inkscape, but the package 'transparent.sty' is not loaded}%
    \renewcommand\transparent[1]{}%
  }%
  \providecommand\rotatebox[2]{#2}%
  \newcommand*\fsize{\dimexpr\f@size pt\relax}%
  \newcommand*\lineheight[1]{\fontsize{\fsize}{#1\fsize}\selectfont}%
  \ifx\svgwidth\undefined%
    \setlength{\unitlength}{357.92188074bp}%
    \ifx\svgscale\undefined%
      \relax%
    \else%
      \setlength{\unitlength}{\unitlength * \real{\svgscale}}%
    \fi%
  \else%
    \setlength{\unitlength}{\svgwidth}%
  \fi%
  \global\let\svgwidth\undefined%
  \global\let\svgscale\undefined%
  \makeatother%
  \begin{picture}(1,0.65447243)%
    \lineheight{1}%
    \setlength\tabcolsep{0pt}%
    \put(0,0){\includegraphics[width=\unitlength,page=1]{isoperimetric_profile_tau1.41.pdf}}%
    \put(0.11450893,0.59709842){\makebox(0,0)[lt]{\lineheight{1.25}\smash{\begin{tabular}[t]{l}\textbf{$\tau^2=2$}\end{tabular}}}}%
  \end{picture}%
\endgroup%

%% file: Literatur.bib
@article{andrews_li,
        author = {Andrews, Ben and Li, Haizhong},
        title = {Embedded constant mean curvature tori in the three-sphere},
        journaltitle = {J. Differential Geom.},
        year = {2015},
        volume = {99},
        issue = {2},
        pages = {169-189},
        doi = {10.4310/jdg/1421415560},
}

@article{brendle,
        author = {Brendle, Simon},
        title = {Embedded minimal tori in $\s^3$ and the Lawson conjecture},
        journaltitle = {Acta Math.},
        year = {2013},
        volume = {211},
        issue = {2},
        pages = {177-190},
        doi = {10.1007/s11511-013-0101-2},
}

@article{burghelea_verona,
	author = {Burghelea, Dan and Verona, Andrei},
	title = {Local homological properties of analytic sets},
	journaltitle = {Manus. Math.},
	year = {1972},
	volume = {7},
	issue = {},
	pages = {55-62},
	doi = {10.1007/BF01303536},
}

@article{figueroa_mercuri_pedrosa,
	author = {Figueroa, Christiam B. and Mercuri, Francesco and Pedrosa, Renato H. L.},
	title = {Invariant surfaces of the Heisenberg groups},
	journaltitle = {Ann. Mat. Pura Appl.},
	year = {1999},
	volume = {177},
	issue = {},
	pages = {173-194},
	doi = {10.1007/BF02505908},
}

@article{fu_mccrory,
	author = {Fu, Joseph H. G. and McCrory, Clint},
	title = {Stiefel-Whitney Classes and the Conormal Cycle of a Singular Variety},
	journaltitle = {Transactions of the American Mathematical Society},
	year = {1997},
	volume = {349},
	issue = {2},
	pages = {809-835},
	doi = {},
}

@article{hardt,
	author = {Hardt, Robert M.},
	title = {Sullivan's local Euler characteristic theorem},
	journaltitle = {Manus. Math.},
	year = {1974},
	volume = {12},
	issue = {},
	pages = {87-92},
	doi = {10.1007/BF01166236},
}

@article{kaese,
	author = {Käse, Philipp},
	title = {Screw motion surfaces of constant mean curvature in homogeneous 3-manifolds},
	journaltitle = {J. Math. Anal. Appl.},
	year = {2024},
	volume = {538},
	issue = {1},
	pages = {128420},
	doi = {10.1016/j.jmaa.2024.128420},
}

@article{korevaar_kusner_meeks_solomon,
	author = {Korevaar, Nicholas J. and Kusner, Rob and Meeks III, William H. and Solomon, Bruce},
	title = {Constant Mean Curvature Surfaces in Hyperbolic Space},
	journaltitle = {Amer. J. Math.},
	year = {1992},
	volume = {114},
	issue = {1},
	pages = {1-43},
	doi = {10.2307/2374738},
}

@article{korevaar_kusner_solomon,
	author = {Korevaar, Nicholas J. and Kusner, Rob and Solomon, Bruce},
	title = {The structure of complete embedded surfaces with constant mean curvature},
	journaltitle = {J. Diff. Geom.},
	year = {1989},
	volume = {30},
	issue = {2},
	pages = {465-503},
	doi = {10.4310/jdg/1214443598},
}

@article{lawson,
	author = {Lawson, H. Blaine},
	title = {Complete Minimal Surfaces in $\s^3$},
	journaltitle = {Ann. of Math. (2)},
	year = {1970},
	volume = {921},
	issue = {3},
	pages = {335-374},
	doi = {10.2307/1970625},
}

@article{lopez,
	author = {L\'{o}pez, Rafael},
	title = {Invariant surfaces in $\Soldrei$ with constant mean curvature and their computer graphics},
	journaltitle = {Advances in Geometry},
	year = {2014},
	volume = {14},
	issue = {},
	pages = {31-48},
	doi = {10.1515 / advgeom-2013-0015},
}

@article{manzano24,
        author = {Manzano, Jos\'{e} M.},
        title = {Invariant constant mean curvature tubes around a horizontal geodesic in $\mathbb{E}(\kappa,\tau)$-spaces},
        journaltitle = {J. Math. Anal. Appl.},
        year = {2024},
        volume = {531},
        issue = {1, Part 2},
        pages = {127878},
        doi = {10.1016/j.jmaa.2023.127878},
}

@article{manzano_torralbo,
	author = {Manzano, Jos\'{e} M. and Torralbo, Francisco},
	title = {Horizontal Delaunay surfaces with constant mean curvature in $\mathbb{S}^2\times\mathbb{R}$ and $\mathbb{H}^2\times\mathbb{R}$},
	journaltitle = {Camb. J. Math.},
	year = {2022},
	volume = {10},
	issue = {3},
	pages = {657-688},
	doi = {10.4310/cjm.2022.v10.n3.a2},
}

@article{mazet15,
        author = {Mazet, Laurent},
        title = {Cylindrically bounded constant mean curvature surfaces in $\HR$},
        journaltitle = {Amer. J. Math.},
        year = {2015},
        volume = {367},
        issue = {8},
        pages = {5329–5354},
        doi = {10.1090/tran/6171},
}

@article{meeks,
        author = {Meeks III, William H.},
        title = {The topology and geometry of embedded surfaces of constant mean curvature},
        journaltitle = {J. Differential Geom.},
        year = {1988},
        volume = {27},
        issue = {3},
        pages = {539-552},
        doi = {10.4310/jdg/1214442008},
}

@article{montaldo_onnis,
	author = {Montaldo, Stefano and Onnis, Irene I.},
	title = {Invariant CMC surfaces in $\mathbb{H}^2\times\mathbb{R}$},
	journaltitle = {Glasgow Math. J.},
	year = {2004},
	volume = {46},
	issue = {},
	pages = {311-321},
	doi = {10.1017/S001708950400179X},
}

@article{onnis08,
        author = {Onnis, Irene I.},
        title = {Invariant surfaces with constant mean curvature in $\mathbb{H}^2\times\mathbb{R}$},
        journaltitle = {Ann. Mat. Pura Appl.},
        year = {2008},
        volume = {187},
        issue = {},
        pages = {667–682},
        doi = {10.1007/s10231-007-0061-2},
}

@article{pedrosa,
	author = {Pedrosa, Renato H. L.},
	title = {The Isoperimetric Problem in Spherical Cylinders},
	journaltitle = {Ann. of Global Analysis and Geometry},
	year = {2004},
	volume = {26},
	issue = {},
	pages = {333–354},
	doi = {10.1023/B:AGAG.0000047528.20962.e2},
}

@article{pedrosa_ritore,
	author = {Pedrosa, Renato H. L. and Ritor\'{e}, Manuel},
	title = {Isoperimetric domains in the Riemannian product of a circle with a simply connected space form and applications to free boundary problems},
	journaltitle = {Indiana Univ. Math. J.},
	year = {1999},
	volume = {48},
	issue = {4},
	pages = {1357-1394},
	doi = {10.1512/iumj.1999.48.1614},
}

@article{penafiel12,
	author = {Pe\~{n}afiel, Carlos},
	title = {Invariant surfaces in $\widetilde{\operatorname{PSL}}_2(\mathbb{R},\tau)$ and applications},
	journaltitle = {Bull. Braz. Math. Soc.},
	year = {2012},
	volume = {43},
	issue = {4},
	pages = {545-578},
	doi = {10.1007/s00574-012-0026-y},
}

@article{penafiel15,
	author = {Pe\~{n}afiel, Carlos},
	title = {Screw motion surfaces in $\widetilde{\operatorname{PSL}}_2(\mathbb{R},\tau)$},
	journaltitle = {Asian J. Math.},
	year = {2015},
	volume = {19},
	issue = {2},
	pages = {265-280},
	doi = {10.4310/AJM.2015.v19.n2.a4},
}

@article{saearp,
	author = {Sa Earp, Ricardo},
	title = {Parabolic and hyperbolic screw motion surfaces in $\mathbb{H}^2\times\mathbb{R}$},
	journaltitle = {J. Aust. Math. Soc.},
	year = {2008},
	volume = {85},
	issue = {},
	pages = {113–143},
	doi = {10.1017/S1446788708000013},
}

@article{shin_kim_koh_lee_yang,
	author = {Shin, Heayong and Kim, Young Wook and Koh, Sung-Eun and Lee, Hyung Yong and Yang, Seong-Deog},
	title = {Ruled minimal surfaces in the Berger sphere},
	journaltitle = {Diff. Geom. Appl.},
	year = {2015},
	volume = {40},
	issue = {},
	pages = {209-222},
	doi = {10.1016/j.difgeo.2015.02.007},
}

@article{torralbo10,
	author = {Torralbo, Francisco},
	title = {Rotationally invariant constant mean curvature surfaces in homogeneous 3-manifolds},
	journaltitle = {Differential Geom. Appl.},
	year = {2010},
	volume = {28},
	issue = {},
	pages = {593-607},
	doi = {10.1016/j.difgeo.2010.04.007},
}

@article{torralbo12,
	author = {Torralbo, Francisco},
	title = {Compact minimal surfaces in the Berger spheres},
	journaltitle = {Ann. Global Anal. Geom.},
	year = {2012},
	volume = {41},
	issue = {},
	pages = {391-405},
	doi = {10.1007/s10455-011-9288-7},
}

@article{torralbo_urbano,
	author = {Torralbo, Francisco and Urbano, Francisco},
	title = {Compact stable constant mean curvature surfaces in homogeneous 3-manifolds},
	journaltitle = {Indiana Univ. Math. J.},
	year = {2012},
	volume = {61},
	issue = {3},
	pages = {1129–1156},
	doi = {10.1512/iumj.2012.61.4667},
}

@article{vrzina18,
	author = {Vr\v{z}ina, Miroslav},
	title = {Cylinders as left invariant CMC surfaces in $\Soldrei$ and $\EKT$-spaces diffeomorphic to $\R^3$},
	journaltitle = {Differential Geom. Appl.},
	year = {2018},
	volume = {58},
	issue = {},
	pages = {141-176},
	doi = {10.1016/j.difgeo.2018.01.005},
}

@book{apostol,
	author = {Apostol, Tom M.},
	title = {Mathematical Analysis},
	year = {1974},
	edition = {2},
	publisher = {Addison-Wesley},
	location = {Reading, MA},
	doi = {},
}

@book{krantz_parks,
	author = {Krantz, Steven G. and Parks, Harold R.},
	title = {A Primer of Real Analytic Functions},
	year = {1981},
	edition = {2},
	publisher = {Birkhäuser},
	location = {Boston, MA},
	doi = {10.1007/978-0-8176-8134-0},
}

@partofbook{sullivan,
	author = {Sullivan, Dennis},
	title = {Combinatorial invariants of analytic spaces},
	booktitle = {Liverpool Singularities Symposium, Lecture Notes in Math.},
	volume = {192},
	year = {1971},
	edition = {},
	publisher = {Springer},
	location = {Berlin, Heidelberg},
	pages = {165-169},
	doi = {10.1007/BFb0066822},
}

@online{vrzina17,
	author = {Vr\v{z}ina, Miroslav},
	title = {On the existence problem for tilted unduloids in $\mathbb{H}^2\times\mathbb{R}$},
	year = {2017},
	note = {preprint},
	eprinttype = {arxiv},
	eprint = {1702.02761},
}

@thesis{kaese_phd,
	author = {Käse, Philipp},
	title = {Minimal surfaces and surfaces of constant mean curvature in homogeneous spaces},
	year = {2024},
	note = {PhD thesis, Technische Universit{\"a}t Darmstadt},
	url = {}
}

@thesis{vrzina16,
	author = {Vr\v{z}ina, Miroslav},
	title = {Constant Mean Curvature Annuli in Homogeneous Manifolds},
	year = {2016},
	note = {PhD thesis, Technische Universit{\"a}t Darmstadt},
	url = {http://tuprints.ulb.tu-darmstadt.de/5444/}
}
